\documentclass[amssymb,11pt]{amsart}
\usepackage[left=4cm,right=4cm,top=4cm,bottom=4cm]{geometry}

\usepackage{amssymb,amsfonts}
\usepackage[all,arc]{xy}
\usepackage{enumerate}
\usepackage{mathrsfs}
\usepackage{mathtools}
\usepackage{calligra}
\usepackage[hidelinks]{hyperref}
\usepackage{bbold}
\usepackage{enumitem}
\usepackage[dvipsnames]{xcolor}
\usepackage{tabularray}
\usepackage{float}
\UseTblrLibrary{amsmath}
\usepackage{fancyhdr}
\usepackage{stmaryrd}

\newtheorem{thm}{Theorem}[section]
\newtheorem{cor}[thm]{Corollary}
\newtheorem{prop}[thm]{Proposition}
\newtheorem{lem}[thm]{Lemma}

\newtheorem{defn}[thm]{Definition}

\theoremstyle{definition}

\newtheorem{exmp}[thm]{Example}

\newtheorem{notns}[thm]{Notations}

\newtheorem{assp}{Assumption}

\theoremstyle{remark}
\newtheorem{rem}[thm]{Remark}

\numberwithin{equation}{section}

\newcommand{\A}[1]{\alpha_{#1}}
\newcommand{\Ab}[1]{\alpha_{#1}^{-1/2}}
\newcommand{\C}{\mathbb{C}}
\newcommand{\R}{\mathbb{R}}
\newcommand{\N}{\mathbb{N}}
\newcommand{\Z}{\mathbb{Z}}

\newcommand{\T}{\mathbb{T}}
\newcommand{\F}[1]{\mc{F}_{\Tilde{h}}\pr{#1}}
\newcommand{\Prob}[1]{\mathbb{P}\pr{#1}}
\newcommand{\Rd}{\mathbb{R}^{2d}}

\newcommand{\m}{\mbox{ }}
\newcommand{\mm}{\m\m\m\m}

\newcommand{\mc}[1]{\mathcal{#1}}
\newcommand{\norme}[1]{\left\Vert #1\right\Vert}

\newcommand{\atc}[1]{\textit{\v{#1}}}
\newcommand{\ps}[2]{\left\langle #1,#2 \right\rangle}
\newcommand{\jap}[1]{\left\langle #1 \right\rangle}
\newcommand{\disc}[2]{\left\llbracket #1,#2\right\rrbracket}
\newcommand{\bigslant}[2]{{\raisebox{.2em}{$#1$}\left/\raisebox{-.2em}{$#2$}\right.}}
\newcommand{\Hd}{\mc{H}^d_{\Tilde{h},\alpha_1,\alpha_2}}

\newcommand{\pr}[1]{\!\left( #1 \right)}
\newcommand{\croch}[1]{\left[ #1 \right]}

\newcommand{\bb}[1]{\left| #1 \right|}
\newcommand{\acc}[1]{\left\{ #1 \right\}}

\newcommand{\Moy}{\sharp_{\Tilde{h}}}
\newcommand{\Schw}[1]{\mathscr{S}(\R^{#1})}
\newcommand{\Schwp}[1]{\mathscr{S}'(\R^{#1})}

\newcommand{\nocontentsline}[3]{}
\newcommand{\tocless}[2]{\bgroup\let\addcontentsline=\nocontentsline#1{#2}\egroup}

\definecolor{vert}{RGB}{30,206,32}
\definecolor{Violet}{RGB}{128,41,226}
\definecolor{bleu}{RGB}{28,153,212}
\definecolor{turquoise}{RGB}{38,242,208}
\definecolor{Orange}{RGB}{250,125,28}
\definecolor{fushia}{RGB}{220,28,250}
\definecolor{Bleu}{RGB}{4,42,171}
\definecolor{rouge}{RGB}{203,6,6}
\definecolor{nous}{RGB}{202,214,7}
\definecolor{marron}{RGB}{118,64,3}
\definecolor{pustule}{RGB}{150,22,218}
\definecolor{pequinois}{RGB}{42,209,235}
\definecolor{manteuse}{RGB}{232,186,17}
\definecolor{rausepal}{RGB}{226,133,194}

\usepackage[backend=biber,style=numeric,sorting=nyt]{biblatex}
\let\bibliography\addbibresource
\AtBeginBibliography{\footnotesize}

\title{Probabilistic Weyl law for twisted Toeplitz matrices with rough symbols}
\author{Lucas Noël}

\begin{document}
\maketitle 
\begin{abstract}
In this article, we study the convergence of the empirical spectral measure of twisted Toeplitz matrices subject to small random perturbations. We show that the empirical spectral measure converges weakly in probability to the push-forward of the Lebesgue measure by the symbol. The symbol of the twisted Toeplitz matrices is assumed to be smooth in frequency, and only piecewise Hölder continuous with respect to the position variable with discontinuities of jump type.
\end{abstract}
\tableofcontents

\section{Introduction and statement of the main result}
\mm We\label{Introduction} are interested in the following construction. Take a function $p:[0,1]^d_x\times \T^d_\xi\to \C$, which is smooth with respect to $\xi$ in $\T^d:=(\R/\Z)^d$. Using the partial Fourier transform with respect to $\xi$, we can write $p$ as
\begin{equation*}
    p(x,\xi):=\sum_{\nu\in \Z^d}p_\nu(x)e^{2i\pi\ps{\nu}{\xi}}
\end{equation*}
where 
\begin{equation*}
    p_\nu(x)=\int_{\T^d}p(x,\xi)e^{-2i\pi\ps{\nu}{\xi}}d\xi
\end{equation*}
for all $\nu\in \Z^d$. Let $0< h\leq 1$ be the semiclassical parameter. Then, we associate with $p$ the linear operator
\begin{equation}
    \mathrm{Op}_h(p):\left\{\begin{array}{ccccc}
         & \mathrm{PS} & \longrightarrow & \mathrm{PS} \\
         & c:=(c_n)_{n\in \Z^d} & \mapsto & \Bigl(\croch{\widehat{p}(2\pi h m)\ast c}(m)\Bigl)_{m\in \Z^d}
    \end{array}\right.
\label{Expression nouvelle procédure de quantif}
\end{equation}
where for $x\in \R^d$, $\widehat{p}(x)\in \C^{\Z^d}$ is such that for all $\nu\in \Z^d$,
\begin{equation*}
    [\widehat{p}(x)]_\nu:=\left\{\begin{array}{ccc}
         p_\nu(x) & \text{if} & x\in [0,1]^d \\
          0 & \text{otherwise} &
    \end{array}\right.
\end{equation*}
and $\mathrm{PS}$ refers to the set of elements of $\C^{\Z^d}$ that have at most polynomial growth, see \eqref{Suites croissance poly}. This quantization procedure \eqref{Expression nouvelle procédure de quantif} is explained in more details in Section \ref{Explanation of the quantization procedure}. \\

Let
\begin{equation}\label{Nombre de points dans l'intersection de [0,1] et 2pihZ}
    N:=\#\Bigl([0,1]\cap 2\pi h\Z\Bigl).
\end{equation}
Let us define the operator 
\begin{equation}\label{Définition de l'opérateur pour la comparaison avec BPZ}
    M_N(p):=\mathbb{1}_{\disc{1}{N}^d}\mathrm{Op}_h(p)\mathbb{1}_{\disc{1}{N}^d}
\end{equation}
which can be identified with an operator $\C^{N^d}\to \C^{N^d}$ and thus with its matrix in the canonical basis. \\
\mm For example, let us take $d=1$, and the function $p:[0,1]_x\times \T_\xi\to \C$ of the form
\begin{equation}\label{Forme symbole dans les travaux de BPZ}
    p(x,\xi)=\sum_{k=-N_-}^{N_+}p_k(x)e^{2i\pi k\xi}
\end{equation}
where $N_+,N_-\geq 0$ and $p_k:[0,1]\to \C$, $k=-N_-,\dots,N_+$ is a measurable function. Then for $N>\max(N_+,N_-)$, the previous construction leads to the matrix $M_N(p)$ given by\\
\begin{equation}\label{Matrice T_N(p) BPZ}
    \begin{pmatrix}
                p_0(x_1)&p_{1}(x_1)&\dots&p_{N_+}(x_1)&0&\dots&0\\
                p_{-1}(x_2)&p_0(x_2)& \ddots&\ddots &\ddots&\ddots &\vdots\\
                \vdots &\ddots&\ddots&\ddots&\ddots&\ddots&0\\
                p_{-N_-}(x_{N_-+1})&\ddots&\ddots&\ddots&\ddots&\ddots &p_{N_+}(x_{N-N_+})\\
                0&\ddots &\ddots&\ddots&p_{0}(x_{N-2})&\ddots&\vdots\\
                \vdots &\ddots&\ddots &\ddots &\ddots& p_0(x_{N-1})&p_1(x_{N-1})\\
                0&\dots &0&p_{-N_-}(x_N)&\dots &p_{-1}(x_N) &p_0(x_N)
            \end{pmatrix}
\end{equation}
where $x_j=j/N$, $j=1,\dots,N$.\\

\m\m\m\m In this paper, we consider the symbol class given by the functions $p:[0,1]^d_x\times\T^d_\xi\to \C$ that are smooth with respect to $\xi$ and piecewise $\varrho$-Hölder continuous with respect to $x$ for some $\varrho>0$. This symbol class will be denoted $\mathscr{C}^{0,\varrho}_{\mathrm{pw}}S_d$. More precisely:

\begin{defn}\label{Définition espace d'intérêt}
    Let $\varrho\in \hspace{0.1cm} ]0,1]$. We define $\mathscr{C}^{0,\varrho}_{\mathrm{pw}}S_d$ as the set of functions $f:[0,1]^d_x\times \T^d_\xi\to \C$ that are smooth with respect to $\xi$ and such that 
    \begin{enumerate}
        \item there exists a finite family of open connected disjoint sets $U_1,...,U_s$ of $[0,1]^d$ so that
        \begin{equation}\label{Propriété des ouverts dans la def des fonctions de l'espace d'intérêt}
            \overline{\bigsqcup_{j=1}^s U_j}=[0,1]^d,
        \end{equation}
        \item for all $j\in \disc{1}{s}$, there exists  $f_j:\overline{U_j}\times \T^d\to \C$ such that ${f_j}_{|U_j\times \T^d}={f}_{|U_j\times \T^d}$ and such that for all $\beta\in \N^d$, there exists $C_{j,\beta}>0$ so that for all $\xi \in \T^d$,
        \begin{equation}\label{Norme Hölder des dérivées partielles par rapport à xi}
            \norme{\partial_\xi^\beta f_j(\cdot,\xi)}_{\mathscr{C}^{0,\varrho}\pr{\overline{U_j}}}\leq C_{j,\beta}.
        \end{equation}
    \end{enumerate}
    Additionally, for $f\in \mathscr{C}^{0,\varrho}_{\mathrm{pw}}S_d$, we will denote 
    \begin{equation}
        \mathscr{U}_f:=\bigcup_{j=1}^s \partial U_j
    \label{Ensemble des singularités}
    \end{equation}
    the set of potential singularities of $f(\cdot,\xi)$ in $[0,1]^d$.
\end{defn}
\m\\
We will work in the regime $0<h\ll 1$ and therefore $N\gg 1$ via \eqref{Nombre de points dans l'intersection de [0,1] et 2pihZ}. Our main purpose is to obtain a result on the eigenvalue distribution of 
\begin{equation*}
    M_N(p)+\delta Q_N
\end{equation*}
for $p\in \mathscr{C}^{0,\varrho}_{\mathrm{pw}}S_d$, where $\delta$ decays polynomially in $N$ and where $Q_N$ is an $N^d\times N^d$ random matrix. To do so, we will work with the following assumptions on $p$.
\begin{assp}\label{Hypothèse théorème final - volume singularités}
    Let $p\in \mathscr{C}^{0,\varrho}_{\mathrm{pw}}S_d$ and let $\mathscr{U}_p$ be as in \eqref{Ensemble des singularités}. Suppose that there exists ${\kappa_1}\in \hspace{0.1cm} ]0,1]$ such that, when $0<r\ll 1$, 
    \begin{equation*}
        \mathrm{L}\Bigl(\pr{\mathscr{U}_p+B(0,r)}\cap[0,1]^d\Bigl)=\mc{O}(r^{{\kappa_1}}),
    \end{equation*}
    where $\mathrm{L}$ denotes the Lebesgue measure on $[0,1]^d_x$.
\end{assp}
Notice that
\begin{equation}\label{Fait que l'hypothèse sur les singularité soit toujours vraie en dimension 1}
    \text{when $d=1$, Assumption \ref{Hypothèse théorème final - volume singularités} always holds with $\kappa_1=1$.}
\end{equation}

\begin{assp}\label{Hypothèse théorème final - volume préimage}
    Let $p\in \mathscr{C}^{0,\varrho}_{\mathrm{pw}}S_d$. Suppose that there exists ${\kappa_2}\in \hspace{0.1cm} ]0,1]$ such that, uniformly for $z\in \C$, the following holds: for $0\leq t\ll 1$,
    \begin{equation*}
        \mathscr{V}_z(t):=\mathrm{L}\pr{\acc{\rho\in [0,1]_x^d\times \T_\xi^d\m ;\m \bb{p(\rho)-z}^2\leq t}}=\mc{O}(t^{\kappa_2}),
    \end{equation*}
    where $\mathrm{L}$ denotes the normalized Lebesgue measure on $[0,1]^d_x\times \T^d_\xi$.
\end{assp}
Assumptions \ref{Hypothèse théorème final - volume singularités} and \ref{Hypothèse théorème final - volume préimage} are satisfied by a wide range of functions in $\mathscr{C}^{0,\varrho}_{\mathrm{pw}}S_d$, as it is discussed in Section \ref{Comments sur les hypothèses du théorème}. \\

\mm We will also make the following assumption on the random perturbation $Q_N$.
\begin{assp}
    $Q_N$ is a random $N^d\times N^d$ matrix satisfying what follows:
    \begin{enumerate}
        \item \textbf{Norm bound:} There exists $\kappa_3>0$ such that 
        \begin{equation*}
            \mathbb{E}\pr{\norme{Q_N}}=\mc{O}(N^{\kappa_3}).
        \end{equation*}
        \item \textbf{Anti-concentration bound:} For each $\theta>0$, there exists $\beta>0$ such that for every sequence $(A_N)$ of matrices satisfying $\norme{A_N}=\mc{O}(N^\theta)$, where for all $N\geq 1$, $A_N$ is of size $N^d\times N^d$, it holds
        \begin{equation*}
            \Prob{s_{N^d}\pr{M+Q_N}\leq N^{-\beta}}\underset{N\to +\infty}{\longrightarrow} 0.
        \end{equation*}
        Here, for all $N^d\times N^d$ matrix $A_N$, $s_{N^d}(A_N)$ and $\norme{A_N}$ denote the smallest and greatest singular value of $A_N$ respectively.
    \end{enumerate}
\label{Hypothèse théorème final - matrice aléatoire}
\end{assp}
This Assumption regarding $Q_N$ is not restrictive. In Section \ref{Comments sur les hypothèses du théorème}, we provide numerous examples of random matrices that satisfy Assumption \ref{Hypothèse théorème final - matrice aléatoire}.  \\

\m\m\m\m Our main theorem is as follows.
\begin{thm}
    Let $p\in \mathscr{C}^{0,\varrho}_{\mathrm{pw}}S_d$ and $\mathscr{U}_p$ be given by \eqref{Ensemble des singularités}. Assume that $p$ satisfies Assumptions \ref{Hypothèse théorème final - volume singularités} and \ref{Hypothèse théorème final - volume préimage}. Let $Q_N$ be an $N^d\times N^d$ random matrix satisfying Assumption \ref{Hypothèse théorème final - matrice aléatoire} for some $\kappa_3>0$. Then, for every $\delta_0>0$, setting 
    \begin{equation}
        \delta:=N^{-(\kappa_3+\delta_0)},
    \label{Hypothèse expression de delta théorème final}
    \end{equation}
     the empirical spectral measure 
     \begin{equation}
        \mu_{N}:=\frac{1}{N^d}\sum_{\lambda\in \sigma(M_N(p)+\delta Q_N)}\delta_{\lambda}
    \label{Mesure empirique - théorème final}
    \end{equation}
     of $M_N(p)+\delta Q_N$ converges weakly in probability to the push-forward $p_*\mathrm{L}$ of the normalized Lebesgue measure on $[0,1]^d_x\times \T_\xi^{d}$ by $p$.
\label{Théorème final}
\end{thm}

\m\m\m\m The proof of this result also leads to the following
\begin{cor}
    Let $p$, $\delta$ and $Q_N$ be as in Theorem \ref{Théorème final}. Let $R_N$ be an $N^d\times N^d$ complex deterministic matrix such that there exists $0<\kappa_4\leq d$ so that
    \begin{equation}
        \mathrm{rank}(R_N)=\mc{O}(N^{d-\kappa_4})\m\m\m\m\text{and}\m\m\m\m \norme{R_N}=\mc{O}(1) .
    \label{Hypothèse matrice de perturbation}
    \end{equation}
    Then, the empirical spectral measure $\mu_N$ of $M_N(p)+R_N+\delta Q_N$ converges weakly in probability to $p_*\mathrm{L}$.
\label{Corollaire du théorème final}
\end{cor}

\mm Theorem \ref{Théorème final} shows that for $N\gg 1$, the spectrum of small random perturbations of $M_N(p)$ roughly equidistributes in $\Sigma:=\overline{p([0,1]^d\times \T^d)}$. More precisely, it provides the limiting spectral distribution of $M_N(p)+\delta Q_N$ which is the leading part of the empirical spectral measure $\mu_N$ when $N$ is large.  Additionally, Corollary \ref{Corollaire du théorème final} states that this leading part of the empirical spectral measure remains unchanged if we add to $M_N(p)$ a matrix $R_N$ that can have a large rank but a controlled norm.\\

\mm The aim of this paper is to extend Theorem 4.1 in \cite{BasakPaquetteZeitouni01} and Theorem 1.2 in \cite{BPZ2019}, both proved by Basak, Paquette and Zeitouni. In these papers the authors considered one-dimensional symbols $p:[0,1]_x\times \T_\xi\to \C$ of the form \eqref{Forme symbole dans les travaux de BPZ} where $N_-,N_+\geq 0$ are finite integers. On the one hand, in \cite{BasakPaquetteZeitouni01}, they considered the particular case where $N_-=0$ and for all $k\in \disc{0}{N_+}$, $p_k:[0,1]\to \C$ is $\alpha_k$-Hölder continuous where $\alpha_k\in \hspace{0.1cm}]0,1]$, $k\geq 1$ and with the restriction that $1/2<\alpha_0\leq 1$. On the other hand, in \cite{BPZ2019}, they took $N_-,N_+\geq 0$ and treated the case where $p_k$ is constant for all $k\in \disc{-N_-}{N_+}$, so that $p$ in \eqref{Forme symbole dans les travaux de BPZ} becomes a so-called \textit{Laurent polynomial}. \\
\mm Since functions $p_k$ in \eqref{Forme symbole dans les travaux de BPZ} do not have any singularity, symbols considered in \cite{BasakPaquetteZeitouni01} are in $\mathscr{C}^{0,\alpha}_{\mathrm{pw}}S_1$ for $\alpha:=\min_k\alpha_k>0$, and those in \cite{BPZ2019} are in $\mathscr{C}^{0,\alpha}_{\mathrm{pw}}S_1$ for $\alpha=1$. In particular, as mentioned in \eqref{Fait que l'hypothèse sur les singularité soit toujours vraie en dimension 1}, for such symbols, Assumption \ref{Hypothèse théorème final - volume singularités} always holds with $\kappa_1=1$. Furthermore, Proposition \ref{Généralisation BPZ quat} in the Appendix shows that under very mild assumptions, for symbols of the form \eqref{Forme symbole dans les travaux de BPZ}, Assumption \ref{Hypothèse théorème final - volume préimage} always holds for an explicit $\kappa_2$ (see Proposition \ref{Généralisation BPZ quat}). Moreover, the authors associated with the symbol \eqref{Forme symbole dans les travaux de BPZ} the matrix \eqref{Matrice T_N(p) BPZ} so that, in view of definition \eqref{Définition de l'opérateur pour la comparaison avec BPZ} above, our construction coincides with theirs for such symbols. \\
\mm In \cite{BasakPaquetteZeitouni01} (resp. in \cite{BPZ2019}), they proved that, for $\gamma>1/2$, the empirical spectral measure of $M_N(p)+N^{-\gamma}Q_N$ weakly converges in probability to $p_*\mathrm{L}$, the push-forward of the normalized Lebesgue measure on $[0,1]\times \mathbb{S}^1$ (resp. $\mathbb{S}^1$) by $p$ (see \cite[Theorem 4.1]{BasakPaquetteZeitouni01} and \cite[Theorem 1.2]{BPZ2019}). Note that in \cite{BPZ2019}, what they assumed on the random matrix $Q_N$ is quite close to Assumption \ref{Hypothèse théorème final - matrice aléatoire} (see \cite[Assumption 1.1]{BPZ2019}).  \\
\mm As we can observe, the matrix $M_N(p)$ given in \eqref{Matrice T_N(p) BPZ} for one-dimensional symbols $p$ of the form \eqref{Forme symbole dans les travaux de BPZ} is always \textit{finite-banded}. This means that the number of non zero diagonals of $M_N(p)$ (here equals to $N_-+N_++1$) is finite, independent of $N$. In \cite{BasakPaquetteZeitouni01}, more specifically, $M_N(p)$ is exclusively upper triangular. \\
\mm In this paper, we extend their construction to the new symbol class $\mathscr{C}^{0,\alpha}_{\mathrm{pw}}S_d$; that is, in contrast to their results, $M_N(p)$ can now be infinite-banded, defined for general $d$-dimensional symbols $p:[0,1]^d_x\times \T^d_\xi\to \C$, which are also allowed to have singularities of jump type as described in Definition \ref{Définition espace d'intérêt}. \\

\mm In the proof of \cite[Theorem 4.1]{BasakPaquetteZeitouni01} and \cite[Theorem 1.2]{BPZ2019}, Basak, Paquette and Zeitouni involve methods relying on random matrix theory, which are quite different from ours. In this paper, the methods involved are based on tools developed especially in the articles of Vogel \cite{Vogel_2020}, Christiansen-Zworski \cite{Christiansen_2010}, and Hager-Sjöstrand \cite{hager2007eigenvalue}, using semiclassical analysis.

\begin{exmp}
    To illustrate Theorem \ref{Théorème final}, we consider the one-dimensional symbol $p$ defined on $[0,1]_x\times \T_\xi$ by
    \begin{equation}\label{Exemple de symbole pour la simulation}
        p(x,\xi)=f(x)+i\cos(2\pi\xi)
    \end{equation}
    where 
    \begin{equation}\label{Expression de la fonction HPM}
        f(x):=\begin{cases}
         \sqrt{x}-5/2 & \text{if}\m\m\m\m  0\leq x<1/3, \\
         \bb{12x-6}-1 & \text{if} \mm 1/3\leq x<2/3,\\
         1/2+e^{3(1-x)} & \text{if} \mm 2/3\leq x\leq 1,
        \end{cases}
    \end{equation}
\end{exmp}
is a piecewise $1/2$-Hölder continuous function and whose graph has been plotted in Figure \ref{fig:graphe de la fonction HPM}. We also plot in Figure \ref{fig:graphe de la fonction HPM} the red dotted lines to emphasize the gaps in the numerical range of $f$.

\begin{figure}[H]
    \centering
    \includegraphics[width=0.5\linewidth]{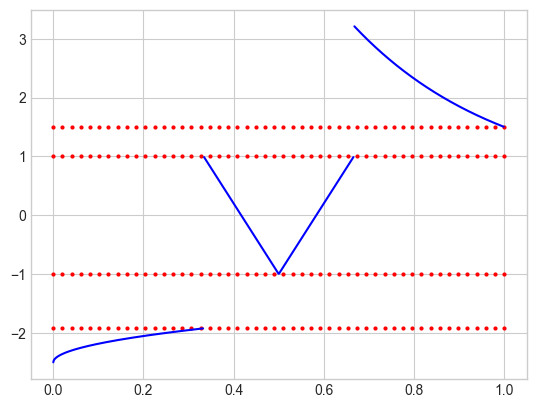}
    \caption{Graph of $f$ given by \eqref{Expression de la fonction HPM}.}
    \label{fig:graphe de la fonction HPM}
\end{figure}

The symbol $p$ defined in \eqref{Exemple de symbole pour la simulation} belongs to $\mathscr{C}^{0,1/2}_{\mathrm{pw}}S_1$. As mentioned in \eqref{Fait que l'hypothèse sur les singularité soit toujours vraie en dimension 1}, $p$ verifies Assumption \ref{Hypothèse théorème final - volume singularités}. Since $f$ is piecewise Hölder continuous on $[0,1]$, Proposition \ref{Généralisation BPZ quat} ensures that Assumption \ref{Hypothèse théorème final - volume préimage} holds for $p$, with $\kappa_2=1/4$. Theorem \ref{Théorème final} then applies.\\

Figure \ref{fig:Simu spectre fonction HPM} shows a numerical simulation of the spectrum of $M_N(p)$ and $M_N(p)+\delta Q_N$. In particular, on the right hand side, we can see that the spectrum fills up the numerical range of the symbol $p$. The gaps in the spectrum come from the gaps in the numerical range of $f$. \\

\begin{figure}[H]
    \centering
    \includegraphics[width=1\linewidth]{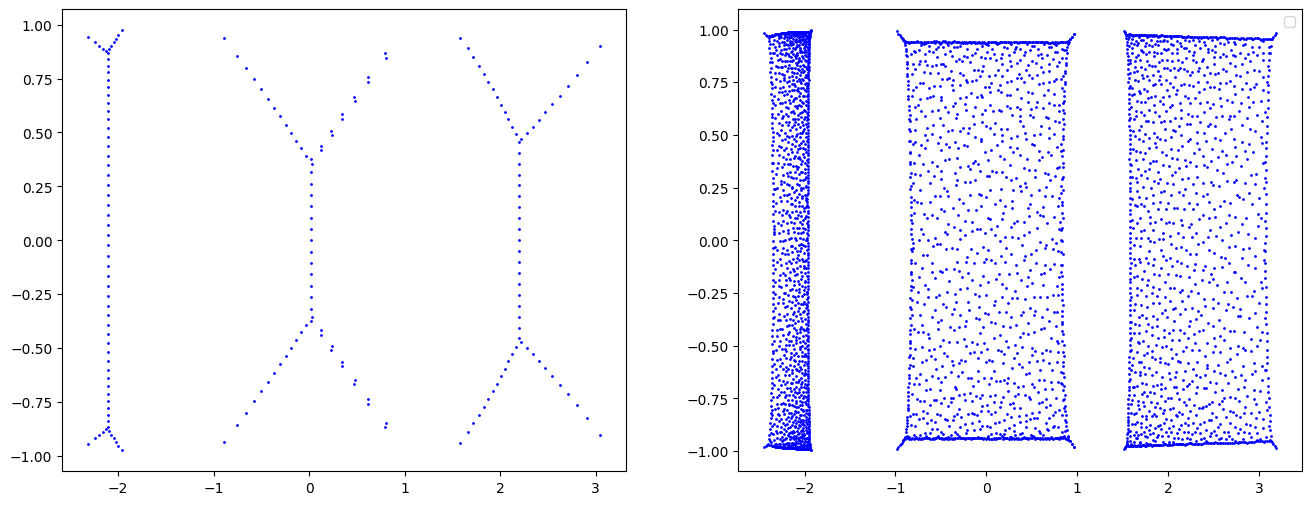}
    \caption{The left hand side presents the spectrum of $M_N(p)$ without any randomness, for $p$ given by \eqref{Exemple de symbole pour la simulation}, and on the right hand side, the spectrum of $M_N(p)+\delta Q_N$ where $Q_N$ is a random Gaussian matrix, $N=4000$ and $\delta=N^{-1.6}$.}
    \label{fig:Simu spectre fonction HPM}
\end{figure}

\begin{exmp}\label{Exemple bloc de Jordan pertubé}
    We now turn to illustrate Corollary \ref{Corollaire du théorème final}. We consider the perturbative Jordan bloc, say $\widetilde{J}_N$ given by
    \begin{equation*}
        \widetilde{J}_N:=\begin{pmatrix}
            0&1&0&0&\cdots&0\\
            0&0&1&0&\ddots&\vdots\\
            0&0&0&1&\ddots&\vdots\\
            \vdots&\ddots&\ddots&\ddots&\ddots&0\\
            0&\ddots&\ddots&\ddots&\ddots&1\\
            2&0 &\cdots&\cdots&0&0
        \end{pmatrix}=J_N+2E_{N,1}
    \end{equation*}
    where $J_N$ is the actual Jordan bloc 
    \begin{equation}\label{Matrice bloc de Jordan}
        J_N=\begin{pmatrix}
            0&1&0&0&\cdots&0\\
            0&0&1&0&\ddots&\vdots\\
            0&0&0&1&\ddots&\vdots\\
            \vdots&\ddots&\ddots&\ddots&\ddots&0\\
            \vdots&\ddots&\ddots&\ddots&\ddots&1\\
            0&\cdots &\cdots&\cdots&0&0
        \end{pmatrix}
    \end{equation}
    and $E_{N,1}$ is the matrix such that for all $j,k\in \disc{1}{N}$,
    \begin{equation*}
        \Bigl(E_{N,1}\Bigl)_{j,k}=\left\{\begin{array}{ccccc}
             & 1 & \text{if} & j=N, \m k=1\\
             & 0 & \text{otherwise}
        \end{array}\right..
    \end{equation*}
    In particular, we have
    \begin{equation*}
        \mathrm{rank}(2E_{N,1})=1\m\m\m\m\text{and}\m\m\m\m \norme{2E_{N,1}}=2.
    \end{equation*}
    One can observe, that $J_N=M_N(p)$ where $p$ is the one-dimensional symbol defined on $[0,1]_x\times\T_\xi$ by $p(x,\xi)=e^{2i\pi\xi}$. Since $p$ does not depend on $x$ and is smooth with respect to $\xi$, we have $p\in \mathscr{C}^{0,1}_{\mathrm{pw}}S_1$.\\
    \mm First, since $d=1$ here, $p$ satisfies Assumption \ref{Hypothèse théorème final - volume singularités} with $\kappa_1=1$. Second, Proposition \ref{Généralisation BPZ quat} in the Appendix shows that, for such a symbol $p$, Assumption \ref{Hypothèse théorème final - volume préimage} holds with $\kappa_2=1/2$. \\
    
    \mm Figure \ref{fig:spectre bloc de jordan perturbé} is a numerical simulation of the spectrum of $\widetilde{J}_N+\delta Q_N$. We can see that the spectrum fills up the numerical range of $p$, that is $\mathbb{S}^1$, even with the deterministic perturbation $E_{N,1}$. We decided to take $2E_{N,1}$ instead of $E_{N,1}$ to show that the eigenvalue $2$ seems still present in the spectrum of $\widetilde{J}_N+\delta Q_N$, even if most of its eigenvalues are on $\mathbb{S}^1$ as predicted.

    \begin{figure}[H]
        \centering
        \includegraphics[width=0.5\linewidth]{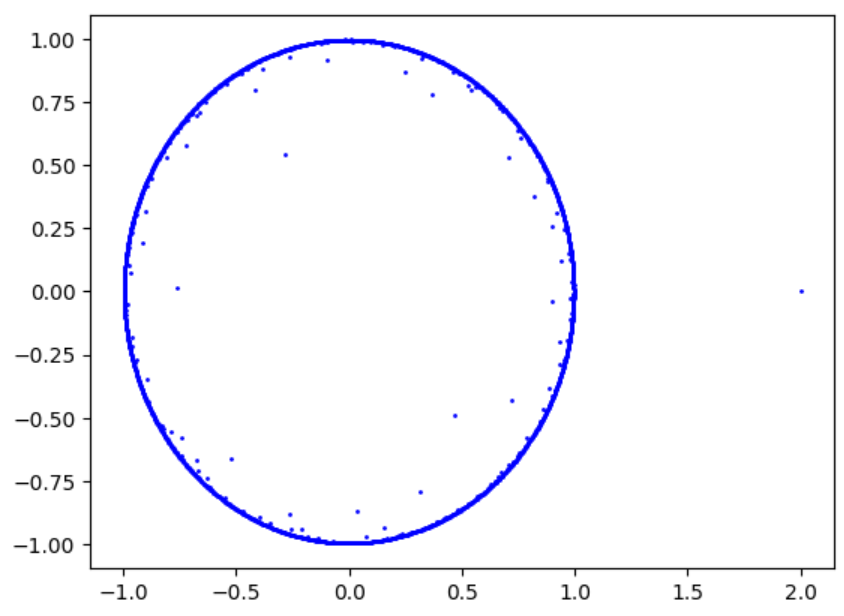}
        \caption{Spectrum of $\widetilde{J}_N+\delta Q_N$, where $Q_N$ is a Gaussian random matrix, $N=2000$ and $\delta=N^{-1.6}$.}
        \label{fig:spectre bloc de jordan perturbé}
    \end{figure}
\end{exmp}

In addition, we can see from Example \ref{Exemple bloc de Jordan pertubé} how important it is to add a random perturbation to obtain the conclusion of Theorem \ref{Théorème final}. In fact, it is clear that the spectrum of $\widetilde{J}_N$ is $\acc{0,2}$ for all $N$. 

\subsection*{Singularities and boundary conditions}
In symbol class $\mathscr{C}^{0,\alpha}_{\mathrm{pw}}S_d$, symbols are allowed to have singularities of jump type. This is a way to tackle the non periodicity of a symbol $p:[0,1]^d_x\times \T^d_\xi\to \C$ with respect to $x$. To apply the semiclassical analysis tools developed in the papers mentioned above (\cite{Vogel_2020,Christiansen_2010,hager2007eigenvalue}), one of the main steps of the proof of Theorem \ref{Théorème final} will be to transform the elements of $\mathscr{C}^{0,\alpha}_{\mathrm{pw}}S_d$ into periodic functions. Let us illustrate this in the one-dimensional case. If we want to transform $p\in \mathscr{C}^{0,\alpha}_{\mathrm{pw}}S_1$ into a periodic function without losing too much information, we can remove its values for $x=0$ and assign at $x=0$ the values it takes at $x=1$. That is, we construct from $p$ the function
\begin{equation*}
    \widetilde{p}:\left\{\begin{array}{ccccc}
         & [0,1]\times \T & \longrightarrow & \C\\
         & (x,\xi) & \mapsto & \left\{\begin{array}{ccccc}
              & p(x,\xi) & \text{if} & x\in\hspace{0.1cm} ]0,1] \\
              & p(1,\xi) & \text{if} & x=0
         \end{array}\right.
    \end{array}\right.
\end{equation*}
which now satisfies $\widetilde{p}(0,\cdot)=\widetilde{p}(1,\cdot)$. However, $\widetilde{p}$ now has a singularity of jump type at $x=0$ but it can be extended by $1$-periodicity to $\R\times \T$. Since we will force an element of $\mathscr{C}^{0,\alpha}_{\mathrm{pw}}S_1$ to have a singularity of jump type through this construction, it is quite natural to allow them to have many singularities of this type along the whole interval $[0,1]$, as long as the volume of the singularities is not too large (see Assumption \ref{Hypothèse théorème final - volume singularités}). 

\section{Related results and comments}
\subsection{Related results}
In the history of Toeplitz matrices, numerous results analogous to Theorem \ref{Théorème final} have been proven. Among them, there are many articles in which the Jordan matrix has been the subject. As mentioned in Example \ref{Exemple bloc de Jordan pertubé}, it is a Toeplitz matrix associated with the one-dimensional symbol $p$ defined on $[0,1]_x\times \T_\xi$ by $p(x,\xi)=e^{2i\pi \xi}$. Davies and Hager \cite{Davies_Hager} and Sjöstrand \cite{Sjöstrand_livre} showed under different assumptions on the coupling constant, that most of the eigenvalues of $J_N+\delta Q_N$ is close to the circle $\mathbb{S}^1$ with probability close to $1$. These results are supplemented by Sjöstrand and Vogel in \cite{sjoestrand2014interioreigenvaluedensityjordan} who provide an accurate description of the eigenvalue distribution of $J_N+\delta Q_N$ in the interior of the disc $D(0,1)$. Moreover, in \cite{Sniadi}, \'{S}niadi made a breakthrough regarding the convergence of empirical spectral measures of randomly perturbed matrices. These results have been enhanced by Guionnet, Wood and Zeitouni in \cite{guionnet2011convergencespectralmeasurenon}, who especially obtain the weak convergence in probability of the empirical spectral measure of $J_N+\delta Q_N$ towards the uniform measure on $\mathbb{S}^1$. In \cite{guionnet2011convergencespectralmeasurenon}, the authors also obtain a similar result to Corollary \ref{Corollaire du théorème final} for the convergence of the empirical measure of $J_N+R_N+\delta Q_N$, but when $\norme{R_N}$ decays polynomially. \\
\mm Toeplitz matrices associated with more general symbols of the form 
\begin{equation}\label{Symbole à partir d'un poly de Laurent}
    p(x,\xi)=\sum_{\nu=-N_-}^{N_+}p_\nu e^{2i\pi\nu\xi}
\end{equation}
where $N_-,N_+\geq 0$, have been studied. In \cite{trefethen2005spectra}, Trefethen and Embree studied the particular case of bidiagonal matrices (that is, $p(x,\xi)=ae^{2i\pi\xi}+be^{-2i\pi\xi}$ with $a,b\in \C$). They emphasized the spectral instability of these matrices under random perturbations with numerical simulations (see \cite[Figures 3.2 and 3.3]{trefethen2005spectra}). Later, in \cite{sjoestrand2015largebidiagonalmatricesrandom}, under the hypothesis $0<\bb{b}<\bb{a}$, Sjöstrand and Vogel proved that the eigenvalue distribution of $M_N(p)+\delta Q_N$ is governed by a Weyl law, so in particular that the eigenvalues tend to be uniformly distributed on the ellipse $p(\T)$ with high probability. They also showed a similar result for $M_N(p)+\delta Q_N$ in \cite{sjoestrand2019toeplitzbandmatricessmall} where $p$ is of the form \eqref{Symbole à partir d'un poly de Laurent}. As mentioned in Section \ref{Introduction}, Basak, Paquette and Zeitouni also obtained in \cite{BPZ2019} the convergence of the empirical spectral measure of $M_N(p)+\delta Q_N$ to $p_*\mathrm{L}$. Their study was strengthened by the article of Bordenave, Capitaine and Chapon \cite{Bordenave_Capitaine_Chapon} in which they provide an accurate and detailled analysis of stable and unstable outliers of such randomly perturbed Toeplitz matrices. This article follows the strategies initiated in \cite{Bordenave_Capitaine} by Bordenave and Capitaine. Finally, Sjöstrand and Vogel showed in \cite{sjoandvogel_general_2021} that for infinite symbols (that is $N_-=N_+=+\infty$ in \eqref{Symbole à partir d'un poly de Laurent}), under mild assumptions on $p(\mathbb{S}^1)$, the eigenvalue distribution of $M_N(p)+\delta Q_N$ is governed by a Weyl law, with probability exponentially close to one. They also show the almost sure convergence of the empirical measure of $M_N(p)+\delta Q_N$ to $p_*\mathrm{L}$.\\
\mm In a similar setting to the one in the present paper, results on the Toeplitz quantization of symbols have been proven. If $p\in \mathscr{C}^\infty(\T^{2d})$ satisfies a condition similar to Assumption \ref{Hypothèse théorème final - volume préimage} with ${\kappa_2}\in \hspace{0.1cm} ]1/2,1]$, then Christiansen and Zworski \cite{Christiansen_2010} proved that the number of eigenvalues in complex domains follows a Weyl law in expectation. They conjectured the almost sure convergence of the associated empirical measure in that setting, which was proven by Vogel in \cite{Vogel_2020}. In this paper, Vogel obtained a Weyl law with high probability, also showed that the result of Christiansen and Zworski remained true when ${\kappa_2} \in \hspace{0.1cm} ]0,1]$ and for a more general class of random perturbations (see \cite[Theorem 8]{Vogel_2020} for further details). In a recent result Oltman \cite{Oltman} extended the result of Vogel to the case of Berezin-Toeplitz operators $p_N$ on a $d$-dimensional Kähler manifold $X$. 

\begin{rem}\label{Remarque inclusion des classes de symboles}
    In \cite{Vogel_2020} and \cite{Christiansen_2010}, the authors make use of symbols $p\in \mathscr{C}^\infty(\T^{2d})$ that are naturally embedded in $\mathscr{C}^{0,1}_{\mathrm{pw}}S_d$. Indeed, such a symbol $p$ can be extended by periodicity to $\R^d\times \T^d$ (through the natural projection $\R^d\to \T^d$), and then restricted to $[0,1]^d\times \T^d$. In this case, we can see that $p\in \mathscr{C}^{0,1}_{\mathrm{pw}}S_d$. 
\end{rem}

\subsection{Comments on the hypotheses of Theorem \ref{Théorème final}}
Assumption \ref{Hypothèse théorème final - volume singularités} is a mild assumption. \label{Comments sur les hypothèses du théorème} As noticed in \eqref{Fait que l'hypothèse sur les singularité soit toujours vraie en dimension 1}, for $d=1$, Assumption \ref{Hypothèse théorème final - volume singularités} is always satisfied with $\kappa_1=1$. For $d>1$, it is well-known that if $A\subset \R^d$ is a set with Lipschitz boundary, then there exists a constant $C>0$ such that for $0<r\ll 1$, 
\begin{equation*}
    L\Bigl(\partial A+B(0,r)\Bigl)\leq Cr.
\end{equation*}
In that case, if $p\in \mathscr{C}^{0,\varrho}_{\mathrm{pw}}S_d$ with $\mathscr{U}_p=\cup_{j=1}^s \partial U_j$ as in Definition \ref{Définition espace d'intérêt}, it is enough that $\partial U_j$ is Lipschitz for all $j$, for $p$ to satisfy Assumption \ref{Hypothèse théorème final - volume singularités}.\\
\mm Otherwise, Assumption \ref{Hypothèse théorème final - volume singularités} is strongly linked to the Minkowski dimension (also called the box-counting dimension) of $\mathscr{U}_p$, which allows $\mathscr{U}_p$ to have a fractal structure. In particular, if $\mathscr{U}_p$ has a Minkowski dimension, say $\mathfrak{s}$, then Assumption \ref{Hypothèse théorème final - volume singularités} is fulfilled with $\kappa_1=d-\mathfrak{s}>0$ (see \cite{Federer,Evans_Gariepy,Mattila} for further details).\\

\mm Assumption \ref{Hypothèse théorème final - volume préimage} was similarly made in \cite{Christiansen_2010,hager2007eigenvalue,Oltman,Vogel_2020}. As for symbols $p\in \mathscr{C}^\infty(\T^{2d})$, which are in $\mathscr{C}^{0,1}_{\mathrm{pw}}S_d$ by Remark \ref{Remarque inclusion des classes de symboles}, Christiansen and Zworki \cite{Christiansen_2010}, and Vogel \cite{Vogel_2020} provide criteria for $p$ to satisfy Assumption \ref{Hypothèse théorème final - volume préimage}. We refer to \cite{Christiansen_2010} and \cite{Vogel_2020} for further details. In addition, Proposition \ref{Généralisation BPZ quat} in the Appendix establishes a criterion for higher-dimensional analog of \eqref{Forme symbole dans les travaux de BPZ} to satisfy Assumption \ref{Hypothèse théorème final - volume préimage}. In fact, if $p$ is of the form 
\begin{equation}
    p(x,\xi)=\sum_{k\in \Lambda}p_k(x)e^{2i\pi \ps{k}{\xi}}, \quad(x,\xi)\in [0,1]^d\times \T^d,
\end{equation}
where $\Lambda=\prod_{j=1}^d\disc{-n_j}{m_j}$, $n_j,m_j\geq 0$, and $p_k:[0,1]^d\to \C$, $k\in \Lambda$ is measurable, then it requires that: there exists $m>0$ such that for almost every $x\in [0,1]^d$, 
\begin{equation*}
    \sum_{\substack{k\in \Lambda\\k\neq 0}}^{N_+}\bb{p_k(x)}\geq m>0.
\end{equation*} 
In this case, Assumption \ref{Hypothèse théorème final - volume préimage} is fulfilled with an explicit $\kappa_2>0$ (see Proposition \ref{Généralisation BPZ quat}).\\

\m\m\m\m Furthermore, Assumption \ref{Hypothèse théorème final - matrice aléatoire} on the random perturbation appears in \cite{vogel2020deterministicequivalencenoisyperturbations} and holds for a large class of random matrices. Vogel and Zeitouni gave in \cite[Remark 4]{vogel2020deterministicequivalencenoisyperturbations} and Basak, Paquette and Zeitouni in \cite[Remark 1.3]{BPZ2019}  examples of random matrices $Q_N$ fulfilling Assumption \ref{Hypothèse théorème final - matrice aléatoire}. We give here their examples for the reader's convenience.
\begin{enumerate}
    \item When the entries of $Q_N$ are independent copies of a random variable $Z$ with expectation $0$ and finite variance.
    \item When the entries of $Q_N$ are independent and controlled by a single distribution $Z$ in the Fourier-analytic sense that has a ${\kappa}$-controlled second moment for some ${\kappa}>0$ (see \cite[Definition 2.2 and Remark 2.8]{Tao_2009}).
    \item When $Q_N=\sqrt{N}U_N$ when $U_N$ is a Haar distributed unitary matrix. This is a result by Rudelson and Vershynin in \cite[Theorem 1.1]{Rudelson_Vershynin}.
    \item When the entries $q^N_{j,k}$ of $Q_N$ are such that there exists $a>0$ so that 
    \begin{equation*}
        \min_{1\leq j,k\leq N^d}\mathbb{P}\pr{\bb{q^N_{j,k}}\leq a}>0 \mm\mm\text{and}\mm\mm \min_{1\leq j,k\leq N^d}\mathrm{Var}\pr{q^N_{j,k}\mathbb{1}_{\acc{\bb{q^N_{j,k}}\leq a}}}>0,
    \end{equation*}
    (see \cite[Lemma A.1]{Bordenchafai_potentielloga} for further details).
    \item When $Q_N=A\cdot X+B$, where $A$ is an $N^d\times N^d$ deterministic matrix whose entries are elements of $[0,1]$ satisfying a super-regularity condition (see \cite[Definition 1.23]{Cook}), $B$ is a deterministic complex matrix. Here, $\cdot$ denotes the Hadamard product of matrices. This result is attributed to Cook in \cite[Theorem 1.24]{Cook}.
\end{enumerate}

\subsection{Outline of the proof}\label{Outline}
We will start by recalling basic notions of semiclassical analysis in Section \ref{tools}, which is an essential tool for this paper. \\
\mm Moreover, in Section \ref{Quantif}, we define and study the quantization procedure for rough symbols, which justifies the construction of the operator \eqref{Expression nouvelle procédure de quantif}. We also discuss another quantization procedure mapping smooth periodic symbols $p$ to the $N^d\times N^d$ matrices $p_N$, for which we develop a functional calculus. The following result is obtained; we present here a simplified version.

\begin{prop}
    Let $N\gg 1$ and let $0<\A{1},\A{2}\leq 1$ be such that $N^{-1}\ll  \pr{\A{1}\A{2}}^{1/2}$. Let $p:\R^d_x\times \R^d_\xi\to \C $ be smooth and such that $p(\cdot,\xi)$ is $\alpha_1^{-1/2}\Z^d$ periodic and $p(x,\cdot)$ is $\alpha_2^{-1/2}\Z^d$ periodic. Assume that $p\geq 0$, $p+i$ is elliptic and admitting an asymptotic expansion. Then, for all $\psi\in \mathscr{C}^{\infty}_c(\R)$, there exists a smooth periodic $f$ on $\R^d_x\times \R^d_\xi$ (with same periodicity as that of $p$) such that
    \begin{equation*}
         \psi(p_{N})=f_{N}.
    \end{equation*}
\end{prop}
\mm In Section \ref{Résu funda}, we will approximate rough symbols in $\mathscr{C}^{0,\varrho}_{\mathrm{pw}}S_d$ with smooth $h$-dependent periodic symbols in $\mathscr{C}^\infty(\T^{2d})$. The idea is to restrict $p\in\mathscr{C}^{0,\varrho}_{\mathrm{pw}}S_d$ to $]0,1]^d_x\times\T^d_\xi$ and to extend it by $\Z^d$ periodicity with respect to $x$, on $\R^d_x\times \T^d_\xi$ to a function still denoted $p$. Taking an $h$-dependent regularizing function $\psi_h$, we set 
\begin{equation*}
    \widetilde{p}(x,\xi):=\Bigl(p(\cdot,\xi)\ast \psi_h\Bigl)(x), \mm x\in \R^d\m
\end{equation*}
and
\begin{equation*}
    \widetilde{q}:=\overline{\widetilde{p}}\sharp_h\widetilde{p},
\end{equation*}
which belongs to a slightly exotic symbol class (see Section \ref{tools} for the definition of $\sharp_h$). We obtain the following result, a simplified version of which is provided here:

\begin{thm}
    Let $0<\eta<1/4$ and $\varrho\in \hspace{0.1cm} ]0,1]$. Let $p\in \mathscr{C}^{0,\varrho}_{\mathrm{pw}}S_d$ and let $\mathscr{U}_p$ be given in \eqref{Ensemble des singularités}. We assume that there exist ${\kappa_2}\in \hspace{0.1cm} ]0,1]$ such that, for $0\leq t\ll 1$,
    \begin{equation*}
        \mathrm{Vol}\pr{\acc{\rho\in \T^{2d}\m|\m \bb{p(\rho)}^2\leq t}}=\mc{O}(t^{\kappa_2})
    \end{equation*}
    and ${\kappa_1}\in \hspace{0.1cm} ]0,1]$ such that for $0\leq t\ll 1$,
    \begin{equation*}
        \mathrm{L}\Bigl(\pr{\mathscr{U}_p+B(0,t)}\cap[0,1]^d\Bigl)=\mc{O}(t^{{\kappa_1}}).
    \end{equation*}
    Let $N\gg 1$ and $\alpha$ be such that
    \begin{equation*}
        N^{-\eta\min(\varrho,{\kappa_1})}\ll\alpha\ll 1.
    \end{equation*}
    Then, for all $\psi\in \mathscr{C}^\infty_c(\R)$,
    \begin{equation*}
       \mathrm{tr}\!\croch{\psi\pr{\frac{\widetilde{q}_{N}}{\alpha}}} =N^d\pr{\int_{\T^{2d}}\psi\pr{\frac{\bb{p(\rho)}^2}{\alpha}}d\rho+o(1)}.
    \end{equation*}
    If $\chi\in \mathscr{C}^\infty_c(\R,[0,+\infty[)$, such that $\chi(0)>0$ then,
    \begin{equation*}
        \begin{split}
            \log\det\pr{\widetilde{q}_N+\alpha\chi\pr{\frac{\widetilde{q}_N}{\alpha}}} &= N^d\biggl(\int_{\T^{2d}}\log\pr{\bb{p(\rho)}^2}d\rho+o(1)\biggl).
        \end{split}
    \end{equation*}
\end{thm}
Under suitable hypotheses on the parameters, and as a consequence of this result, we can obtain an estimate on the number $\mathfrak{m}$ of singular values of $\widetilde{q}_N$ that are smaller than $\alpha$.

\mm In Section \ref{Grushin}, we set up a Grushin problem for $\widetilde{p}_N-z$. For $z\in \C$, the operator to be examined is
\begin{equation*}
    \mc{P}(z):=\begin{pmatrix}
        \widetilde{p}_N-z&R_-(z)\\
        R_+(z)&0
    \end{pmatrix}:\C^{N^d}\times \C^\mathfrak{m}\longrightarrow\C^{N^d}\times\C^\mathfrak{m}
\end{equation*}
where $R_+(z)$ and $R_+(z)$ are suitably chosen so that $\mc{P}(z)$ is invertible. We show that 
\begin{equation*}
    \begin{split}
        \log\pr{ \bb{\det\pr{\mc{P}(z)}}} &= N^d\biggl(\int_{\T^2}\log\pr{\bb{p(\rho)-z}^2}d\rho+o(1)\biggl).
    \end{split}
\end{equation*}
\mm In Section \ref{log_pot}, we give the proper definition of the logarithmic potential $\phi_\mu$ associated with a measure $\mu$. We will be interested in the empirical spectral measure $\mu_N$ of
\begin{equation*}
    M_N(p)+\delta Q_N
\end{equation*}
where $\delta$ is polynomially decreasing with respect to $N$, $M_N(p)$ is given in \eqref{Définition de l'opérateur pour la comparaison avec BPZ} and $Q_N$ is an $N^d\times N^d$ random matrix. Taking
\begin{equation*}
    \mu:=p_*\mathrm{L}
\end{equation*}
where $\mathrm{L}$ is the Lebesgue measure on $[0,1]_x^d\times \T^d_\xi$ and using the estimates developed in Section \ref{Grushin}, we show that for all $z\in \C$,
\begin{equation*}
    \bb{\phi_{\mu_N}(z)-\phi_\mu(z)}=o(1)
\end{equation*}
with high probability as $N\to +\infty$.\\
\mm We finally provide in Section \ref{Démo du théorème final} a proof of Theorem \ref{Théorème final} using a criterion for weak convergence of measures (see e.g. \cite{taotopics}) and the estimates of the Section \ref{log_pot}.

\subsection*{Notations}
Throughout this paper, we will use the following notations: $\ps{\cdot}{\cdot}$ represents the canonical inner product on $\C^m$, for all $m\in \N^*:=\N\setminus\{0\}$ and the associated $\ell^2$-norm will be denoted $\norme{\cdot}$, as well as for the $\ell^2\to \ell^2$ norm of $m\times m$ complex square matrices. We will also frequently use the Hilbert-Schmidt norm of a matrix defined by $\norme{A}_{\mathrm{HS}}:=\mathrm{tr}(A^*A)^{1/2}$. The Japanese bracket of $x\in \R^m$ will be written as
\begin{equation*}
    \jap{x}:=\sqrt{1+\norme{x}^2}.
\end{equation*}
\mm If $A$ is a measurable set and $p\in [1,+\infty]$, we will denote $\norme{f}_{L^p(A)}$ the $L^p$-norm of $f$ on $A$ defined by
\begin{equation*}
    \norme{f}_{L^p(A)}:=\pr{\int_{A}\bb{f(x)}^p dx}^{1/p}
\end{equation*}
if $1\leq p<\infty$, and $\norme{f}_{L^{\infty}(A)}:=\mathrm{ess\m sup}_{A}(f)$. In this paper, unless otherwise stated, $\mathrm{L}$ will always denote the Lebesgue measure, whether on $\C$, $\R^d$ or even $\T^d$ without distinction. However, when integrating with respect to the Lebesgue measure, we will use the notations $d\rho$, $dx$, $d\xi$, $\mathrm{L}(dz)$ whenever convenient. \\
\mm When we write $a_N=\mc{O}(b_N)$ (as $N\to +\infty$), we mean that there exists a constant $C>0$ independent of $N$ such that $\bb{a_N}\leq C\bb{b_N}$. If we want to emphasize that the constant depends on a parameter $t$, then, we will write it $C_t$. We adopt the same idea for the big-$\mc{O}$ notation denoting $\mc{O}_t(b_N)$. For another comparison, we use the notation $a_N=o(b_N)$ (as $N\to +\infty$) when $\frac{a_N}{b_N}\to 0$ as $N\to +\infty$.\\
The notation $a\ll b$ means that $Ca\leq b$ for some sufficiently large constant $C>0$.\\ 
\mm Finally, for all $a,b\in \R$ with $a\leq b$, we denote $\disc{a}{b}:=\Z\cap [a,b]$.

\subsection*{Acknowledgment}
I am deeply grateful to Martin Vogel for many enlightening discussions and for his insightful comments on early drafts of this paper. I also thank Raphaël Côte for his helpful suggestions.\\
\mm I would also like to thank the Mittag-Leffler Institute (Sweden) for providing a warm, welcoming, and fruitful environment, where the core ideas of this work were developed. 

\section{Semiclassical calculus}
\mm In \label{tools} this section, we introduce some notions of semiclassical calculus that are needed for the rest of the paper. For more details, the reader can see \cite{dimassi1999spectral,martinez,Zworski}.\\

\mm Let $h\in \hspace{0.1cm} ]0,1]$ denote the semiclassical parameter and let $d\geq 1$ be an integer. We will often use in the following the semiclassical Fourier transform 
\begin{equation*}
    \mc{F}_h(f)(\xi):=\int_{\R^d}e^{-\frac{i}{h}\ps{x}{\xi}}f(x)dx,\mm f\in \mathscr{S}(\R^d).
\end{equation*}
It maps bijectively the Schwartz functions space $\mathscr{S}(\R^d)$ into $\Schw{d}$, $\Schwp{d}$ into $\Schwp{d}$ by duality and is a bijective isometry (up to an $h$-dependent constant) of $L^2(\R^d)$.\\

\mm As mentioned in Section \ref{Outline}, a symbol $p\in \mathscr{C}^{0,\varrho}_{\mathrm{pw}}S_d$ will be extended on $\R^d_x\times \T^d_\xi$, and regularized using a mollifier. To study such symbols, we introduce the basic tools we will be needed below. 

\mm We say that a continuous function $m:\R^{2d}\to \hspace{0.1cm}]0,+\infty[$ in an \textit{order function} if there exist $C_0>0$ and $N_0\geq 0$ such that for all $\rho,\mu\in \R^{2d}$,
\begin{equation}
    m(\rho)\leq C_0\jap{\rho-\mu}^{N_0}m(\mu).
\label{Recall d'analyse semicl - def fonction d'ordre}
\end{equation}
Notice that the set of order functions is closed under multiplication and taking the multiplicative inverse. For $\eta_1,\eta_2\in [0,1/2]$, we define the following symbol class associated with $m$, 
\begin{equation}\label{Recall d'analyse semicl - classe de symbole}
    S_{\eta_1,\eta_2}(m)
\end{equation}
by the set of functions $p\in \mathscr{C}^{\infty}(\R^{2d})$ such that for all $\beta,\gamma\in \N^d$, there exists $C_{\beta,\gamma}>0$ such that
\begin{equation}\label{Recall d'analyse semicl - classe de symbole version pour les constantes}
    \bb{\partial^\beta_x\partial^\gamma_\xi p(x,\xi)}\leq C_{\beta,\gamma}h^{-\eta_1\bb{\beta}-\eta_2\bb{\gamma}}m(x,\xi).
\end{equation}
A symbol $p$ in $S_{\eta_1,\eta_2}(m)$ is allowed to depend on $h$, but in this case, we require that the constants $C_{\beta,\gamma}$ in  \eqref{Recall d'analyse semicl - classe de symbole version pour les constantes} are uniform with respect to $h$. We denote $S(m):=S_{0,0}(m)$ and $S_\xi(m)$ the symbols in $S(m)$ that are $\Z^d$-periodic with respect to $\xi$, that is for all $(x,\xi)\in \R^{2d}$, $\gamma\in\Z^d$,
\begin{equation}\label{Def Z^d-périodicité en xi}
    p(x,\xi+\gamma)=p(x,\xi).
\end{equation}
\m\\
For a given symbol $p\in S_{\eta_1,\eta_2}(m)$, we write
\begin{equation}
    p\sim p_0+hp_1+\dots \mm\text{in}\mm S_{\eta_1,\eta_2}(m), \mm p_\nu\in S_{\eta_1,\eta_2}(m)
\label{Recall d'analyse semicl - devpt asymptotique}
\end{equation}
if for all $N\in \N^*$, $p-\sum_{\nu=0}^{N-1}h^\nu p_\nu\in h^NS_{\eta_1,\eta_2}(m)$. In that case, $p_0$ is said to be the \textit{principal symbol} of $p$. By Borel summation, we can show that for a given sequence $(p_\nu)_\nu$ of $S_{\eta_1,\eta_2}(m)$, there exists a symbol $p\in S_{\eta_1,\eta_2}(m)$ such that \eqref{Recall d'analyse semicl - devpt asymptotique} holds (see for example \cite[Theorem 4.15]{Zworski}).\\

For $t\in [0,1]$, we define the $t$-quantization of a symbol $p\in S_{\eta_1,\eta_2}(m)$ acting on $\mathscr{S}(\R^d)$ by
\begin{equation}
    \mathrm{Op}_{h,t}(p)u(x):=\frac{1}{(2\pi h)^d}\int_{\R^d}\int_{\R^d}e^{\frac{i}{h}\ps{x-y}{\xi}}p\Bigl(tx+(1-t)y,\xi\Bigl)u(y)dyd\xi\m
\label{Recall d'analyse semicl - quantification indice t}
\end{equation}
where the integral with respect to $\xi$ must be seen as an oscillatory integral. For the $1$-quantization, we will denote $\mathrm{Op}_h$ instead of $\mathrm{Op}_{h,1}$ and for the $1/2$-quantization, i.e. the Weyl quantization, $\mathrm{Op}_h^w$, $p^w(x,hD)$ or simply $p^w$ instead of $\mathrm{Op}_{h,1/2}$.\\
Via integration by parts, $\mathrm{Op}_{h,t}$ can be shown to map continuously $\mathscr{S}(\R^d)\to \mathscr{S}(\R^d)$ and $\mathscr{S}'(\R^d)\to \mathscr{S}'(\R^d)$ by duality. Furthermore, if $m$ is bounded, $\mathrm{Op}_{h,t}$ is bounded as an operator $L^2(\R^d)\to L^2(\R^d)$. We recall that he formal adjoint of $\mathrm{Op}_{h,t}(p)$ is
\begin{equation*}
    \mathrm{Op}_{h,t}(p)^*=\mathrm{Op}_{h,1-t}(\overline{p}).
\end{equation*}

\mm If $p,q\in S_{\eta_1,\eta_2}(m)$, then we have the following composition rule
\begin{equation}\label{Recall d'analyse semicl - formule de composition}
    \mathrm{Op}_h^w(p)\circ\mathrm{Op}_h^w(q)=\mathrm{Op}_h^w(p\sharp_h q)
\end{equation}
where for every order functions $m_1,m_2$ and $\eta_1,\eta_2,\delta_1,\delta_2\in [0,1/2]$,
\begin{equation}
    \begin{array}{ccccc}
         & S_{\eta_1,\eta_2}(m_1)\times S_{\delta_1,\delta_2}(m_2) & \to & S_{\omega_1,\omega_2}(m_1m_2) \\
         & (p,q) & \mapsto & p\sharp_h q:=e^{\frac{ih}{2}\sigma(D_x,D_\xi,D_y,D_\eta)}\big(p(x,\xi)q(y,\eta)\big)_{\Bigl|\substack{y=x\\\eta=\xi}}
    \end{array}
\label{Recall d'analyse semicl - expression produit de Moyal}
\end{equation}
with $\omega_j=\max(\eta_j,\delta_j)$, $j\in \{1,2\}$, $\sigma(x,\xi,y,\eta):=\ps{\xi}{y}-\ps{x}{\eta}$ and $D=\frac{1}{i}\partial$. Moreover, if $\omega_1+\omega_2<1$, then we also have the asymptotic expansion
\begin{equation}
    p\sharp_hq\sim \sum_{k\in \N}\frac{1}{k!}\pr{\frac{ih}{2}\sigma(D_x,D_\xi,D_y,D_\eta)}^k\big(p(x,\xi)q(y,\eta)\big)_{\Bigl|\substack{y=x\\\eta=\xi}}\m\m \text{in}\m\m S_{\omega_1,\omega_2}(m_1m_2).
\label{Recall d'analyse semicl - dvpt asymptotique du produit de Moyal}
\end{equation}
\m\\
Finally, a symbol $p\in S_{\eta_1,\eta_2}(m)$ is said to be \textit{elliptic} if there exists a constant $C>0$ independent of $h$ such that for all $\rho=(x,\xi)\in \R^{2d}$,
\begin{equation*}
    \bb{p(\rho)}\geq \frac{1}{C}m(\rho).
\end{equation*}

\section{Quantization procedures}\label{Quantif}
\subsection{The non periodic case}
In this section, \label{Explanation of the quantization procedure} we justifies the construction of the operator \eqref{Expression nouvelle procédure de quantif}.\\

\mm Let us consider $m$ an order function. In what follows, if $f$ is a function defined on $\R^d$, we will denote for all $\gamma\in \R^d$,
\begin{equation}\label{Translation par gamma}
    \tau_\gamma f:=f(\cdot-\gamma).
\end{equation}

\begin{defn}
    We define $\mathscr{S}'_p(\R^{d})$ as the set of tempered distributions that are $\Z^d$ periodic in frequency, that is, the set of $u\in \mathscr{S}'(\R^d)$ such that,
\begin{equation*}
    \forall n\in \Z^d, \m\m\m\tau_{n}\mc{F}_{h}(u)=\mc{F}_{h}(u).
\end{equation*}
\end{defn}

\begin{lem}
    Let $p\in S_{\xi}(m)$. Then, $\mathrm{Op}_h(p):\mathscr{S}'_p(\R^d)\to \mathscr{S}'_p(\R^d)$.
\end{lem}
\begin{proof}
    First, we can notice that for all $\mu\in \R^d$, $e^{\frac{i}{h}\ps{\cdot}{\mu}}\mathrm{Op}_h(p)e^{-\frac{i}{h}\ps{\cdot}{\mu}}=\mathrm{Op}_h(\tau_{(0,\mu)}p)$. In particular, if $\mu\in \Z^d$, 
    \begin{equation}
        e^{\frac{i}{h}\ps{\cdot}{\mu}}\mathrm{Op}_h(p)e^{-\frac{i}{h}\ps{\cdot}{\mu}}=\mathrm{Op}_h(p).
    \label{Conjugation by M_gamma_mu bis}
    \end{equation}
    The result then follows by duality, using \eqref{Conjugation by M_gamma_mu bis} and standard properties of the Fourier transform.
\end{proof}

\m\m\m\m Until the end of the paper, we will denote $\mathrm{PS}$ the set of sequences with at most polynomial growth, that is
\begin{equation}
    \mathrm{PS}:=\bigcup_{k\in \N}\mathrm{PS}_k
\label{Suites croissance poly}
\end{equation}
where
\begin{equation*}
    \mathrm{PS}_k:=\acc{(c_n)_{n\in \Z^d},\m\left|\m \exists C>0, \m \forall n\in \Z, \m \bb{c_n}\leq C\jap{n}^k\right.}.
\end{equation*}

\mm We have the useful result: 
\begin{lem}
    Let $u\in \mathscr{S}'(\R^d)$. Then,
    \begin{equation*}
        u\in \mathscr{S}_p'(\R^d)\m\m\Longleftrightarrow \m\m\exists (c_n)\in \mathrm{PS}, \m u=\sum_{n\in \Z^d}c_n \delta_{2\pi n h}.
    \end{equation*}
    In that case, the sequence $(c_n)$ is unique. 
\label{Lemme peigne de Dirac}
\end{lem}

\begin{proof}
    1. Let us write $u=\sum_{n\in \Z^d}c_n \delta_{2\pi n h}$ for some $(c_n)\in \mathrm{PS}$. Then, the sum converges in $\Schwp{d}$. Indeed, let us take $C>0$ and $k\in \N$ satisfying the property of $\mathrm{PS}$ for $(c_n)$. If $\varphi\in \mathscr{S}(\R^d)$, there exists $C'>0$ such that for all $\rho\in \Rd$,
    \begin{equation*}
        \bb{\varphi(x)}\leq \frac{C'}{\jap{x}^{k+d+1}}\max_{\substack{\gamma\in \N^d\\\bb{\gamma}\leq k+d+1}}\norme{x^\gamma \varphi}_{L^\infty(\R^d)}
    \end{equation*}
    So, 
    \begin{equation}\label{Peigne de Dirac bien def}
        \begin{split}
            \bb{\ps{u}{\varphi}} &= \bb{\sum_{n\in \Z}c_n\varphi(2\pi n h)}\\
            &\leq \pr{\frac{1}{(2\pi h)^{d+1}}\sum_{n\in \Z^d}\frac{C}{\jap{n}^{d+1}}}\max_{\substack{\gamma\in \N^d\\\bb{\gamma}\leq k+d+1}}\norme{x^\gamma \varphi}_{L^\infty(\R^d)}<+\infty.
        \end{split}
    \end{equation}
    Inequality \eqref{Peigne de Dirac bien def} also ensures the continuity of $u$. Furthermore, $\mc{F}_h(u)$ is $\Z^d$ periodic since for all $\gamma\in \Z^d$,
    \begin{equation*}
        \tau_{\gamma}\mc{F}_h(u)=\tau_{\gamma}\pr{\sum_{n\in\Z}c_n e^{-2i\pi \ps{n}{\cdot}}}=\sum_{n\in \Z}c_n e^{-2i\pi \ps{n}{\cdot-\gamma}}=\mc{F}_h(u).
    \end{equation*}
    \mm 2. Conversely, we assume that $u\in \mathscr{S}'_p(\R^d)$. Take $\phi\in \mathscr{C}_c^{\infty}(\R^d,\R)$ such that $\sum_{n\in \Z^d}\tau_{n}\phi=1$ (see for example \cite[Theorem 1.4.6]{Hörmander_1}). By duality, for all $\psi\in \mathscr{S}(\R^d)$,
    \begin{equation*}
        \begin{split}
            \ps{u}{\psi} &= \ps{\mc{F}_h(u)}{1\times\mc{F}_h^{-1}(\psi)}=\ps{\mc{F}_h(u)}{\sum_{n\in \Z^d}\tau_{n}\phi\mc{F}_h^{-1}(\psi)}\\
            &=\sum_{n\in \Z^d}\ps{\mc{F}_h(u)}{\tau_{n}\phi\mc{F}_h^{-1}(\psi)}= \sum_{n\in \Z^d}\ps{\mc{F}_h(u)}{\phi\tau_{-n}\mc{F}_h^{-1}(\psi)}\\
            &=\ps{\phi \mc{F}_h(u)}{\sum_{n\in \Z^d}\tau_{-n}\mc{F}_h^{-1}(\psi)}
        \end{split}
    \end{equation*}
    where the fourth equality comes from the $\Z^d$ periodicity of $\mc{F}_h(u)$. Moreover, using Poisson's summation formula (see \cite[Section 7.2]{Hörmander_1}), we have for all $\xi\in \R^d$,
    \begin{equation*}
        \begin{split}
            \sum_{n\in \Z^d}\tau_{-n}\mc{F}_h^{-1}(\psi)(\xi) &= \sum_{n\in \Z^d}\tau_{-\xi}\mc{F}_h^{-1}(\psi)(n)=\sum_{n\in \Z^d}\mc{F}_h^{-1}\pr{e^{\frac{i}{h}\ps{\xi}{\cdot}}\psi}\!(n)\\
            &= \sum_{n\in \Z^d}e^{2i\pi \ps{\xi}{n}}\psi\pr{2\pi hn}.
        \end{split}
    \end{equation*}
    Finally, we obtain
    \begin{equation*}
        \ps{u}{\psi}=\sum_{n\in \Z^d}\ps{\mc{F}_h(u)}{e^{2i\pi \ps{n}{\cdot}}\phi}\psi\pr{2\pi hn}.
    \end{equation*}
    Thus, denoting $c_n(u):=\ps{\mc{F}_h(u)}{e^{2i\pi \ps{n}{\cdot}}\phi}$, we have in the space $\mathscr{S}'(\R^d)$,
    \begin{equation}
        u=\sum_{n\in \Z^d}c_n(u)\delta_{2\pi hn}.
    \end{equation}
    Furthermore, from the continuity of $u$, there exists $r\geq 0$ and $C_r>0$ such that 
    \begin{equation}\label{Continuité d'une distri tempérée}
        \bb{\ps{u}{\varphi}}\leq C_r \max_{\substack{\gamma,\beta\in \N^d\\\bb{\gamma},\bb{\beta}\leq r}}\norme{x^\gamma\partial ^\beta_x\varphi}_{L^\infty(\R^d)}.
    \end{equation}
    After noticing that for all $n\in \Z^d$, 
    \begin{equation*}
        c_n(u)=\ps{u}{\mc{F}_h(e^{2i\pi \ps{n}{\cdot}}\phi)}=\ps{u}{\mc{F}_h(\phi)(\cdot+2\pi n h)},
    \end{equation*}
    we have from \eqref{Continuité d'une distri tempérée}, that for all $n\in \Z^d$,
    \begin{equation*}
        \bb{c_n(u)}\leq C_r \max_{\substack{\gamma,\beta\in \N^d\\\bb{\gamma},\bb{\beta}\leq r}}\norme{(\cdot-2\pi h n)^\gamma\partial ^\beta_x\mc{F}_h(\phi)}_{L^\infty(\R^d)}\leq \widetilde{C}_r\jap{n}^r
    \end{equation*}
    for some constant $\widetilde{C}_r>0$. So $(c_n(u))\in \mathrm{PS}_r\subset \mathrm{PS}$.\\

    \mm To treat the uniqueness, let us first notice that the linearity of the mapping $\mathrm{PS}\ni (c_n)\mapsto \sum_{n\in \Z}c_n\delta_{2\pi n h}$ allows us to consider a sequence $(c_n)$ such that $\sum_{n\in \Z}c_n\delta_{2\pi n h}=0$ in $\mathscr{S}'(\R^d)$. Let us fix $m\in \Z^d$ and $\phi\in \mathscr{S}(\R^d)$ such that $\mathrm{supp}(\phi)\subset B(m,\pi h)$ and $\phi(m)=1$. Then $0=\ps{u}{\phi}=c_m$. Since $m$ has been arbitrarily chosen, $(c_n)_n=0$.
\end{proof}

\m\m\m\m Thanks to Lemma \ref{Lemme peigne de Dirac}, we can identify every $u\in \mathscr{S}_p'(\R^d)$ with its associated sequence $(c_n)_n$ so that we have: $\mathscr{S}'_p(\R^d)\cong \mathrm{PS}$. When $p\in S_{\xi}(m)$, this therefore allows us to consider that
\begin{equation*}
    \mathrm{Op}_h(p):\mathrm{PS}\to \mathrm{PS}.
\end{equation*}
\m\\
\mm Now fix $p\in S_{\xi}(1)$ and let us determine the behavior of $\mathrm{Op}_h(p)$ upon $\mathrm{PS}$. To do so, we decompose $p$ with its partial Fourier transform:
\begin{equation*}
    p(x,\xi)=\sum_{\nu\in \Z^d}p_\nu(x)e^{2i\pi\ps{\nu}{\xi}}
\end{equation*}
where for all $\nu\in \Z^d$, $p_\nu(x)=\int_{\T^d}p(x,\xi)e^{-2i\pi\ps{\nu}{\xi}}d\xi$. Furthermore, integration by part gives that for all $k\in \N$, $x,\xi\in \R^d$,
\begin{equation}
    p_\nu(x)=\mc{O}_k(\jap{\nu}^{-k}).
\label{Décroissance des coef de Fourier par rapport à l'indice}
\end{equation} 
\mm In what follows, we will denote for all $x\in \R^d$, $\widehat{p}(x):=(p_\nu(x))_{\nu}$ and $c=(c_n)_n$. Assume that $c\in \mathrm{PS}$ and take $k\in \N$ such that $c\in \mathrm{PS}_k$. We can then notice from \eqref{Décroissance des coef de Fourier par rapport à l'indice} that for all $n,\nu\in \Z^d$, $M\in \N$
\begin{equation*}
    \begin{split}
        \bb{p_\nu(2\pi nh)c_{n-\nu}} &\leq \mc{O}_M(1) \jap{\nu}^{-M}\times\jap{n-\nu}^k\leq \mc{O}_M(1)\jap{n}^{k}\jap{\nu}^{k-M}.
    \end{split}
\end{equation*}
Choosing $M=k+d+1$, we first obtain that for all $n\in \Z^d$, $\Bigl(p_\nu(2\pi nh)c_{n-\nu}\Bigl)_{\nu}$ is summable, which means that for all $n\in \Z^d$, the sum
\begin{equation*}
    \croch{\widehat{p}(2\pi h n)\ast c}\!(n):=\sum_{j\in \Z^d}p_j(2\pi nh)c_{n-j}
\end{equation*}
converges absolutely, and second that $\Bigl(\croch{\widehat{p}(2\pi h n)\ast c}\!(n)\Bigl)_n\in \mathrm{PS}_{k}\subset\mathrm{PS}$. In that case, if $u=\sum_{n\in \Z^d}c_n\delta_{2\pi n h}$ and $\varphi\in \mathscr{S}(\R^d)$, we have
\begin{equation}\label{Détermination, expression de l'opérateur - 1}
    \begin{split}
        \ps{\mathrm{Op}_h(p)u}{\varphi} &= \ps{u}{\mathrm{Op}_{h,0}(p)\varphi}=\sum_{n\in \Z^d}c_n\mathrm{Op}_{h,0}(p)\varphi(2\pi h n).
    \end{split}
\end{equation}
But a simple calculation shows that for all $x\in \R^d$, 
\begin{equation*}
    \mathrm{Op}_{h,0}(p)\varphi(x)=\sum_{\nu}(p_\nu\varphi)(x+2\pi \nu h)
\end{equation*}
so that \eqref{Détermination, expression de l'opérateur - 1} becomes
\begin{equation}
    \begin{split}
        \ps{\mathrm{Op}_h(p)u}{\varphi} &= \sum_{n,\nu\in \Z^d}c_n p_\nu(2\pi h (\nu+n))\varphi(2\pi h (\nu+n))\\
        &= \sum_{n,m\in \Z^d}c_n p_{m-n}(2\pi h m)\varphi(2\pi h m)\\
        &= \ps{\sum_{m\in \Z^d}\croch{\widehat{p}(2\pi h m)\ast c}\!(m)\delta_{2\pi h m}}{\varphi}.
    \end{split}
\end{equation}

Finally, with the previous identification $\mathscr{S}'_p(\R^d)\cong \mathrm{PS}$, we can write
\begin{equation}
    \mathrm{Op}_h(p):\left\{\begin{array}{ccccc}
         & \mathrm{PS} & \longrightarrow & \mathrm{PS} \\
         & c:=(c_n)_{n\in \Z^d} & \mapsto & \Bigl(\croch{\widehat{p}(2\pi h n)\ast c}\!(n)\Bigl)_{n\in \Z^d}
    \end{array}\right..
\label{Expression nouvelle procédure de quantif bis}
\end{equation}
In this case, we associate $p$ with an infinite matrix $M_\infty(p)$ defined for all $n,j\in \Z^d$ by
\begin{equation}
    M_\infty(p)_{n,j}=p_{n-j}(2\pi n h).
\label{Coef infinite matrix quantization}
\end{equation}
where $M_\infty$ acts on $\mathrm{PS}$ in a way that for all $c\in \mathrm{PS}$ and $n\in \Z^d$, 
\begin{equation*}
    \Bigl(M_\infty(p)c\Bigl)(n)=\sum_{j\in \Z^d}M_\infty(p)_{n,j}c_j=\croch{\widehat{p}(2\pi h n)\ast c}\!(n)=\Bigl[\mathrm{Op}_h(p)(c)\Bigl]\!(n).
\end{equation*}
In particular, $M_\infty:\mathrm{PS}\to \mathrm{PS}$, 
which is a way to attribute a matrix representation to $\mathrm{Op}_h(p)$.\\

\m\m\m\m Notice that for the expression \eqref{Expression nouvelle procédure de quantif bis} to be well defined, it only requires that $p$ satisfies \eqref{Décroissance des coef de Fourier par rapport à l'indice} and that $p(\cdot,\xi)$ can be evaluated. From this remark, the expression \eqref{Expression nouvelle procédure de quantif bis} can be extended to a more general symbol class, which is given by the two following definitions.

\begin{defn}\label{Définition de la classe de symboles avec la régularité la plus basse}
    We define $BS_\xi$ as the set of functions $p:\R^{2d}\to \C$ such that:
    \begin{enumerate}
        \item for all $x\in \R^d$, $p(x,\cdot)\in \mathscr{C}^\infty(\R^d)$ satisfies: for all $\gamma\in \N^d$, there exists $C_\gamma>0$ (independent of $h$ and $x$) such that
        \begin{equation*}
            \bb{\partial^\gamma_\xi p(x,\xi)}\leq C_\gamma,
        \end{equation*} 
        \item for all $x\in \R^d$, $p(x,\cdot)$ is $\Z^d$ periodic in the sense of \eqref{Def Z^d-périodicité en xi}.
    \end{enumerate}
\end{defn}

\begin{defn}
    Let $p\in BS_\xi$. Write $p$ with its partial Fourier transform 
\begin{equation*}
    p(x,\xi)=\sum_{\nu\in \Z^d}p_\nu(x)e^{2i\pi\ps{\nu}{\xi}}
\end{equation*}
where for all $\nu\in \Z^d$, $p_\nu(x)=\int_{\T^d}p(x,\xi)e^{-2i\pi\ps{\nu}{\xi}}d\xi$. We define $\mathrm{Op}_h(p)$ as the operator acting upon $\mathrm{PS}$ by \eqref{Expression nouvelle procédure de quantif bis} and associate it with the infinite matrix $M_\infty(p)$ given by \eqref{Coef infinite matrix quantization}.
\label{Nouvelle procédure de quantification}
\end{defn}

\begin{notns}\label{Notation matrice nouvelle procédure de quantif via la matrice infinie}
    Let $N\geq 1$ and $p\in BS_\xi$. We set $M_N(p):=\mathbb{1}_{\disc{1}{N}^d}M_\infty(p)\mathbb{1}_{\disc{1}{N}^d}$ the submatrix of $M_\infty(p)$ such that
\begin{equation}\label{Définition de la matrice avec la nouvelle procédure de quantif}
    \forall k,j\in {\disc{1}{N}^d}^d, \m\m\m M_N(p)_{k,j}=M_\infty(p)_{k,j}.
\end{equation}
In that case, we now see $M_N(p)$ as a matrix, that is $M_N(p):\C^{N^d}\to \C^{N^d}$.
\end{notns}

\begin{lem}\label{Norme bornée des matrices pour la nouvelle procédure de quantification}
    Let $p\in BS_\xi$. Then, there exists a constant $C$ such that for all $N\in \N^*$, 
    \begin{equation*}
        \norme{M_N(p)}\leq C.
    \end{equation*}
\end{lem}

\begin{proof}
    Let $u\in \C^{N^d}$. Let $\widetilde{u}\in \C^{\Z^d}$ be such that for all $j\in \Z^d$,
    \begin{equation*}
        \widetilde{u}_j=\left\{\begin{array}{ccccc}
             & u_j & \text{if} & j\in {\disc{1}{N}^d}^d\\
             & 0 & \text{otherwise}
        \end{array}\right..
    \end{equation*}
    Then, we have
    \begin{equation*}
        \begin{split}
            \norme{M_N(p)u} &= \pr{\sum_{n\in \disc{1}{N}^d}\bb{\sum_{k\in \disc{1}{N}^d}p_{n-k}\pr{\frac{n}{N}}u_k}^2}^{1/2}\\
            &\leq \pr{\sum_{n\in \disc{1}{N}^d}\pr{\sum_{k\in \disc{1}{N}^d}\bb{p_{n-k}\pr{\frac{n}{N}}}\bb{u_k}}^2}^{1/2}\\
            &\leq\pr{\sum_{n\in \Z^d}\pr{\sum_{k\in \Z^d}\bb{p_{k}\pr{\frac{n}{N}}}\bb{\widetilde{u}_{n-k}}}^2}^{1/2}\\
            &\leq \sum_{k\in\Z^d}\pr{\sum_{n\in\Z^d}\bb{p_{k}\pr{\frac{n}{N}}}^2\bb{\widetilde{u}_{n-k}}^2}^{1/2},
        \end{split}
    \end{equation*}
    where the third inequality comes from the Minkowski's inequality. But we also know from \eqref{Décroissance des coef de Fourier par rapport à l'indice} that for all $K\in \N^*$, there exists $C_K>0$ such that, for all $\nu\in\Z^d$, $x\in  \R^d$,
    \begin{equation*}
        \bb{p_\nu(x)}\leq C_K\!\jap{\nu}^{-K}.
    \end{equation*}
    Therefore, for $K\geq d+1$,
    \begin{equation*}
        \begin{split}
            \norme{M_N(p)u} &\leq C_K \sum_{k\in\Z^d}\jap{k}^{-K}\pr{\sum_{n\in\Z^d}\bb{\widetilde{u}_{n-k}}^2}^{1/2}=\widetilde{C}_K\pr{\sum_{n\in \Z^d}\bb{\widetilde{u}_n}^2}^{1/2}=\widetilde{C}_K\norme{u},
        \end{split}
    \end{equation*}
    which gives the result since $\widetilde{C}_K$ does not depend on $N$.
\end{proof}

\subsection{Quantization of symbols on a torus}
This section is highly inspired by the work of Vogel in \cite{Vogel_2020} who considered the case $\alpha_1=\alpha_2=\alpha$, and \cite{Christiansen_2010} and \cite{Nonnenmacher_2006} studied the case $\alpha_1=\alpha_2=1$ in the following definition.
\begin{defn}\label{Definition de l'espace des disctri tempérées périodiques}
    For $\alpha_1$ and $\alpha_2$ such that $h\leq \alpha_1,\alpha_2\leq 1$, we denote $\Tilde{h}:=\frac{h}{\sqrt{\alpha_1\alpha_2}}$. We define $\mc{H}^d_{\Tilde{h},\alpha_1,\alpha_2}$ as the set of $u\in \mathscr{S}'(\R^d)$ such that for all $n\in \Z^d$, 
    \begin{equation*}
        u(x+\alpha_1^{-1/2}n)=u(x)\m\m\m\m\m\text{and} \m\m\m\m\m\mc{F}_{\Tilde{h}}(u)(\xi+\alpha_2^{-1/2}n)=\mc{F}_{\Tilde{h}}(u)(\xi) .
    \end{equation*}
\end{defn}
In particular, it is clear that $\Hd$ is a complex vector space.
\begin{lem}
    With the same notations as in Definition \ref{Definition de l'espace des disctri tempérées périodiques}, we have :
    \begin{equation*}
        \Hd\neq \{0\}\Longleftrightarrow \exists N\in \N^*, \m h=\frac{1}{2\pi N}\m .
    \end{equation*}
    In that case, $\Hd$ is finite dimensional with $\mathrm{dim}\pr{\Hd}=N^d$.
    Furthermore, $\Hd=\mathrm{Span}\left\{Q_k^{\A{1}}, \m k\in \disc{0}{N-1}^d\right\}$ where 
    \begin{equation}
        Q_k^{\A{1}}=\pr{\A{1}^{1
        /2}N}^{-d/2}\sum_{n\in \Z^d}\delta_{\A{1}^{-1/2}(n+k/N)}.
    \label{Basis vectors}
    \end{equation}
    and we will denote 
    \begin{equation}\label{Def de la base de l'espace d'intérêt}
        \mc{B}^{\alpha_1}=\pr{Q_k^{\A{1}}}_{k\in \disc{0}{N-1}^d}
    \end{equation}
    the basis given by the elements of \eqref{Basis vectors}.
    \label{Lemme dimension et base de l'espace de travail}
\end{lem}

\begin{proof}
    1. Let us take $\phi\in \mathscr{C}_c^{\infty}(\R^d;\R)$ such that $\sum_{g\in \Z^d}\tau_{\Ab{1}g}\phi=1$ (see \cite[Theorem 1.4.6]{Hörmander_1}) and where $\tau_{\Ab{1}g}\phi$ is given in \eqref{Translation par gamma}. We fix $u\in \Hd\setminus\{0\}$. Then, by duality, for all $\psi\in \mathscr{S}(\R^d)$,
    \begin{equation*}
        \begin{split}
            \ps{\F{u}}{\psi} &= \ps{u}{1\times\F{\psi}}=\ps{u}{\sum_{g\in \Z^d}\tau_{\Ab{1}q}\phi\F{\psi}}\\
            &=\sum_{g\in \Z^d}\ps{u}{\tau_{\Ab{1}g}\phi\F{\psi}}= \sum_{g\in \Z^d}\ps{u}{\phi\tau_{-\Ab{1}g}\F{\psi}}\\
            &=\ps{\phi u}{\sum_{g\in \Z^d}\tau_{-\Ab{1}g}\F{\psi}}
        \end{split}
    \end{equation*}
    where the fourth equality comes from the $\Ab{1}\Z^d$ periodicity of $u$. Moreover, using the Poisson's summation formula (see \cite[Section 7.2]{Hörmander_1}), we have for all $\xi\in \R^d$,
    \begin{equation*}
        \begin{split}
            \sum_{g\in \Z^d}\tau_{-\Ab{1}g}\F{\psi}(\xi) &= \sum_{g\in \Z^d}\tau_{-\xi}\F{\psi}(\Ab{1}g)=\sum_{g\in \Z^d}\F{e^{-\frac{i}{\Tilde{h}}\ps{\xi}{\cdot}}\psi}(\Ab{1}g)\\
            &= \pr{2\pi \Tilde{h}\A{1}^{1/2}}^{d}\sum_{g\in \Z^d}e^{-2i\pi \A{1}^{1/2}\ps{\xi}{g}}\psi\pr{2\pi \Tilde{h}\A{1}^{1/2}g}.
        \end{split}
    \end{equation*}
    Finally, we obtain
    \begin{equation*}
        \ps{\F{u}}{\psi}=\pr{2\pi \Tilde{h}\A{1}^{1/2}}^{d}\sum_{g\in \Z^d}\ps{u}{e^{-2i\pi \A{1}^{1/2}\ps{g}{\cdot}}\phi}\psi\pr{2\pi \Tilde{h}\A{1}^{1/2}g}.
    \end{equation*}
    Denoting
    \begin{equation}\label{Expression des c_g(u)}
        c_g(u):=\ps{u}{e^{-2i\pi \A{1}^{1/2}\ps{g}{\cdot}}\phi},
    \end{equation}
    we have in the space $\mathscr{S}'(\R^d)$
    \begin{equation} \label{Expression_of_TF_in_tempered_distributions}
        \F{u}=\pr{2\pi \Tilde{h}\A{1}^{1/2}}^{d}\sum_{g\in \Z^d}c_g(u)\delta_{2\pi \Tilde{h}\A{1}^{1/2}g}.
    \end{equation}
    Furthermore, for all $y\in \Z^d$, by properties of the convolution of distributions, we have $\delta_{\Ab{2}y}\ast \F{u}=\tau_{-\Ab{2}y}\F{u}$ in $\mathscr{S}'(\R^d)$. Since $\F{u}$ is $\Ab{2}\Z^d$ periodic, we obtain that 
    \begin{equation}\label{Invariance par translation de la TF}
        \delta_{\Ab{2}y}\ast \F{u}=\F{u}
    \end{equation}
    in $\mathscr{S}'(\R^d)$. Using the equality \eqref{Expression_of_TF_in_tempered_distributions}, \eqref{Invariance par translation de la TF} provides that for all $y\in \Z^d$,
    \begin{equation*}
        \sum_{g\in \Z^d}c_g(u)\delta_{2\pi \Tilde{h}\A{1}^{1/2}g+\Ab{2}y}=\sum_{g\in \Z^d}c_g(u)\delta_{2\pi \Tilde{h}\A{1}^{1/2}g}.
    \end{equation*}
    These two distributions being equal implies the equality of their supports. But this is equivalent to saying that there exists $N\in \N^*$ such that $\Ab{2}=2\pi\Tilde{h}\A{1}^{1/2}N$, which is equivalent to $h=(2\pi N)^{-1}$ since $\Tilde{h}=h/\sqrt{\alpha_1\alpha_2}$.\\

    \m\m\m\m Conversely, if there exists $N\in \N^*$ such that $h=(2\pi N)^{-1}$, then $\sum_{g\in \Z^d}\delta_{\Ab{1}g}\in \Hd$. Indeed, $\sum_{g\in \Z^d}\delta_{\Ab{1}g}$ is $\Ab{1}\Z^d$ periodic and, as a consequence of the Poisson's summation formula, in $\mathscr{S}'(\R^d)$, 
    \begin{equation*}
        \F{\sum_{g\in \Z^d}\delta_{\Ab{1}g}}= \pr{2\pi\Tilde{h}\A{1}^{1/2}}^{d}\sum_{g\in \Z^d}\delta_{\Ab{2}N^{-1}g}.
    \end{equation*}
    This guarantees that $\F{\sum_{g\in \Z^d}\delta_{\Ab{1}g}}$ is $\Ab{2}\Z^d$ periodic and thus $\Hd\neq \{0\}$.\\

    \mm 2. Now, we focus on the second part of the proof. We assume that 
    \begin{equation}\label{Hypothèse expression de h comme l'inverse de N}
        h=(2\pi N)^{-1}
    \end{equation}
    for some $N\in \N^*$. Let $u\in \Hd$. Given the $\Ab{2}\Z^d$ periodicity of $\F{u}$, for all $\eta\in \Z^d$, $u=e^{-\frac{i}{\Tilde{h}}\Ab{2}\ps{\cdot}{\eta}}u$ in $\mathscr{S}'(\R^d)$. In view of \eqref{Expression des c_g(u)}, we therefore have for all $m$, $\ell\in \Z^d$, $c_{m+N\ell}(u)=c_m(u)$. Thus, expression \eqref{Expression_of_TF_in_tempered_distributions} can be rewritten as
    \begin{equation}\label{Expression_of_TF_in_tempered_distributions_v2}
        \begin{split}
            \F{u} &= \pr{2\pi \Tilde{h}\A{1}^{1/2}}^{d}\sum_{j\in \disc{0}{N-1}^d}c_j(u)\sum_{n\in \Z^d}\delta_{2\pi\Tilde{h}\A{1}^{1/2}(j+Nn)}\\
            &= \pr{2\pi \Tilde{h}\A{1}^{1/2}}^{d}\sum_{j\in \disc{0}{N-1}^d}c_j(u)\pr{\delta_{N^{-1}\Ab{2}j}\ast\sum_{n\in \Z^d}\delta_{\Ab{2}n}}.
        \end{split}
    \end{equation}
    Recall that $\F{\F{u}}=\pr{2\pi \Tilde{h}}^{d} \Tilde{u}$ where $\Tilde{u}(x)=u(-x)$ and that if $S$ is a compactly supported distribution and $T$ a tempered one, then $\F{S*T}=\F{S}\F{T}$. In that case, with \cite[Theorem 7.2.1]{Hörmander_1}, composing the last equality of \eqref{Expression_of_TF_in_tempered_distributions_v2} by $\mc{F}_{\Tilde{h}}$, we obtain
    \begin{equation*}
        \begin{split}
            \pr{2\pi \Tilde{h}}^{d}\Tilde{u} &= \pr{2\pi \Tilde{h}\A{1}^{1/2}}^{d}\sum_{j\in \disc{0}{N-1}^d}c_j(u)\F{\delta_{N^{-1}\Ab{2}j}}\pr{2\pi \Tilde{h}\A{2}^{1/2}}^{d}\sum_{n\in \Z^d}\delta_{N^{-1}\Ab{1}n}.\\
        \end{split}    
    \end{equation*}
    But for all $j\in \Z^d$, $\F{\delta_{N^{-1}\Ab{2}j}}=e^{-\frac{i}{\Tilde{h}}\ps{\cdot}{N^{-1}\Ab{2}j}}$. Recalling \eqref{Hypothèse expression de h comme l'inverse de N} and that $\Tilde{h}=\frac{h}{\sqrt{\A{1}\A{2}}}$, we have
    \begin{equation*}
        \begin{split}
            \Tilde{u} &= N^{-d}\sum_{j\in \disc{0}{N-1}^d}c_j(u)e^{-2i\pi \A{1}^{1/2}\ps{\cdot}{j}}\sum_{n\in \Z^d}\delta_{N^{-1}\Ab{1}n}\\
            &= N^{-d}\sum_{j\in \disc{0}{N-1}^d}c_j(u)\sum_{n\in \Z^d}e^{-\frac{2i\pi}{N}\ps{n}{j}}\delta_{N^{-1}\Ab{1}n}\\
            &= N^{-d}\sum_{j\in \disc{0}{N-1}^d}c_j(u)\sum_{s\in \disc{0}{N-1}^d}\sum_{r\in \Z^d}e^{-\frac{2i\pi}{N}\ps{s+Nr}{j}}\delta_{N^{-1}\Ab{1}(s+Nr))} \\
            &= \sum_{s\in \disc{0}{N-1}^d}\pr{\alpha_1^{d/4}\sum_{j\in \disc{0}{N-1}^d}c_j(u)e^{-\frac{2i\pi}{N}\ps{s}{j}}}Q^{\alpha_1}_s.
        \end{split}
    \end{equation*}
    This ensures that $\Tilde{u}\in \mathrm{Span}\pr{\left\{Q_k^{\alpha_1}, \m k\in \disc{0}{N-1}^d\right\}}$, then so does $u$ and the result follows.
\end{proof}

\mm Until further notice, we fix 
\begin{equation}\label{Hypothèse expression de h comme l'inverse de N - 2}
    h=(2\pi N)^{-1}    
\end{equation}
for some $N\in \N^*$.

\begin{notns}\label{Notation pour le tore biscornu}
    For $0<\alpha_1,\alpha_2\leq 1$ and $\T_{\alpha_j}:=\bigslant{\R}{\alpha_j^{-1/2}\Z}$, $j=1,2$, we denote 
\begin{equation*}
    \mc{T}_{\A{1},\A{2}}:=\T^d_{\A{1}}\times\T^d_{\A{2}}.
\end{equation*}
\end{notns}

\begin{defn}
    We say that a continuous function $m:\mc{T}_{\A{1},\A{2}}\to (0,+\infty)$ is an order function on $\mc{T}_{\A{1},\A{2}}$ if there exist $C_0>0$ and $N_0\geq 0$ such that for all $\rho$, $\mu\in \mc{T}_{\A{1},\A{2}}$
    \begin{equation*}
        m(\rho)\leq C_0\jap{\rho-\mu}_{\mc{T}_{\A{1},\A{2}}}^{N_0}m(\mu)
    \end{equation*}
    with $\jap{\rho}_{\mc{T}_{\A{1},\A{2}}}=\pr{1+\bb{\rho}_{\mc{T}_{\A{1},\A{2}}}}^{1/2}$, where for $\rho=(x,\xi)\in \mc{T}_{\A{1},\A{2}}$,
    \begin{equation}
        \bb{\rho}_{\mc{T}_{\A{1},\A{2}}}:=\inf\limits_{{\gamma_1,\gamma_2\in \Z^d}}\pr{\bb{x-\Ab{1}\gamma_1}^2+\bb{\xi-\Ab{2}\gamma_2}^2}^{\frac{1}{2}}.
    \label{Norme weird sur le tore biscornu}
    \end{equation}
\label{Def order function sur le tore biscornu}
\end{defn}
Using the natural projection $\R^{2d}\to \mc{T}_{\A{1},\A{2}}$, $m$ can be seen as a periodic function on $\Rd$ so that it becomes an order function in the sense of \eqref{Recall d'analyse semicl - def fonction d'ordre}.\\

\mm As mentioned in Section \ref{Outline}, in the next section, a symbol $p\in \mathscr{C}^{0,\varrho}_{\mathrm{pw}}S_d$ will be extended by periodicity, and regularized with respect to $x$ thanks to an $h$-dependent mollifier. But the singularities of jump type for the extended symbol will result in an $h$-degeneracy of the regularized one. To treat this new symbol, we introduce the following symbol class.

\begin{defn}
For $0<\A{1},\A{2}\leq 1$ such that $h\ll  \pr{\A{1}\A{2}}^{1/2}$, and $m$ an order function on $\mc{T}_{\A{1},\A{2}}$, $\eta_1,\eta_2\in [0,1/4]$, we set $\Tilde{h}=h/\sqrt{\alpha_1\alpha_2}$ and define
\begin{enumerate}\label{Definition des espaces de travail dans le cas rescaling}
    \item $S(m,\alpha_1,\A{2})$ as the set of $p\in \mathscr{C}^{\infty}(\mc{T}_{\A{1},\A{2}})$ such that for all $\beta,\gamma\in \N^d$, there exists $C_{\beta,\gamma}>0$ (uniform with respect to $\Tilde{h}$) satisfying: for all $(x,\xi)\in \mc{T}_{\A{1},\A{2}}$,
    \begin{equation*}
        \bb{\partial^\beta_x\partial^\gamma_\xi p(x,\xi)}\leq C_{\beta,\gamma}m(x,\xi),
    \end{equation*}
    \item $S_{\eta_1,\eta_2}(m,\alpha_1,\A{2})$ as the set of $p\in \mathscr{C}^{\infty}(\mc{T}_{\A{1},\A{2}})$ such that for all $\beta,\gamma\in \N^d$, there exists $C_{\beta,\gamma}>0$ (uniform with respect to $\Tilde{h}$) satisfying: for all $(x,\xi)\in \mc{T}_{\A{1},\A{2}}$,
    \begin{equation*}
        \bb{\partial^\beta_x\partial^\gamma_\xi p(x,\xi)}\leq C_{\beta,\gamma}\Tilde{h}^{-\bb{\beta}\eta_1-\bb{\gamma}\eta_2}m(x,\xi).
    \end{equation*}
\end{enumerate}
\end{defn}
Whenever convenient, tanks to the projection $\R^{2d}\to \mc{T}_{\A{1},\A{2}}$, we will use that $S_{\eta_1,\eta_2}(m,\alpha_1,\A{2})\subset S_{\eta_1,\eta_2}(m)$ (see \eqref{Recall d'analyse semicl - classe de symbole}).
\mm\\
\mm When $p\in S_{\eta_1,\eta_2}(m_1,\alpha_1,\A{2})$ and $q\in S_{\delta_1,\delta_2}(m_2, \A{1},\A{2})$, we directly obtain from \eqref{Recall d'analyse semicl - expression produit de Moyal} that $p\sharp_{\Tilde{h}} q$ gets the same periodicity as $p$ and $q$. In particular, in view of \eqref{Recall d'analyse semicl - formule de composition}, we have
\begin{equation}
    \mathrm{Op}_{\Tilde{h}}^w(p)\circ\mathrm{Op}_{\Tilde{h}}^w(q)=\mathrm{Op}_{\Tilde{h}}^w(p\sharp_{\Tilde{h}} q)\mm\text{where}\mm p\sharp_{\Tilde{h}} q\in S_{\omega_1,\omega_2}(m_1m_2,\A{1},\A{2})
\label{Expression compo symboles périodiques}
\end{equation}
and $\omega_j=\max(\eta_j,\delta_j)$, $j\in \{1,2\}$. In the same way as for \eqref{Recall d'analyse semicl - devpt asymptotique}, for $p\in S_{\eta_1,\eta_2}(m,\A{1},\A{2})$, we write
\begin{equation}
    p\sim p_0+{\Tilde{h}}p_1+\dots \mm\text{in}\mm S_{\eta_1,\eta_2}(m,\A{1},\A{2}), \mm p_\nu\in S_{\eta_1,\eta_2}(m,\A{1},\A{2})
\label{Recall d'analyse semicl - devpt asymptotique périodique}
\end{equation}
if for all $N\in \N^*$, $p-\sum_{\nu=0}^{N-1}{\Tilde{h}}^\nu p_\nu\in {\Tilde{h}}^NS_{\eta_1,\eta_2}(m,\A{1},\A{2})$.\\
\m\\
\begin{prop}
    If $p\in S_{\eta_1,\eta_2}(m,\alpha_1,\A{2})$, then for all $t\in [0,1]$, 
    \begin{equation*}
        \mathrm{Op}_{\Tilde{h},t}(p):\Hd\longrightarrow \Hd.
    \end{equation*}
\end{prop}
\begin{proof}
    For $\gamma,\mu\in \R^d$, we denote the operator
    \begin{equation}\label{Def de l'opérateur M_gamma,mu}
        M_{\gamma,\mu}:=\tau_{\gamma}e^{\frac{i}{\Tilde{h}}\ps{.}{\mu}}    
    \end{equation}
    with $\tau_\gamma$ defined by \eqref{Translation par gamma}. Noticing that 
    \begin{equation}\label{Conjugaison par M_gamma,mu}
        M_{\gamma,\mu}\mathrm{Op}_{\Tilde{h},t}(p)M_{\gamma,\mu}^{-1}=\mathrm{Op}_{\Tilde{h},t}\pr{\tau_{(\gamma,-\mu)}p},
    \end{equation}
    if $p\in S_{\eta_1,\eta_2}(m,\A{1},\A{2})$, $\gamma\in \Ab{1}\Z^d$ and $\mu\in \Ab{2}\Z^d$, we get 
    \begin{equation}
        M_{\gamma,\mu}\mathrm{Op}_{\Tilde{h},t}(p)M_{\gamma,\mu}^{-1}=\mathrm{Op}_{\Tilde{h},t}\pr{p}.
    \label{Conjugation by M_gamma_mu}
    \end{equation}
    The result follows.
\end{proof}

\begin{notns}
    Until the end, for $p\in S_{\eta_1,\eta_2}(m,\A{1},\A{2})$ and $t\in [0,1]$, we will denote 
\begin{equation}
    p_{N,\A{1},\A{2}}^t:=\mathrm{Op}_{\Tilde{h},t}(p)_{\left|\Hd\right.}. 
\label{Definition de P_N_alpha_1_alpha_2}
\end{equation}
In the case $t=1/2$, we drop the "$1/2$" to shorten the notations. We will always identify $p_{N,\A{1},\A{2}}^t$ as its matrix in the basis $\mc{B}^{\A{1}}$ defined by \eqref{Def de la base de l'espace d'intérêt}. In the case $\alpha_1=\alpha_2=1$, we will simply write $p_{N}^t$ instead of $p_{N,1,1}^t$, and so $p_N$ instead of $p_{N,1,1}^{1/2}$.
\label{Notations pour les matrices des pseudo}
\end{notns}

\begin{lem}\label{Coefficients of the Toeplitz matrix}
    Let $f\in S_{\eta_1,\eta_2}(m,\A{1},\A{2})$ and let $t\in [0,1]$. Then, the matrix coefficients of $f_{N,\A{1},\A{2}}^t$ in the basis $\mc{B}^{\alpha_1}$ given by \eqref{Def de la base de l'espace d'intérêt} are
    \begin{equation}
        F^t_{s,j}=\sum_{n,r\in \Z^d}c_{n,j-s+rN}(f)e^{\frac{2i\pi}{N}(1-t)\ps{n}{j}+\frac{2i\pi}{N}t\ps{n}{s-rN}},\m\m\m\m  s,j\in \disc{1}{N}^d,
    \end{equation}
    where for all $n,k\in \Z^d$,
    \begin{equation}\label{Coeffs de Fourier en les deux variables}
        c_{n,k}(f)=(\A{1}\A{2})^d\int_{\mc{T}_{\A{1},\A{2}}}f(x,\xi)e^{-2i\pi(\alpha_1^{1/2}\ps{x}{n}+\alpha_2^{1/2}\ps{\xi}{k})}dxd\xi.
    \end{equation}
\end{lem}

\m\m\m\m Notice that the coefficients $F^t_{s,j}$ depend on $\alpha_1$ and $\A{2}$ because of the Fourier coefficients, although not explicitly denoted here.

\begin{proof}
    Write $\theta_j=2\pi \A{j}^{1/2}$, $j\in \{1,2\}$ and for $n$, $m\in \Z^d$, define $L_{n,m}(x,\xi)=\theta_1\!\ps{x}{n}+\theta_2\!\ps{\xi}{m}$. The periodicity of $f$ allows us to write
    \begin{equation}
        \begin{split}
           f(x,\xi)=&\sum_{n,m\in \Z^d}c_{n,m}(f)e^{iL_{n,m}(x,\xi)} \\
           \text{where}\m\m\m\m c_{n,m}(f)=(\A{1}\A{2})^d&\int_{\mc{T}_{\A{1},\A{2}}}f(x,\xi)e^{-iL_{n,m}(x,\xi)}dxd\xi.
        \end{split}
    \label{devpt en série de Fourier}
    \end{equation}
    In particular,
    \begin{equation*}
        \mathrm{Op}_{\Tilde{h},t}(f)=\sum_{n,m\in \Z^d}c_{n,m}(f)\mathrm{Op}_{\Tilde{h},t}\pr{e^{iL_{n,m}}}.
    \end{equation*}
    But for all $u\in \mathscr{S}(\R^d)$, from \cite[Chapter 7]{dimassi1999spectral}, we have
    \begin{equation*}
        \mathrm{Op}_{\Tilde{h},t}(\exp\pr{iL_{n,m}})u(x)=\exp(i\theta_1 t \ps{n}{x})\tau_{-\Tilde{h}\theta
        _2 m}\Bigl[\exp(i(1-t)\theta_1\ps{n}{\cdot})u(\cdot)\Bigl](x),
    \end{equation*}
    so by duality, we directly have for all $j\in \disc{0}{N-1}^d$,
    \begin{equation*}
        \begin{split}
            \mathrm{Op}&_{\Tilde{h},t}(\exp(iL_{n,m}))Q_j^{\alpha_1}\\
            &=\frac{1}{(\alpha_1^{1/2}N)^{d/2}}\sum_{k\in \Z^d}e^{it\theta_1\ps{n}{\cdot}}\tau_{-\Tilde{h}\theta_2 m}\croch{\exp(i(1-t)\theta_1\ps{n}{\cdot})\delta_{\alpha_1^{-1/2}(k+j/N)}}.
        \end{split}
    \end{equation*}

    For all $u\in \mathscr{S}(\R^d)$,
    \begin{equation*}
        \begin{split}
            &\left\langle e^{it\theta_1\ps{n}{\cdot}}\right.  \left.\tau_{-\Tilde{h}\theta_2 m} 
            \croch{\exp(i(1-t)\theta_1\ps{n}{\cdot})\delta_{\alpha_1^{-1/2}(k+j/N)}},u\right\rangle\\
            &= \ps{\delta_{\Ab{1}(k+j/N)}}{\exp(i(1-t)\theta_1\ps{n}{\cdot})\tau_{\Tilde{h}\theta_2 m}\croch{e^{it\theta_1\ps{n}{\cdot}}u(\cdot)}}\\
            &= \exp\biggl(i(1-t)\theta_1\Ab{1}\ps{n}{k+\frac{j}{N}}\\
            &\mm\mm+it\theta_1\ps{n}{\Ab{1}\pr{k+\frac{j}{N}}-\Tilde{h}\theta_2 m}\biggl)u\pr{\Ab{1}\pr{k+\frac{j}{N}}-\Tilde{h}\theta_2 m}\\
            &= \ps{\exp\pr{\frac{2i\pi}{N}\ps{n}{j}-\frac{2i\pi}{N}t\ps{n}{m}}\delta_{\Ab{1}(k+\frac{j-m}{N})}}{u}.
        \end{split}
    \end{equation*}
    We can deduce from this that,
    \begin{equation*}
        \mathrm{Op}_{\Tilde{h},t}(\exp(iL_{n,m}))Q_j^{\alpha_1}=\exp\pr{\frac{2i\pi}{N}\ps{n}{j}-\frac{2i\pi}{N}t\ps{n}{m}}Q_{j-m}^{\A{1}}
    \end{equation*}
    where the index is meant modulo $N\Z^d$. Thus, for all $n$, $m\in \Z^d$, $j\in \disc{0}{N-1}^d$,
    \begin{equation*}
        \begin{split}
            \mathrm{Op}&_{\Tilde{h},t}(f)Q_j^{\A{1}}= \sum_{n,m\in \Z^d}c_{n,m}(f)\exp\pr{\frac{2i\pi}{N}\ps{n}{j}-\frac{2i\pi}{N}t\ps{n}{m}}Q_{j-m}^{\A{1}}\\
            &=\sum_{s\in \disc{0}{N-1}^d} \pr{\sum_{n,r\in \Z^d}c_{n,j-s+rN}(f)\exp\pr{\frac{2i\pi}{N}(1-t)\ps{n}{j}+\frac{2i\pi}{N}t\ps{n}{s-rN}}}Q_s^{\A{1}},
        \end{split}
    \end{equation*}
    which gives the result.
\end{proof}

\begin{rem}\label{Expression des coefficients des matrices ancienne quantif, pour les valeurs de t qui nous intéressent}
    The two cases we are going to focus on are:
    \begin{itemize}
        \item $t=1/2$:
    \begin{equation}\label{Coeffs matriciels - cas t=1/2}
        F^{1/2}_{s,j}=\sum_{n,r\in \Z^d}c_{n,j-s+rN}(f)\exp\pr{\frac{i\pi}{N}\ps{n}{j+s}}(-1)^{\ps{n}{r}},
    \end{equation}
    \item $t=1$:
    \begin{equation}\label{Coeffs matriciels - cas t=1}
        F^{1}_{s,j}=\sum_{n,r\in \Z^d}c_{n,j-s+rN}(f)\exp\pr{\frac{2i\pi}{N}\ps{n}{s}}.
    \end{equation}
    \end{itemize}
\end{rem}

\begin{lem}\label{Comportement des coef de Fourier}
    Let $f\in \mathscr{C}^{\infty}(\mc{T}_{\A{1},\A{2}})$ and recall that $c_{n,m}(f)$ is defined by \eqref{Coeffs de Fourier en les deux variables}. Then, for all $k\in \N$, $n,m\in \Z^d$,
    \begin{equation}\label{Estimation des coeffs de Fourier en les deux variables - 1}
        c_{n,m}(f)=\mc{O}_k(1)(\A{1}\A{2})^{d/2}\jap{(n,m)}^{-k}\sum_{\bb{\beta}+\bb{\gamma}\leq k}\A{1}^{-\bb{\beta}/2}\A{2}^{-\bb{\gamma}/2}\norme{\partial_x^\beta\partial_\xi^\gamma f}_{L^1(\mc{T}_{\A{1},\A{2}})}.
    \end{equation}
    In particular, if $0<\A{1}\leq\A{2}\leq 1$, then \eqref{Estimation des coeffs de Fourier en les deux variables - 1} becomes
    \begin{equation}\label{Estimation des coeffs de Fourier en les deux variables - 2}
        c_{n,m}(f)=\mc{O}_k(1)\A{1}^{(d-k)/2}\A{2}^{d/2}\jap{(n,m)}^{-k}\sum_{\bb{\beta}+\bb{\gamma}\leq k}\norme{\partial_x^\beta\partial_\xi^\gamma f}_{L^1(\mc{T}_{\A{1},\A{2}})}.
    \end{equation}
\label{Décroissance poly des coef de Fourier sur le tore biscornu}
\end{lem}

\begin{proof}
    We still denote $\theta_j=2\pi \A{j}^{1/2}$, $j\in \{1,2\}$. Let $\mc{L}$ be the operator such that its formal adjoint (in the sense of the inner product $\ps{\cdot}{\cdot}_{L^2(\mc{T}_{\A{1},\A{2}})}$) is 
    \begin{equation*}
        \mc{L}^*=\frac{1+\theta_1^{-1}\!\ps{n}{D_x}+\theta_2^{-1}\!\ps{m}{D_\xi}}{1+\bb{n}^2+\bb{m}^2}.
    \end{equation*}
    Since
    \begin{equation*}
        \mc{L}\pr{e^{-2i\pi \left[\A{1}^{1/2}\ps{x}{n}+\A{2}^{1/2}\ps{\xi}{m}\right]}}=e^{-2i\pi \left[\A{1}^{1/2}\ps{x}{n}+\A{2}^{1/2}\ps{\xi}{m}\right]},
    \end{equation*}
    integration by parts shows that for all $M\in \N$,
    \begin{equation*}
        c_{n,m}(f)=(\A{1}\A{2})^{d/2}\int_{\mc{T}_{\A{1},\A{2}}} e^{-2i\pi \left[\A{1}^{1/2}\ps{x}{n}+\A{2}^{1/2}\ps{\xi}{m}\right]} (\mc{L}^*)^M(f)(x,\xi)dxd\xi.
    \end{equation*}
    However, we can show by induction that for all $M\in \N$,
    \begin{equation*}
        (\mc{L}^*)^M(f)(x,\xi)=\sum_{j=0}^M\binom{M}{j}\sum_{\bb{\beta}+\bb{\gamma}=j}\theta_1^{-\bb{\beta}}\theta_2^{-\bb{\gamma}}\m\frac{n^\beta m^\gamma D_x^\beta D_\xi^\gamma f(x,\xi)}{\jap{(n,m)}^{2M}}.
    \end{equation*}
    and since for all $\beta$, $\gamma\in \N^d$ satisfying $\bb{\beta}+\bb{\gamma}\leq M$, $\bb{n^\beta m^\gamma}\leq \jap{(n,m)}^M$, the result follows immediately by triangular inequality.
\end{proof}
Recall $f_{N,\A{1},\A{2}}:=f_{N,\A{1},\A{2}}^{1/2}$ given in Notations \ref{Notations pour les matrices des pseudo} and the basis $\mc{B}^{\alpha_1}$ given in \eqref{Def de la base de l'espace d'intérêt}.
\begin{lem}
    There exists a unique Hilbert space structure on $\Hd$ such that for all real-valued $f\in \mathscr{C}^{\infty}(\mc{T}_{\A{1},\A{2}})$, $f_{N,\A{1},\A{2}}:\Hd\to \Hd$ is selfadjoint and $\mc{B}^{\A{1}}$ is an orthonormal basis.
\label{Hilbert structure}
\end{lem}

\begin{proof}
    Let $\ps{\cdot}{\cdot}_0$ denote the inner product on $\Hd$ for which the basis $\mc{B}^{\A{1}}$ is orthonormal. Let $f\in \mathscr{C}^{\infty}(\mc{T}_{\A{1},\A{2}};\R)$. We are going to show that $f_{N,\A{1},\A{2}}$ is selfadjoint. Given the expression \eqref{Coeffs de Fourier en les deux variables} of $c_{n,m}(f)$ and the fact that $f$ is real-valued, we get $\overline{c_{n,m}(f)}=c_{-n,-m}(f)$. Using \eqref{Coeffs matriciels - cas t=1/2}, we easily obtain that for all $s,j\in \disc{1}{N}^d$, $\overline{F^{1/2}_{s,j}}=F^{1/2}_{j,s}$. 
\end{proof}

\begin{rem}
    From now on, we will always equip $\Hd$ with the inner product $\ps{\cdot}{\cdot}_0$, the induced norm $\norme{\cdot}_0$, and we drop the subscript $0$. Furthermore, using the basis $\mc{B}^{\alpha_1}$ given by \eqref{Def de la base de l'espace d'intérêt}, we will identify 
    \begin{equation}\label{Identification de l'espace d'intérêt avec C^N^d}
        \Hd\cong \ell^2\pr{\disc{1}{N}^d}\cong \C^{N^d}.
    \end{equation}
\end{rem}

\begin{prop}\label{Trace result}
    Let $0<\A{1},\A{2}\leq 1$ be such that $N^{-1}\ll  \pr{\A{1}\A{2}}^{1/2}$ and let $f\in S(m,\A{1},\A{2})$. Then,  
    \begin{equation*}
        \mathrm{tr}\pr{f_{N,\A{1},\A{2}}}=\pr{N\A{1}^{1/2}\A{2}^{1/2}}^d\int_{\mc{T}_{\A{1},\A{2}}}f(x,\xi)dxd\xi+r_N
    \end{equation*}
    where, for all $k\in \N$, 
    \begin{equation*}
        r_N=\mc{O}_k(1)N^{d-k}\A{1}^{(d-k)/2}\A{2}^{d/2}\sum_{\bb{\beta}+\bb{\gamma}\leq \max(2d+1,k)}\norme{\partial_x^\beta\partial_\xi^\gamma f}_{L^1(\mc{T}_{\A{1},\A{2}})}.
    \end{equation*}
\end{prop}

\begin{proof}
    From \eqref{Coeffs matriciels - cas t=1/2}, we get
    \begin{equation}\label{Trace result - 1}
        \begin{split}
            \mathrm{tr}&(f_{N,\A{1},\A{2}}) = \sum_{m\in\disc{1}{N}^d}\sum_{n,r\in \Z^d}c_{n,Nr}(f)e^{\frac{2i\pi}{N}\ps{m}{n}}(-1)^{\ps{n}{r}}\\
            &= \sum_{s\in\disc{0}{N-1}^d}\sum_{m\in \disc{1}{N}^d}\sum_{q,r\in \Z^d}c_{s+Nq,Nr}e^{\frac{2i\pi}{N}\ps{m}{s}}(-1)^{\ps{s+qN}{r}}\\
            &= \sum_{s\in\disc{0}{N-1}^d}\pr{\sum_{m\in\disc{1}{N}^d}e^{\frac{2i\pi}{N}\ps{m}{s}}}\pr{\sum_{q,r\in \Z^d}c_{s+Nq,Nr}(f)(-1)^{\ps{s+qN}{r}}}.
        \end{split}
    \end{equation}
    But notice that for all $s=(s_1,\dots,s_{d})\in \disc{0}{N-1}^d$,
    \begin{equation}
        \sum_{m\in\disc{1}{N}^d}e^{\frac{2i\pi}{N}\ps{m}{s}}=\left\{\begin{array}{ccccc}
             & N^d & \text{if} & s=0 \\
             & 0 & \text{otherwise} &
        \end{array}\right.
    \end{equation}
    in such a way that \eqref{Trace result - 1} becomes
    \begin{equation*}
        \begin{split}
            \mathrm{tr}(f_{N,\A{1},\A{2}}) &= \sum_{m\in \disc{1}{N}^d} F_{m,m}=N^d\sum_{q,r\in \Z^d}c_{qN,rN}(f)(-1)^{N\ps{q}{r}}\\
            &= \pr{N\A{1}^{1/2}\A{2}^{1/2}}^d\int_{\mc{T}_{\A{1},\A{2}}}f(x,\xi)dxd\xi+r_N
        \end{split}
    \end{equation*}
    where $r_N=N^d\sum_{(q,r)\neq (0,0)}c_{qN,rN}(f)(-1)^{N\ps{q}{r}}$. Applying \eqref{Estimation des coeffs de Fourier en les deux variables - 2} to $r_N$ concludes the proof.
\end{proof}

\begin{prop}\label{Chris and Zworski adapté}
    Let $0<\A{1}\leq \A{2}\ll 1$ be such that there exists $\omega\in \hspace{0.1cm} ]0,1[$ such that 
    \begin{equation}\label{Conditions entre alpha 1 et 2, et N - 1}
        N^{-\omega}\leq \pr{\alpha_1\alpha_2}^{1/2}.
    \end{equation}
    Let $p\in S(1,\A{1},\A{2})$. Then, as $N\to +\infty$,
    \begin{equation*}
        \norme{p_{N,\A{1},\A{2}}}\leq \sup_{\rho\in \mc{T}_{\A{1},\A{2}}}\bb{p(\rho)}+o(1).
    \end{equation*}
    In particular, there exists a constant $C>0$ independent of $N$, $\alpha_1$ and $\A{2}$ such that 
    \begin{equation*}
        \norme{p_{N,\A{1},\A{2}}}\leq C.
    \end{equation*}
\end{prop}

We essentially follow the proof of \cite[Proposition 2.7]{Christiansen_2010} which we present here an adapted version for the reader's convenience. It relies on Hörmander's idea for deriving $L^2$-boundedness from the semiclassical calculus.

\begin{proof}
    Recall that $\Tilde{h}=h/\sqrt{\alpha_1\alpha_2}$. First, from \eqref{Expression compo symboles périodiques} and Proposition \ref{Trace result}, we know that 
    \begin{equation}\label{Chris and Zworski adapté - 1}
        \begin{split}
            \norme{p_{N,\A{1},\A{2}}}_{HS}^2 &= \mathrm{tr}\pr{p_{N,\A{1},\A{2}}^*p_{N,\A{1},\A{2}}}=\mathrm{tr}\pr{\pr{\Bar{p}\sharp_{\Tilde{h}}p}_{N,\A{1},\A{2}}}\\
            &= (N\A{1}^{1/2}\A{2}^{1/2})^{d}\int_{\mc{T}_{\A{1},\A{2}}}\Bar{p}\Moy p (\rho)d\rho+r_N
        \end{split}
    \end{equation}
    where 
    \begin{equation}\label{Chris and Zworski adapté - 2}
        r_N=\mc{O}(1)N^{d-1}\A{1}^{(d-1)/2}\A{2}^{d/2}\sum_{\bb{\beta}+\bb{\gamma}\leq 2d+1}\norme{\partial_x^\beta\partial_\xi^\gamma \overline{p}\sharp_{\Tilde{h}}p}_{L^1(\mc{T}_{\A{1},\A{2}})}.\\
    \end{equation}
    Since $p\in S(1,\A{1},\A{2})$, so is $\Bar{p}\Moy p$ by \eqref{Expression compo symboles périodiques}. We therefore have for all $\beta,\gamma\in \N^d$,
    \begin{equation*}
        \int_{\mc{T}_{\A{1},\A{2}}}\partial_x^\beta\partial_\xi^\gamma\Bar{p}\Moy p (\rho)d\rho=\mc{O}_{\beta,\gamma}\pr{(\alpha_1\alpha_2)^{-d}}    
    \end{equation*}
    and it can be deduced from \eqref{Chris and Zworski adapté - 1} and \eqref{Chris and Zworski adapté - 2} that 
    \begin{equation}\label{Egalité intermédiaire Chris et Zwor}
        \norme{p_{N,\A{1},\A{2}}}\leq \norme{p_{N,\A{1},\A{2}}}_{HS}=\mc{O}\pr{N^{d/2}}.
    \end{equation}
    \m\\
    Let $M>\pr{\norme{p}_{L^\infty(\mc{T}_{\A{1},\A{2}})}}^2$ and define $a:=M-\Bar{p}\Moy p$. Then, $a$ is real-valued and for $\Tilde{h}$ small enough, there exists a global constant $C>0$ such that $a\geq 1/C>0$, that is, $a$ is elliptic.\\

    Now, we are going to construct by induction a sequence $(b_j)_{j\in \N}$ of real-valued functions such that for all $j\in \N$, $b^j\in \Tilde{h}^{j}S(1,\A{1},\A{2})$ and for all $J\in \N$, setting $B^J:=\sum_{j=0}^J b^j$, we have 
    \begin{equation}\label{Egalité démo Chris & Zwors}
        r_J:=B^J\Moy B^J-a\in \Tilde{h}^{(J+1)}S(1,\A{1},\A{2}).
    \end{equation}
    To this end, we set $b^0:=\sqrt{a}$ and the ellipticity of $a$ gives that $b^0\in S(1,\A{1},\A{2})$. In that case, we have $r^0:=b^0\Moy b^0-a\in \Tilde{h}S(1,\A{1},\A{2})$.\\
    We now assume that for $J\geq 0$, and for all $j\leq J$, $b^j$ is constructed such that $b^j\in \Tilde{h}^{(J+1)}S(1,\A{1},\A{2})$ and \eqref{Egalité démo Chris & Zwors} holds. \\
    Let us remark that for all function $b^{J+1}\in S(1,\A{1},\A{2})$, 
    \begin{equation*}
        \begin{split}
            r^{J+1}:&=(B^J+b^{J+1})\Moy(B^J+b^{J+1})-a\\
            &=r^J+B^J\Moy ^{J+1} +b^{J+1}\Moy B^j+b^{J+1}\Moy b^{J+1}\\
            &=r^J+b^0\Moy b^{J+1} +b^{J+1}\Moy b^0+R^J\Moy b^{J+1} +b^{J+1}\Moy R^J+b^{J+1}\Moy b^{J+1}
        \end{split}
    \end{equation*}
    where $R^J:=B^j-b^0\in \Tilde{h}S(1,\A{1},\A{2})$. This suggests to consider 
    \begin{equation*}
        b^{J+1}=-r^J/(2b^0).
    \end{equation*}
    Indeed, in that case, since $a$ is elliptic, $b^{J+1}\in \Tilde{h}^{(J+1)}S(1,\A{1},\A{2})$. Furthermore, $r^J$ and $b^0=\sqrt{a}$ being real-valued, the same is true for $b^{J+1}$. Finally, thanks to \eqref{Expression compo symboles périodiques}, we get that $r^{J+1}\in \Tilde{h}^{(J+2)}S(1,\A{1},\A{2})$.\\

    Let $J\geq 0$. Since $B^J$ real-valued, $B^J_N$ is selfadjoint. Therefore, remembering that $h=(2\pi N)^{-1}$, for all $u\in \Hd$, 
    \begin{equation}\label{Chris and Zworski adapté - 3}
        \begin{split}
            M\norme{u}^2-\norme{p_{N,\A{1},\A{2}}u}^2 &= \ps{a_{N,\A{1},\A{2}}u}{u}\\
            &=\ps{B^J_{N,\A{1},\A{2}}u}{B^J_{N,\A{1},\A{2}}u}-\ps{r^J_{N,\A{1},\A{2}}u}{u}\\
            &\geq -\norme{r^J_{N,\A{1},\A{2}}}\norme{u}^2.
        \end{split}
    \end{equation}
    Using \eqref{Egalité démo Chris & Zwors}, \eqref{Egalité intermédiaire Chris et Zwor} and \eqref{Conditions entre alpha 1 et 2, et N - 1}, it can be derived from \eqref{Chris and Zworski adapté - 3} that
    \begin{equation*}
        \begin{split}
            M\norme{u}^2-\norme{p_{N,\A{1},\A{2}}u}^2 &\geq -C_JN^{d/2}\pr{N\A{1}^{1/2}\A{2}^{1/2}}^{-(J+1)}\norme{u}^2\\
            &\geq -C_JN^{d/2-(J+1)(1-\omega)/2}.
        \end{split}
    \end{equation*}
    where $C_J$ does not depend on $N$, $\alpha_1$ or $\alpha_2$. Taking $J$ large enough such that $d/2-(J+1)(1-\omega)/2<0$ concludes the proof.
\end{proof}

\mm We now give an adapted version of some functional calculus presented in \cite[Chapter 8]{dimassi1999spectral} for the symbol class $S(m,\alpha_1,\alpha_2)$.
\begin{prop}
    Let $0<\A{1},\A{2}\leq 1$ be such that $N^{-1}\ll  \pr{\A{1}\A{2}}^{1/2}$ and let $m\geq 1$ be an order function on $\mc{T}_{\A{1},\A{2}}$. Let $p\in S(m,\A{1},\A{2})$ such that $p\geq 0$, $p+i$ is elliptic and admitting an asymptotic expansion $p\sim p_0+\Tilde{h}p_1+\dots$ in $S(m,\A{1},\A{2})$. Then, for all $\psi\in \mathscr{C}^{\infty}_c(\R)$, there exists
    \begin{equation}
        f\in S(1/m,\A{1},\A{2}) \m\m\m\m\m \text{such that}\m\m\m\m\m \psi(p_{N,\alpha_1,\A{2}})=f_{N,\alpha_1,\A{2}}
    \label{Proposition formule Hellfer-Sjös - formule 1}
    \end{equation}
    satisfying
    \begin{equation}
        f\sim \sum_{\nu\in \N}\Tilde{h}^\nu f_\nu \m\m \text{in}\m\m S(1/m,\A{1},\A{2}) \m\m\m\m\m\text{where} \m\m\m\m\m f_\nu\in S(1/m,\A{1},\A{2}).
    \label{Proposition formule Hellfer-Sjös - formule 2}
    \end{equation}
    In particular, we have that $f_0(x,\xi)=\psi\Bigl(p_0(x,\xi)\Bigl)$ and for all $\nu\geq 1$,
    \begin{equation}
        f_\nu(x,\xi)=\sum_{k=1}^{2\nu}g_k(x,\xi,\alpha)\psi^{(k)}\Bigl(p_0(x,\xi)\Bigl) \m\m\m\m\m\text{where} \m\m\m\m\m g_k\in S(1,\A{1},\A{2}).
    \label{Proposition formule Hellfer-Sjös - formule 3}
    \end{equation}
\label{Proposition formule Hellfer-Sjös}
\end{prop}

\begin{proof}
    We fix $\psi\in \mathscr{C}^{\infty}_c(\R)$ and denote $\widetilde{\psi}$ an almost holomorphic extension of $\psi$, that is, $\widetilde{\psi}\in \mathscr{C}^{\infty}_c(\C)$, $\widetilde{\psi}_{\left|\R\right.}=\psi$ and for all $N\in \N$, there exists $C_N>0$ such that $\bb{\partial_{\Bar{z}}\Tilde{\psi}(z)}\leq C_N\bb{\mathrm{Im}(z)}^N$ (see for example \cite[Chapter 8]{dimassi1999spectral} or \cite[Theorem 3.6]{Zworski} for more details). The key of the proof is to use the Hellfer-Sjöstrand formula (see \cite[Theorem 8.1]{dimassi1999spectral}): for all selfadjoint operator $A$, 
    \begin{equation}\label{Formule Hellfer-Sjöstrand}
        \psi(A)=-\frac{1}{\pi}\int_{\C}(z-A)^{-1}\partial_{\Bar{z}}\widetilde{\psi}(z)\mathrm{L}(dz).
    \end{equation}
    
    \mm We extend $m$ and $p$ on $\R^{2d}$ by periodicity. The symbol $p+i$ is elliptic (in $S(m)$) by assumption, so $p-z$ is still elliptic when $\bb{z}\leq C$ and $\mathrm{Im}(z)\neq 0$. So by Beals' Lemma, $(p^w-z)^{-1}=r_z^w$ for some $r_z\in S(m^{-1})\subset S(1)$, where we denote $p^w:=\mathrm{Op}^w_{\Tilde{h}}(p)$. Furthermore, the symbol $r_z$ is not only in $S(1)$, but also in $S(1,\alpha_1,\A{2})$. Indeed, let us recall from \eqref{Def de l'opérateur M_gamma,mu} and \eqref{Conjugaison par M_gamma,mu} that, denoting $M_{\gamma,\mu}:=\tau_{\gamma}e^{\frac{i}{\Tilde{h}}\ps{\cdot}{\mu}}$ for all $\gamma$, $\mu\in \R^d$, we have $M_{\gamma,\mu}p^w M_{\gamma,\mu}^{-1}=\pr{\tau_{(\gamma,-\mu)}p}^w$. In that case, for $\gamma\in \Ab{1}\Z^d$ and $\mu\in \Ab{2}\Z^d$, the periodicity of $p$ gives us
    \begin{equation}
        \begin{split}
            \pr{\tau_{(\gamma,-\mu)}r_z}^w &= M_{\gamma,\mu}r_z^w M_{\gamma,\mu}^{-1}=M_{\gamma,\mu}\pr{p^w-z}^{-1} M_{\gamma,\mu}^{-1}\\
            &=\pr{M_{\gamma,\mu}p^wM_{\gamma,\mu}^{-1}-z}^{-1}=\pr{\pr{\tau_{(\gamma,-\mu)}p}^w-z}^{-1}=r_z^w.
        \end{split}
    \label{Egalité périodicité du symbole de la résolvante}
    \end{equation}
    This equality being true in the space of linear continuous maps $\mathscr{S}(\R^d)\to \mathscr{S}'(\R^d)$ and $r_z$ and $\tau_{(\gamma,-\mu)}r_z$ being in $S(1)$, we can deduce from the Schwartz's kernel theorem that $\tau_{(\gamma,-\mu)}r_z=r_z$. Since $\gamma$ and $\mu$ were arbitrarily chosen, this implies that $r_z\in S(1,\A{1},\A{2})$.\\
    
    \mm Next, apply the modified Beals' estimates as in \cite[Proposition 8.6]{dimassi1999spectral}, and proceed as in the proof of Theorem 8.7 and equation (8.16) of \cite[Chapter 8]{dimassi1999spectral}. We obtain that there exists
    \begin{equation}
        f\in S(1/m,\alpha_1,\alpha_2) \m\m\m\m\m \text{such that}\m\m\m\m\m \psi(p^w)=f^w
    \label{Proposition formule Hellfer-Sjös - formule 1 - bis}
    \end{equation}
    satisfying
    \begin{equation}
        f\sim \sum_{\nu\in \N}\Tilde{h}^\nu f_\nu \m\m \text{in}\m\m S(1/m,\alpha_1,\alpha_2) \m\m\m\m\m\text{where} \m\m\m\m\m f_\nu\in S(1/m,\alpha_1,\alpha_2).
    \label{Proposition formule Hellfer-Sjös - formule 2 - bis}
    \end{equation}
    We have in particular that $f_0(x,\xi)=\psi\Bigl(p_0(x,\xi)\Bigl)$ and for all $\nu\geq 1$,
    \begin{equation}
        f_\nu(x,\xi)=\sum_{k=1}^{2\nu}g_k(x,\xi,\alpha)\psi^{(k)}\Bigl(p_0(x,\xi)\Bigl)
    \label{Proposition formule Hellfer-Sjös - formule 3 - bis}
    \end{equation}
    where for all $k\in \N$, $g_k\in S(1,\alpha_1,\alpha_2)$.\\

    \mm It remains to show that $\psi(p_{N,\alpha_1,\A{2}})=f_{N,\alpha_1,\A{2}}$. But from the equality $(p^w-z)^{-1}=r_z^w$ we first obtain that $\pr{p_{N,\A{1},\A{2}}-z}^{-1}=(r_z)_{N,\A{1},\A{2}}$. Using the Helffer-Sjöstrand formula, we finally get
    \begin{equation*}
        \begin{split}
            f_{N,\A{1},\A{2}} &= \psi\pr{p^w}_{\left|\Hd\right.}= -\frac{1}{\pi}\int_{\C}(r_z)_{N,\A{1},\A{2}}\partial_{\Bar{z}}\widetilde{\psi}(z)\mathrm{L}(dz) \\
            &=-\frac{1}{\pi}\int_{\C}\pr{p_{N,\A{1},\A{2}}-z}^{-1}\partial_{\Bar{z}}\widetilde{\psi}(z)\mathrm{L}(dz)=\psi\pr{p_{N,\A{1},\A{2}}}\qedhere.
        \end{split}
    \end{equation*}
\end{proof}

\section{Phase space dilation and functional calculus}\label{Résu funda}
\subsection{Smoothing of rough symbols}
Recall the set $\mathscr{C}^{0,\varrho}_{\mathrm{pw}}S_d$ of functions $[0,1]^d_x\times \T^d_\xi\to \C$ that are smooth with respect to $\xi$ and piecewise Hölder continuous with respect to $x$, given in Definition \ref{Définition espace d'intérêt}. Let $p\in\mathscr{C}^{0,\varrho}_{\mathrm{pw}}S_d$ and let $\mathscr{U}_p$ be the set of potential singularities of $p$ given in \eqref{Ensemble des singularités}. We restrict $p$ to $]0,1]^d\times
\T^d$ and we extend it by $\Z^d$ periodicity with respect to $x$. We denote this extension by $p:\R^d_x\times\T^d_\xi\to \C$. Using the partial Fourier transform with respect to $\xi$, we get 
\begin{equation}\label{Developpement en série de Fourier par rapport à xi}
    p(x,\xi)=\sum_{\nu\in\Z^d}p_\nu(x)e^{2i\pi\ps{\nu}{\xi}}
\end{equation}
where 
\begin{equation}\label{Coeffs de Fourier en une seule variable}
    p_\nu(x)=\int_{\T^d}p(x,\xi)e^{-2i\pi\ps{\nu}{\xi}}d\xi.
\end{equation}

\begin{rem}\label{Remarque sur le fait que les matrices via la nouvelle procédure de quantif sont bornées}
    Extending such a $p\in \mathscr{C}^{0,\varrho}_{\mathrm{pw}}S_d$ in this way to $\R^d\times \T^d\to \C$ and seeing $p:\R^{2d}\to \C$ thanks to the natural projection $\R^d\to \T^d$ allows us to say that
    \begin{equation*}
        p\in BS_\xi
    \end{equation*}
    (see Definition \ref{Définition de la classe de symboles avec la régularité la plus basse}). In that case, from Lemma \ref{Norme bornée des matrices pour la nouvelle procédure de quantification}, there exists a constant $C>0$ such that for all $N>0$, 
    \begin{equation}
        \norme{M_N(p)}\leq C,
    \end{equation}
    where $M_N(p)$ was defined in Notations \ref{Notation matrice nouvelle procédure de quantif via la matrice infinie}.
\end{rem}

\mm Let $0<\eta<1/4$. We consider $\psi\in \mathscr{C}^\infty_c(\R^d;\R)$ such that 
\begin{equation}
    \int_{\R^d}\psi(x)dx=1,\mm \mm\mm  \mm\mm\mm \mathrm{supp}(\psi)\subset B(0,1)
\label{Proprios de la fonction mollifiante - 1}
\end{equation}
and denote 
\begin{equation}
    \psi_h:x\mapsto h^{-d\eta}\psi(h^{-\eta}x).
\label{Proprios de la fonction mollifiante - 2}
\end{equation}
\m\\
\mm We define 
\begin{equation}
    \widetilde{p}(x,\xi):=\Bigl(p(\cdot,\xi)\ast \psi_h\Bigl)(x),\mm x\in \R^d
\label{Definition de p tilde}
\end{equation}
which belongs to $\mathscr{C}^{\infty}(\T^{2d})$ by periodicity. More precisely, remembering Definition \ref{Definition des espaces de travail dans le cas rescaling}, we have
\begin{equation}\label{p mollifié est dans la bonne classe de symboles}
    \widetilde{p}\in S_{\eta,0}(1,1,1).
\end{equation}
Indeed, for all $\beta,\gamma\in \N^d$, it derives from the equality
\begin{equation*}
    \partial^\beta_x\psi_h(x)=h^{-(d+\bb{\beta})\eta}\partial_x^\beta\psi(h^{-\eta}x)
\end{equation*}
that
\begin{equation}\label{Vérifiaction que p mollifié est dans la bonne classe de symboles}
    \begin{split}
        \partial_x^\beta\partial_\xi^\gamma\widetilde{p}(x,\xi) &= \pr{\partial^\gamma_\xi p(\cdot,\xi)\ast\partial_x^\beta\psi_h}\!(x)\\
        &=h^{-(d+\bb{\beta})\eta}\pr{\partial^\gamma_\xi p(\cdot,\xi)\ast\partial_x^\beta\psi(h^{-\eta}\cdot)}\!(x)\\
        &= h^{-\bb{\beta}\eta}\int_{\R^d}\partial_\xi^\gamma p(x-h^{\eta}z,\xi)\partial_x^\beta\psi(z)dz.
    \end{split}
\end{equation}
But using \eqref{Norme Hölder des dérivées partielles par rapport à xi}, we obtain from \eqref{Vérifiaction que p mollifié est dans la bonne classe de symboles} that for all $(x,\xi)\in \R^{2d}$,
\begin{equation*}
    \bb{\partial_x^\beta\partial_\xi^\gamma\widetilde{p}(x,\xi)}\leq C_{\beta,\gamma}h^{-\eta \bb{\beta}},
\end{equation*}
which in view of Definition \ref{Definition des espaces de travail dans le cas rescaling} ensures that \eqref{p mollifié est dans la bonne classe de symboles} holds.\\

\mm Similar to \eqref{Developpement en série de Fourier par rapport à xi}, we get
\begin{equation}\label{Developpement en série de Fourier par rapport à xi - version mollifiée}
    \widetilde{p}(x,\xi)=\sum_{\nu\in\Z^d}\widetilde{p}_\nu(x)e^{2i\pi\ps{\nu}{\xi}},
\end{equation}
where
\begin{equation}
    \widetilde{p}_\nu(x)=\int_{\T^d}\widetilde{p}(x,\xi)e^{-2i\pi\ps{\nu}{\xi}}d\xi=\int_{\T^d}\pr{\int_{\R^d}p(z,\xi)\psi_h(x-z)dz}e^{-2i\pi\ps{m}{\xi}}d\xi.
\label{Interversion convolution-intégrale pour l'expression of p_m tilde}
\end{equation}
\m\\

\subsection{Phase space dilation}
Next, we strengthen our assumption on $\alpha_1$ and $\alpha_2$.\\
Let $\eta\in \hspace{0.1cm} ]0,1/4[$. Let $\alpha$ satisfy
\begin{equation}\label{Conditions entre alpha 1 et 2, et N - 2}
    h^{1-2\eta}\ll \alpha\ll 1
\end{equation}
and let 
\begin{equation}\label{Conditions entre alpha 1 et 2, et N - 3}
    \alpha_1:=h^{2\eta}\alpha,\m\m\m\m\m \A{2}:=\alpha \m\m\m\m\m \text{and}\m\m\m\m\m \Tilde{h}:=\frac{h}{\sqrt{\A{1}\A{2}}}=\frac{h^{1-\eta}}{\alpha} .
\end{equation}
Recall \eqref{Hypothèse expression de h comme l'inverse de N - 2} and let $N\gg1$. Recall \eqref{Definition de p tilde}, \eqref{p mollifié est dans la bonne classe de symboles} and \eqref{Expression compo symboles périodiques}, and define 
\begin{equation}\label{Dvpt de Borel pour le théorème central}
    \widetilde{q}:=\overline{\widetilde{p}}\sharp_{h}\widetilde{p}\sim \widetilde{q}_0+h\widetilde{q}_1+... \m\text{in}\m S_{\eta,0}(1,1,1)\m\m\text{with}\m \widetilde{q}_0=\bb{\widetilde{p}}^2.
\end{equation}
\m\\
\m\m\m\m Next, consider the transformation 
\begin{equation}
    U_{\A{1}}:\phi\mapsto \pr{U_{\A{1}}\phi:x\mapsto \A{1}^{d/4}\phi\pr{\A{1}^{1/2}x}}.
\end{equation}
It maps continuously $\mathscr{S}(\R^d)\to \mathscr{S}(\R^d)$ and $\mathscr{S}'(\R^d)\to \mathscr{S}'(\R^d)$ by duality, and it is a bijection in both cases. By density and thanks to the factor $\A{1}^{d/4}$, $U_{\A{1}}$ is a unitary operator $L^2(\R^d)\to L^2(\R^d)$ satisfying
\begin{equation*}
    U_{\alpha_1}^*=U_{\alpha_1}^{-1}=U_{\alpha_1^{-1}}.
\end{equation*}
A simple calculation shows that 
\begin{equation}\label{Proprios de mapping de la transformation unitaire U_alpha}
    U_{\alpha_1}:\mc{H}^d_{h,1,1}\to \Hd
\end{equation}
unitarily with respect to the inner product $\ps{\cdot}{\cdot}_0$ given in Lemma \ref{Hilbert structure}. Moreover, for all $j\in  \disc{0}{N-1}^d$, in $\mathscr{S}'(\R^d)$,
\begin{equation*}
    U_{\A{1}} Q_j^{1}=Q_j^{\A{1}}.
\end{equation*}
\m\\
\mm We perform the change of variables 
\begin{equation}
    \T^{2d}\ni (x,\xi)=(\alpha_1^{1/2}\widetilde{x},\alpha_2^{1/2}\widetilde{\xi})    
\end{equation}
for all $(\widetilde{x},\widetilde{\xi})\in \mc{T}_{\A{1},\A{2}}$. Let us define the scaled version of $\widetilde{q}$ by 
\begin{equation}\label{Scaled version of q atchek}
    \textit{\v{q}}\pr{\widetilde{x},\widetilde{\xi}}:=\widetilde{q}\pr{\alpha_1^{1/2}\widetilde{x},\alpha_2^{1/2}\widetilde{\xi}}.
\end{equation}
By \eqref{Conditions entre alpha 1 et 2, et N - 3}, we have for all $\gamma$, $\beta\in \N^d$, 
\begin{equation}\label{Vérifiaction que q atchek est dans la bonne classe de symboles}
    \partial^\gamma_{\widetilde{x}}\partial^\beta_{\widetilde{\xi}}\atc{q}(\widetilde{x},\widetilde{\xi})=\alpha^{(\bb{\gamma}+\bb{\beta})/2}h^{\eta\bb{\gamma}}\pr{\partial^\gamma_x\partial^\beta_\xi\widetilde{q}}\pr{\A{1}^{1/2}\widetilde{x},\A{2}^{1/2}\widetilde{\xi}}.
\end{equation}
But $\widetilde{q}\in S_{\eta,0}(1,1,1)$ and $\alpha\leq 1$, so we obtain from \eqref{Vérifiaction que q atchek est dans la bonne classe de symboles} that
\begin{equation*}
    \bb{\partial^\gamma_{\widetilde{x}}\partial^\beta_{\widetilde{\xi}}\atc{q}(x,\xi)}\leq h^{\eta\bb{\gamma}} C_{\gamma,\beta}h^{-\eta\bb{\gamma}}=C_{\gamma,\beta},
\end{equation*}
which ensures that 
\begin{equation}
    \atc{q}\in S(1,\A{1},\A{2}).
\label{Le fait que q atchek soit dans l'espace de symboles rescalés}
\end{equation}
Furthermore, setting for all $\nu\in \N$
\begin{equation}\label{q_nu_tilde en version rescalée pour donner q_nu atchek}
    \atc{q}_\nu(\widetilde{x},\widetilde{\xi}):=(h^\eta\alpha)^\nu \widetilde{q}_\nu(h^\eta\alpha^{1/2}\widetilde{x},\alpha^{1/2}\widetilde{\xi})\in S(1,\alpha_1,\A{2}),
\end{equation}
where $\widetilde{q}_\nu$ are given in \eqref{Dvpt de Borel pour le théorème central}, we have
\begin{equation}\label{Dvpt de Borel de q atchek}
    \atc{q}\sim \atc{q}_0+\Tilde{h}\atc{q}_1+...\mm \text{in}\mm S(1,\A{1},\A{2}).
\end{equation}
Indeed, for all $\gamma,\beta\in \N^d$,
\begin{equation}
    \begin{split}
        \partial_{\widetilde{x}}^\gamma\partial_{\widetilde{x}}^\beta\!&\pr{\atc{q}(\widetilde{x},\widetilde{\xi})-\sum_{k=0}^{N-1}\Tilde{h}^k\atc{q}_k(\widetilde{x},\widetilde{\xi})}\\
        &= \alpha^{(\bb{\gamma}+\bb{\beta})/2}h^{\eta\bb{\gamma}}\!\croch{\pr{\partial^\gamma_x\partial^\beta_\xi\widetilde{q}}\!-\sum_{k=0}^{N-1}\underbrace{\Tilde{h}^k\pr{h^\eta\alpha}^k}_{=h^k}\pr{\partial^\gamma_x\partial^\beta_\xi\widetilde{q}_k}\!}\pr{\A{1}^{1/2}\widetilde{x},\A{2}^{1/2}\widetilde{\xi}}\\
        &= \alpha^{(\bb{\gamma}+\bb{\beta})/2}h^{\eta\bb{\gamma}}\partial^\gamma_x\partial^\beta_\xi\!\croch{\widetilde{q}-\sum_{k=0}^{N-1}h^k\widetilde{q}_k}\pr{\A{1}^{1/2}\widetilde{x},\A{2}^{1/2}\widetilde{\xi}}.
    \end{split}
\label{Estimées pour mq q atchek est dans la bonne classe de symboles}
\end{equation}
But remembering that $\widetilde{q}$ satisfies \eqref{Dvpt de Borel pour le théorème central} and using that $\alpha\leq 1$, we deduce from \eqref{Estimées pour mq q atchek est dans la bonne classe de symboles} that
\begin{equation*}
    \begin{split}
        \bb{\partial_{\widetilde{x}}^\gamma\partial_{\widetilde{\xi}}^\beta\pr{\atc{q}(\widetilde{x},\widetilde{\xi})-\sum_{k=0}^{N-1}\Tilde{h}^k\atc{q}_k(\widetilde{x},\widetilde{\xi})}} &\leq \alpha^{(\bb{\gamma}+\bb{\beta})/2}h^{\eta\bb{\gamma}}C_{\gamma,\beta}h^{N-\eta\bb{\gamma}}\\
        &\leq C_{\gamma,\beta}h^N\leq C_{\gamma,\beta}\Tilde{h}^N,
    \end{split}
\end{equation*}
which proves that \eqref{Dvpt de Borel de q atchek} holds.

\mm Performing the phase space dilation thanks to $U_{\A{1}}$, and remembering that $\Tilde{h}$ is given in \eqref{Conditions entre alpha 1 et 2, et N - 3}, we have 
\begin{equation}
    \mathrm{Op}_h^w(\widetilde{q})=U_{\A{1}}^{-1}\mathrm{Op}_{\widetilde{h}}^w\pr{\textit{\v{q}}}U_{\A{1}}.
\label{Relation scaling q et q atchek première}
\end{equation}
In view of Notations \ref{Notations pour les matrices des pseudo}, \eqref{Proprios de mapping de la transformation unitaire U_alpha} and \eqref{Relation scaling q et q atchek première} provide us 
\begin{equation}
    \widetilde{q}_{N}={U_{\A{1}}^{*}}_{{|\mc{H}^d_{h,\A{1},\A{2}}}}\textit{\v{q}}_{N,\A{1},\A{2}}{U_{\A{1}}}_{|\mc{H}^d_{h,1,1}}
\label{Relation scaling q et q atchek}
\end{equation}
where we see ${U_{\A{1}}}_{|\mc{H}^d_{h,1,1}}$ and ${U_{\A{1}}^{*}}_{{|\mc{H}^d_{h,\A{1},\A{2}}}}$ as their matrices in the basis $\mc{B}^1$ and $\mc{B}^{\A{1}}$ respectively. In particular, thanks to Proposition \ref{Chris and Zworski adapté}, 
\begin{equation}
    \norme{\widetilde{q}_N}\leq \underbrace{\Bigl\|{U_{\A{1}}^{*}}_{{|\mc{H}^d_{h,\A{1},\A{2}}}}\Bigl\|}_{=1}\norme{\atc{q}_{N,\alpha_1,\alpha_2}}\underbrace{\norme{{U_{\A{1}}}_{|\mc{H}^d_{h,1,1}}}}_{=1}\leq C
\label{Majoration de la norme de la matrice de Toeplitz}
\end{equation}
with $C$ independent of $N$.\\

\m\m\m\m Next, we aim to use another order function adapted to the rescaled symbol $\alpha^{-1}\textit{\v{q}}$, with $\alpha$ as in \eqref{Conditions entre alpha 1 et 2, et N - 2}. To this end, we define 
\begin{equation}
    \textit{\v{m}}(\widetilde{\rho}):=1+\frac{\widetilde{q}_0\pr{h^\eta\alpha^{1/2}\widetilde{\rho}_1,\alpha^{1/2}\widetilde{\rho}_2}}{\alpha}\geq 1,\mm \widetilde{\rho}=(\widetilde{\rho}_1,\widetilde{\rho}_2)\in \mc{T}_{\alpha_1,\alpha_2}.
\label{Fonction d'ordre théorème central}
\end{equation}
Note that $\atc{m}\in \mathscr{C}^\infty(\mc{T}_{\A{1},\A{2}})$ since from \eqref{Dvpt de Borel pour le théorème central}, $\widetilde{q}_0\in S_{\eta,0}(1,1,1)$. Using that $\widetilde{q}_0=\bb{\widetilde{p}}$, we get for all $\gamma=(\gamma_1,\gamma_2)\in \N^{2d}$, that
\begin{equation}
    \begin{split}
        \text{if}\m\bb{\gamma}=1,& \mm \partial_{\widetilde{\rho}}^\gamma\atc{m}(\widetilde{\rho})=h^{\eta\bb{\gamma_1}}\alpha^{-1/2}\pr{\partial^{\gamma}_{{\rho}}\widetilde{q}_0}\pr{h^\eta\alpha^{1/2}\widetilde{\rho}_1,\alpha^{1/2}\widetilde{\rho}_2}=\mc{O}(1)\atc{m}^{1/2}(\widetilde{\rho}),\\
        \text{if}\m \bb{\gamma}\geq 2, & \mm \partial_{\widetilde{\rho}}^\gamma\atc{m}(\widetilde{\rho})=\mc{O}(1).
    \end{split}
\label{Estimées de symboles pour la fonction d'ordre}
\end{equation}
Then, applying Taylor's formula and the estimates \eqref{Estimées de symboles pour la fonction d'ordre} we obtain
\begin{equation}
    \begin{split}
        \atc{m}(\widetilde{\rho}) &\leq \atc{m}(\widetilde{\mu})+\bb{\nabla \atc{m}(\widetilde{\mu})}\cdot\bb{\widetilde{\mu}-\widetilde{\rho}}+\dfrac{1}{2}\left[\int_{0}^1\bb{\nabla^2\atc{m}(s\widetilde{\rho}+(1-s)\widetilde{\mu})}ds\right]\cdot\bb{\widetilde{\mu}-\widetilde{\rho}}^2\\
        &\leq \atc{m}(\widetilde{\mu})+C\atc{m}^{1/2}(\widetilde{\mu})\cdot\bb{\widetilde{\rho}-\widetilde{\mu}}+C\bb{\widetilde{\mu}-\widetilde{\rho}}^2\\
        &\leq \atc{m}(\widetilde{\mu})+\frac{C}{2}\pr{\atc{m}(\widetilde{\mu})+\bb{\widetilde{\rho}-\widetilde{\mu}}^2}+C\bb{\widetilde{\rho}-\widetilde{\mu}}^2\\
        &\leq \atc{m}(\widetilde{\mu})+\dfrac{C}{2}\pr{1+\bb{\widetilde{\rho}-\widetilde{\mu}}^2}\atc{m}(\widetilde{\mu})+C\bb{\widetilde{\mu}-\widetilde{\rho}}^2\leq \widetilde{C}\jap{\widetilde{\rho}-\widetilde{\mu}}^2\atc{m}(\widetilde{\mu}),
    \end{split}
\label{Inégalité fonction d'ordre}
\end{equation}
where we used the fact that $\atc{m}\geq 1$ for the fourth and fifth inequality. Let $n=(n_1,n_2)\in \Z^{2d}$. Then, using the last inequality of \eqref{Inégalité fonction d'ordre} with $\widetilde{\mu}-(\Ab{1}n_1,\Ab{2}n_2)$ instead of $\widetilde{\mu}$, and the periodicity of $\atc{m}$, it follows that
\begin{equation*}
    \atc{m}(\widetilde{\rho})\leq \widetilde{C}\jap{\widetilde{\rho}-\widetilde{\mu}+(\Ab{1}n_1,\Ab{2}n_2)}^2\atc{m}(\widetilde{\mu}).
\end{equation*}
Taking the infimum over $n\in \Z^{2d}$ shows that $\atc{m}$ is an order function on $\mc{T}_{\alpha_1,\alpha_2}$ (see Definition \ref{Def order function sur le tore biscornu}).\\
\mm Next, we show that $\alpha^{-1}\atc{q}\in S(\atc{m},\A{1},\A{2})$. In view of \eqref{Fonction d'ordre théorème central} and \eqref{Estimées de symboles pour la fonction d'ordre}, we have
\begin{equation*}
    \begin{split}
        \alpha^{-1}\atc{q}(\widetilde{\rho})&\leq \mc{O}(1)\atc{m}(\widetilde{\rho}),\\
        \alpha^{-1}\partial^\beta\atc{q}(\widetilde{\rho}) &= \mc{O}_\beta(1)\atc{m}(\widetilde{\rho})^{1/2} \mm\text{if}\mm \bb{\beta}=1,\\
        \alpha^{-1}\partial^\beta\atc{q}(\widetilde{\rho}) &= \mc{O}_\beta(1)\alpha^{\bb{\beta}/2-1}\mm\text{if}\mm \bb{\beta}\geq 2.
    \end{split}
\end{equation*}
So, remembering that $\atc{m}\geq 1$, these estimates give that for all $\beta\in \N^d$,
\begin{equation*}
    \alpha^{-1}\partial^\beta\atc{q}(\widetilde{\rho})=\mc{O}_\beta(1)\atc{m}(\widetilde{\rho}),
\end{equation*}
which leads to
\begin{equation}
    \alpha^{-1}\atc{q}\in S(\atc{m},\A{1},\A{2}).
\label{Le fait que alpha inverse times q atchek soit dans l'espace de symboles rescalés}
\end{equation}
Similarly as for showing \eqref{Dvpt de Borel de q atchek}, we have
\begin{equation}\label{Dvpt de Borel de alpha-1 q atchek}
    \frac{1}{\alpha}\atc{q}\sim \frac{1}{\alpha}\atc{q}_0+\Tilde{h}\frac{1}{\alpha}\atc{q}_1+...\mm \text{in}\mm S(\atc{m},\A{1},\A{2}).
\end{equation}

It is clear from the expression \eqref{Fonction d'ordre théorème central} of $\atc{m}$ and \eqref{q_nu_tilde en version rescalée pour donner q_nu atchek} of $\atc{q}_0$ that $\alpha^{-1}\atc{q}_0+i$ is elliptic in $S(\atc{m},\alpha_1,\alpha_2)$. Moreover, \eqref{Dvpt de Borel de alpha-1 q atchek} ensuring that $\bb{\alpha^{-1}\atc{q}-\alpha^{-1}\atc{q}_0}=\mc{O}(\Tilde{h})\atc{m}$, we obtain that $\alpha^{-1}\atc{q}+i$ is
\begin{equation}\label{Ellipticité de alpha-1 q atchek}
    \text{elliptic in $S(\atc{m},\alpha_1,\alpha_2)$}.
\end{equation}
Knowing that $\alpha^{-1}\atc{q}$ satisfies \eqref{Le fait que alpha inverse times q atchek soit dans l'espace de symboles rescalés}, \eqref{Dvpt de Borel de alpha-1 q atchek} and \eqref{Ellipticité de alpha-1 q atchek}, Proposition \ref{Proposition formule Hellfer-Sjös} gives that for all $\psi\in \mathscr{C}^{\infty}_c(\R)$, there exists
\begin{equation}
    f\in S(1/\atc{m},\A{1},\A{2}) \m\m\m\m\m \text{such that}\m\m\m\m\m \psi\pr{\frac{\atc{q}_{N,\alpha_1,\A{2}}}{\alpha}}=f_{N,\alpha_1,\A{2}}
\label{Proposition formule Hellfer-Sjös appliquée - formule 1}
\end{equation}
satisfying
\begin{equation}
    f\sim \sum_{\nu\in \N}\Tilde{h}^\nu f_\nu \m\m \text{in}\m\m S(1/\atc{m},\A{1},\A{2}) \m\m\m\m\m\text{where} \m\m\m\m\m f_\nu\in S(1/\atc{m},\A{1},\A{2}).
\label{Proposition formule Hellfer-Sjös appliquée - formule 2}
\end{equation}
In particular, we have that $f_0(x,\xi)=\psi\pr{\frac{\atc{q}_0(x,\xi)}{\alpha}}$ and for all $\nu\geq 1$,
\begin{equation}
f_\nu(x,\xi)=\sum_{k=1}^{2\nu}g_k(x,\xi,\alpha)\psi^{(k)}\pr{\frac{\atc{q}_0(x,\xi)}{\alpha}} \m\m\m\m\m\text{where} \m\m\m\m\m g_k\in S(1,\A{1},\A{2}).
\label{Proposition formule Hellfer-Sjös appliquée - formule 3}
\end{equation}
\m\\

\mm Next, we adapt \cite[Proposition 4.1]{hager2007eigenvalue} to our situation.
\begin{prop}\label{Proposition Mildred}
    Let $\psi\in \mathscr{C}_c^\infty(\R)$. Let $\alpha, \alpha_1$ and $\alpha_2$ be as in \eqref{Conditions entre alpha 1 et 2, et N - 2} and \eqref{Conditions entre alpha 1 et 2, et N - 3}. Let $\widetilde{m}\in \mathscr{C}^\infty(\mc{T}_{\A{1},\A{2}},]0,+\infty[)$ be an order function on $\mc{T}_{\A{1},\A{2}}$ (as in Definition \ref{Def order function sur le tore biscornu}) such that $\Tilde{m}(\widetilde{\rho})=1$ for all $\widetilde{\rho}=(\widetilde{\rho}_1,\widetilde{\rho}_2)$ satisfying
    \begin{equation}
        \dfrac{\widetilde{q}_0\pr{h^\eta\alpha^{1/2}\widetilde{\rho}_1,\alpha^{1/2}\widetilde{\rho}_2}}{\alpha}\leq \sup\pr{\mathrm{supp}(\psi)}+\dfrac{1}{C}
    \label{Hypothèse fonction d'ordre Mildred}
    \end{equation}
    where $C$ is independent of $\alpha$, and $\widetilde{q}_0$ is given in \eqref{Dvpt de Borel pour le théorème central}. Then, \eqref{Proposition formule Hellfer-Sjös - formule 2} holds in $S(\widetilde{m},\A{1},\A{2})$ for $h$ and $\widetilde{h}$ sufficiently small. 
\end{prop}

\begin{proof}
    In \cite{hager2007eigenvalue}, the proof of Proposition 4.1 involves functionnal calculus via the Hellfer-Sjöstrand formula and symbolic calculs as described in Section \ref{tools}. To get the result, it is enough to follow their proof and add the periodicity of $\atc{m}$, $\widetilde{m}$ and $\widetilde{q}_0$. 
\end{proof}

\subsection{Trace and log-determinant formulae}
In anticipation of the next theorem, we state and prove this useful lemma:

\begin{lem}
    Let $0<\A{1},\A{2}\leq 1$ be such that $N^{-1}\ll  \pr{\A{1}\A{2}}^{1/2}$. Let $0< 1/C\leq \widehat{q}\in S(1,\A{1},\A{2})$ such that $\widehat{q}\sim \widehat{q}_0+\Tilde{h}\widehat{q}_1+...$ in $S(1,\A{1},\A{2})$. Then, 
    \begin{equation}
        \begin{split}
            \log\det&\pr{\widehat{q}_{N,\A{1},\A{2}}}\\
            &= \pr{N\A{1}^{1/2}\A{2}^{1/2}}^{d}\int_{\mc{T}_{\A{1},\A{2}}}\log\pr{\widehat{q}_0(\rho)}d\rho+\mc{O}\pr{N^d\pr{N\A{1}^{1/2}\A{2}^{1/2}}^{-1}}.
        \end{split}
    \end{equation}
    \label{Lemme intermédiaire théorème central}
\end{lem}

\begin{proof}
    For all $t\in [0,1]$, we define $\widehat{q}^t:=t\widehat{q}+(1-t)\in S(1,\A{1},\A{2})$. From the assumption, for all $t\in [0,1]$, $\widehat{q}^t\geq 1/C$ uniformly with respect to $t$. Applying Beals' Lemma (see e.g. \cite[Theorem 8.3]{Zworski} or \cite[Proposition 8.3]{dimassi1999spectral}), we know that $\pr{\pr{\widehat{q}^t}^w}^{-1}$ exists and is equal to $\pr{r^t}^w$ for some $r^t\in S(1)$. Using the same argument as for \eqref{Egalité périodicité du symbole de la résolvante} in Proposition \ref{Proposition formule Hellfer-Sjös}, we get that $r^t\in S(1,\A{1},\A{2})$. Moreover, since for all $t\in [0,1]$, $\widehat{q}^t\geq 1/C$, we have from Easy G\aa rding inequality (see for example \cite[Theorem 4.30]{Zworski}) that for $\Tilde{h}$ small enough, for all $t\in [0,1]$, $\widehat{q}^t_{N,\alpha_1,\alpha_2}\geq 1/2C$. In particular, we get that $\|r^t_{N,\alpha_1,\alpha_2}\|\leq \mc{O}(1)$ uniform with respect to $N$ and $t$. Then applying Beals' Lemma ensures that symbol estimates of $r^t$ are also uniform with respect to $t$.\\

    On the other hand, since $\widehat{q}=\widehat{q}_0+\Tilde{h}\widehat{q}_1$, we have $\widehat{q}^t=\widehat{q}^t_0+\Tilde{h}t\widehat{q}_1$ where $\widehat{q}_0^t=t\widehat{q}_0+(1-t)$, and $\widehat{q}_0^t\geq 1/\widetilde{C}$ for some $\widetilde{C}>0$ independent of $t$. The symbol estimates of $t\widehat{q}_1$ are naturally uniform in $t$. Then, 
    \begin{equation}\label{Estimées de symbole de l'inverse - 1}
        \widehat{q}^t_0\sharp_{\Tilde{h}}\frac{1}{\widehat{q}^t_0}=1+\Tilde{h}\widetilde{r}_1^t
    \end{equation}
    for some $\widetilde{r}_1^t\in S(1,\alpha_1,\alpha_2)$ with symbol estimates uniform in $t$. Notice that from \eqref{Estimées de symbole de l'inverse - 1}, we have
    \begin{equation}\label{Estimées de symbole de l'inverse - 2}
        \begin{split}
            \pr{\frac{1}{\widehat{q}_0^t}}^w &= \pr{r^t}^w\pr{\widehat{q}^t}^w \pr{\frac{1}{\widehat{q}_0^t}}^w = \pr{r^t}^w\pr{\pr{\widehat{q}^t_0\sharp_{\Tilde{h}}\frac{1}{\widehat{q}^t_0}}+\Tilde{h}\pr{t\widehat{q}_1\sharp_{\Tilde{h}}\frac{1}{\widehat{q}^t_0}}}^w\\
            &= \pr{r^t}^w\pr{1+\Tilde{h}\widehat{r}_1^t}^w= \pr{r^t}^w+\Tilde{h}\pr{r^t\sharp_{\Tilde{h}}\widehat{r}_1^t}^w,
        \end{split}
    \end{equation}
    where $\widehat{r}_1^t=\widetilde{r}_1^t+t\widehat{q}_1\sharp_{\Tilde{h}}\frac{1}{\widehat{q}^t_0}$. In particular, denoting $r_1^t:=-r^t\sharp_{\Tilde{h}}\widehat{r}_1^t$, it can be deduced from \eqref{Estimées de symbole de l'inverse - 2} that 
    \begin{equation*}
        r^t=\frac{1}{\widehat{q}_0^t}+\Tilde{h}r_1^t
    \end{equation*}
    and $r_1^t$ has symbol estimates uniform in $t$. In that case, recalling Notations \ref{Notations pour les matrices des pseudo}, that $h$ is given in \eqref{Hypothèse expression de h comme l'inverse de N - 2} and that $\Tilde{h}=h/\sqrt{\alpha_1\alpha_2}$, we have
    \begin{equation}\label{Estimées de symbole de l'inverse - 3}
        \begin{split}
            \frac{d}{dt}\log\det\pr{\widehat{q}^t_{N,\A{
            1},\A{2}}} &= \mathrm{tr} \pr{\pr{\widehat{q}^t_{N,\A{
            1},\A{2}}}^{-1}\frac{d}{dt}\pr{\widehat{q}^t_{N,\A{1}, \A{2}}}}=\mathrm{tr}\pr{\pr{r^t\sharp_{\Tilde{h}}\frac{d}{dt}\widehat{q}^t}_{N,\A{
    1},\A{2}}}\\
            &= \mathrm{tr}\pr{\croch{\pr{\widehat{q}^t_0}^{-1}\sharp_{\Tilde{h}}\frac{d}{dt}\pr{\widehat{q}^t_0}}_{N,\A{
            1},\A{2}}+\Tilde{h}u^t_{N,\alpha_1,\alpha_2}}\\
            &=\mathrm{tr}\pr{\croch{\pr{\widehat{q}^t_0}^{-1}\frac{d}{dt}\pr{\widehat{q}^t_0}}_{N,\A{
            1},\A{2}}+\Tilde{h}v^t_{N,\alpha_1,\alpha_2}},
        \end{split}
    \end{equation}
    for some $u^t,v^t\in S(1,\alpha_1,\alpha_2)$ whose symbol estimates are uniform in $t$. Applying Proposition \ref{Trace result} to \eqref{Estimées de symbole de l'inverse - 3}, we obtain
    \begin{equation*}
        \begin{split}
            \frac{d}{dt}\log&\det\pr{\widehat{q}^t_{N,\A{
            1},\A{2}}}\\
            &= \pr{N\A{1}^{1/2}\A{2}^{1/2}}^{d}\int_{\mc{T}_{\A{1},\A{2}}}\frac{d}{dt}\log(\widehat{q}^t_0(\rho))d\rho+\mc{O}\pr{N^d\pr{N\A{1}^{1/2}\A{2}^{1/2}}^{-1}}.
        \end{split}
    \end{equation*}
    Finally, it remains to integrate this equality for $t\in [0,1]$ to get the result.
\end{proof}

\begin{thm}\label{Théorème central}
    Let $0<\eta<1/4$ and $\varrho\in \hspace{0.1cm} ]0,1]$. Let $p$ as in \eqref{Developpement en série de Fourier par rapport à xi} and $q:=\bb{p}^2$. Let $\widetilde{p}$ be as in \eqref{Definition de p tilde}, $\mathscr{U}_p$ be given in \eqref{Ensemble des singularités} and $\widetilde{q}$ be as in \eqref{Dvpt de Borel pour le théorème central}. We assume that there exist ${\kappa_2}\in \hspace{0.1cm} ]0,1]$ such that, for $0\leq t\ll 1$,
    \begin{equation}
        \mathscr{V}(t):=\mathrm{Vol}\pr{\acc{\rho\in \T^{2d}\m|\m q(\rho)\leq t}}=\mc{O}(t^{\kappa_2})
    \label{Hypothèse volume théo central}
    \end{equation}
    and ${\kappa_1}\in \hspace{0.1cm} ]0,1]$ such that for $0\leq t\ll 1$,
    \begin{equation}
        \mathrm{L}\Bigl(\pr{\mathscr{U}_p+B(0,t)}\cap[0,1]^d\Bigl)=\mc{O}(t^{{\kappa_1}}).
    \label{Hypothèse volume des singularités épaissies}
    \end{equation}
    Let $\alpha$ be such that
    \begin{equation}\label{Hypothèse sur alpha - Théorème central}
        N^{-\eta\min(\varrho,{\kappa_1})}\ll\alpha\ll 1.
    \end{equation}
    Then, for all $\psi\in \mathscr{C}^\infty_c(\R)$,
    \begin{equation}
       \mathrm{tr}\!\croch{\psi\pr{\frac{\widetilde{q}_{N}}{\alpha}}} =N^d\pr{\int_{\T^{2d}}\psi\pr{\frac{q(\rho)}{\alpha}}d\rho+\mc{O}\pr{N^{-\eta{\kappa_1}}}+\mc{O}\pr{N^{-\varrho \eta}\alpha^{-1}}}.
    \label{Résultat trace théorème central}
    \end{equation}
    If $\chi\in \mathscr{C}^\infty_c(\R,[0,+\infty[)$, such that $\chi(0)>0$ then,
    \begin{equation}\label{Résultat log-det théorème central}
        \begin{split}
            \log\det\pr{\widetilde{q}_N+\alpha\chi\pr{\frac{\widetilde{q}_N}{\alpha}}} &= N^d\biggl(\int_{\T^{2d}}\log(q(\rho))d\rho+\mc{O}\pr{N^{-\eta{\kappa_1}}\bb{\log(\alpha)}}\\
            &\mm\mm+\mc{O}\pr{\alpha^{\kappa_2}\bb{\log(\alpha)}}+\mc{O}\pr{N^{-\varrho\eta}\alpha^{-1}}\biggl).
        \end{split}
    \end{equation}
\end{thm}

The rest of this section is devoted to the proof of Theorem \ref{Théorème central}.\\

    \mm 1. Let $\chi\in \mathscr{C}^\infty_c([0,+\infty[;[0,+\infty[)$ such that $\chi(0)>0$. Extend $\chi$ in $\mathscr{C}^{\infty}_c(\R;\C)$ such that $\chi(x)>0$ near $0$ and $x+\chi(x)\neq 0$ for all $x\in \R$. Notice in particular that for all $x\in [0,+\infty[$, 
    \begin{equation}\label{Inégalité sur x+chi(x)}
        x+\chi(x)\geq B
    \end{equation}
    for some $B>0$. Let us take
    \begin{equation}\label{Condition sur beta dans le théorème central}
        0<\beta\ll1.
    \end{equation}
    As in \eqref{Relation scaling q et q atchek}, performing the scaling $\A{1}=h^{2\eta}\alpha$, $\A{2}=\alpha$ and $\atc{q}(\widetilde{x},\widetilde{\xi})=\widetilde{q}\pr{h^\eta\alpha^{1/2}\widetilde{x},\alpha^{1/2}\widetilde{\xi}}$ yields 
     \begin{equation*}
        \widetilde{q}_N+\beta\chi(\beta^{-1}\widetilde{q}_N)=U_{\A{1}}^{-1}\pr{\atc{q}_{N,\A{1},\A{2}}+\beta\chi\pr{\beta^{-1}\atc{q}_{N,\A{1},\A{2}}}}U_{\A{1}}
    \end{equation*}
    where we recall that notations $\atc{q}_{N,\A{1},\A{2}}$ and $\widetilde{q}_N$ are given in \eqref{Definition de P_N_alpha_1_alpha_2}. By Proposition \ref{Proposition formule Hellfer-Sjös}, there exists $f\in S(1,\A{1},\A{2})$ satisfying \eqref{Proposition formule Hellfer-Sjös - formule 2} such that $\chi\pr{\beta^{-1}\atc{q}_{N,\A{1},\A{2}}}=f_{N,\A{1},\A{2}}$. In particular, $f_0(\widetilde{\rho})=\chi\pr{\beta^{-1}\atc{q}_0(\widetilde{\rho})}$.\\

    \begin{lem}\label{Estimées intermédiaires log-det}
        Let $\chi$ and $\beta$ be as above. Then, under the assumptions of Theorem \ref{Théorème central}, we have
        \begin{equation}\label{Estimées log-det avec beta}
        \begin{split}
            \log\det\Bigl(\widetilde{q}_N +\beta\chi\pr{\beta^{-1}\widetilde{q}_N}\Bigl) &= N^d\biggl(\int_{\T^{2d}}\log\Bigl(q(\rho)+\beta\chi(\beta^{-1}q(\rho))\Bigl)d\rho\\
            &\mm+\mc{O}\pr{N^{-\min(\varrho,{\kappa_1})\eta}}+\mc{O}\pr{\pr{N^{1-\eta}\alpha}^{-1}}\biggl).
        \end{split}
    \end{equation}
    \end{lem}

    \begin{proof}
        Recall the expression \eqref{Conditions entre alpha 1 et 2, et N - 3} of $\alpha_1$, $\alpha_2$ and $\Tilde{h}$. Using the invariance under the conjugation by unitary matrices of the determinant and applying Lemma \ref{Lemme intermédiaire théorème central}, we get 
    \begin{equation}
        \begin{split}
            \log&\det\pr{\widetilde{q}_N+\beta\chi\pr{\frac{\widetilde{q}_N}{\beta}}} \\
            &= \log\det\pr{\atc{q}_{N,\A{1},\A{2}}+\beta\chi\pr{\frac{\atc{q}_{N,\A{1},\A{2}}}{\beta}}}\\
            &= \pr{N\alpha_1^{1/2}\alpha_2^{1/2}}^{d}\int_{\mc{T}_{\A{1},\A{2}}}\log\pr{\atc{q}_0(\widetilde{\rho})+\beta\chi\pr{\frac{\atc{q}_0(\widetilde{\rho})}{\beta}}}d\widetilde{\rho}+\mc{O}\pr{N^d\pr{N^{1-\eta}\alpha}^{-1}}\\
            &= N^d\croch{\int_{\T^{2d}}\log\pr{\widetilde{q}_0(\rho)+\beta\chi\pr{\frac{\widetilde{q}_0(\rho)}{\beta}}}d\rho+\mc{O}\pr{\pr{N^{1-\eta}\alpha}^{-1}}},
        \end{split}
    \label{Estimée log-det q_N tilde - 1}
    \end{equation}
    where we performed a change of variables $(\rho_1,\rho_2)=(\alpha_1^{1/2}\widetilde{\rho}_1,\alpha_2^{1/2}\widetilde{\rho}_2)$ to obtain the last equality.\\ 
    
    \mm Next, we estimate the integral
    \begin{equation}
        \int_{\T^{2d}}\log\pr{\widetilde{q}_0(\rho)+\beta\chi\pr{\frac{\widetilde{q}_0(\rho)}{\beta}}}d\rho.
    \label{Première intégrale à estimer en fonction de q_frak}
    \end{equation}
    Denote $\phi_\beta:\R\ni x\mapsto x+\beta\chi(\beta^{-1}x)$ and notice that by \eqref{Inégalité sur x+chi(x)}, we have $\phi_\beta(x)\geq \beta B$ for $x\in [0,+\infty[$. Since $\beta$ is independent of $N$ and $\chi$ is compactly supported, $\log(\phi_\beta)'=\frac{\phi_\beta'}{\phi_\beta}$ is bounded uniformly in $h$. Using the Taylor-Laplace formula and integrating over $\T^{2d}$ provides
    \begin{equation}\label{Démo prop 1 du théorème central - 1}
        \begin{split}
            \int_{\T^{2d}}\log(\phi_{\beta}&(\widetilde{q}_0(\rho)))d\rho = \int_{\T^{2d}}\log\pr{\phi_{\beta}(q(\rho))}d\rho+ \Phi\pr{\T^{2d}},
        \end{split}
    \end{equation}
    where for all Borel subset $A$ of $\T^{2d}$ we denote
    \begin{equation*}
        \Phi(A):=\int_{A}(q(\rho)-\widetilde{q}_0(\rho))\left[\int_{s=0}^1\log\!\croch{(\phi_{\beta})'\Bigl((1-s)q(\rho)+s\widetilde{q}_0(\rho)\Bigl)}ds\right]d\rho.
    \end{equation*}
    \m\\
    Recall from \eqref{Ensemble des singularités} that $\mathscr{U}_p$ is the set of potential singularities of $p$. We then split the integral $\Phi\pr{\T^{2d}}$ of \eqref{Démo prop 1 du théorème central - 1} as follows:
    \begin{equation}\label{Décomposition de Phi(T^2d)}
        \Phi\pr{\T^{2d}}=\Phi\pr{\mc{B}_h}+\Phi\pr{\mc{A}_h}
    \end{equation}
    where
    \begin{equation}
        \mc{B}_h:=\Bigl(\Bigl[\partial\pr{[0,1]^d}\cup \mathscr{U}_p+B(0,2h^\eta)\Bigl]\cap [0,1]^d\Bigl)\times\T^d\m\m\text{and}\m\m \mc{A}_h:=\T^{2d}\setminus \mc{B}_h.
    \label{Définition de B_h}
    \end{equation}
 
    By \eqref{Definition de p tilde}, the triangular inequality gives 
    \begin{equation}
        \norme{\widetilde{p}}_{L^\infty(\R^d\times \T^d)}\leq \norme{p}_{L^\infty(\R^d\times \T^d)}\norme{\psi_h}_{L^1(\R^d_x)}=\norme{p}_{L^\infty(\R^d\times \T^d)}\norme{\psi}_{L^1(\R^d_x)}.            
    \label{Majoration norme infinie de p tilde}
    \end{equation}
    By \eqref{Dvpt de Borel pour le théorème central}, for all $\rho\in\R^d\times \T^d$,
    \begin{equation}
        \begin{split}
            \bb{q(\rho)-\widetilde{q}_0(\rho)}&=\bb{\bb{p(\rho)}^2-\bb{\widetilde{p}(\rho)}^2}\\
            &\leq \pr{\norme{p}_{L^\infty(\R^d\times \T^d)}+\norme{\widetilde{p}}_{L^\infty(\R^d\times \T^d)}}\bb{p(\rho)-\widetilde{p}(\rho)}
        \end{split}
    \label{Estimées de q moins q_0 tilde}
    \end{equation}
    and $\norme{p}_{L^\infty(\R^d\times \T^d)}+\norme{\widetilde{p}}_{L^\infty(\R^d\times \T^d)}= \mc{O}(1)$ is uniformly in $h$ by \eqref{Majoration norme infinie de p tilde}. Furthermore, recall that $p\in \mathscr{C}^{0,\varrho}_{\mathrm{pw}}S_d$ (cf Definition \ref{Définition espace d'intérêt}). So remembering \eqref{Proprios de la fonction mollifiante - 1}, \eqref{Proprios de la fonction mollifiante - 2}, \eqref{Definition de p tilde} and \eqref{Hypothèse expression de h comme l'inverse de N - 2}, we have for all $(x,\xi)\in \mc{A}_h$, that
    \begin{equation}
        \begin{split}
            \Bigl|\widetilde{p}(x,\xi)-p(x,\xi)\Bigl| &= \bb{\int_{\R^d}\Bigl[p(x-h^\eta y,\xi)-p(x,\xi)\Bigl]\psi(y)dy}\\
            &\leq \int_{\R^d}\Bigl|p(x-h^\eta y,\xi)-p(x,\xi)\Bigl|\bb{\psi(y)}dy\\
            &\leq \mc{O}(h^{\varrho\eta})=\mc{O}(N^{-\varrho\eta}).
        \end{split}
    \label{Estimées p moins p tilde sur A_h}
    \end{equation}
    Thus we deduce from \eqref{Estimées de q moins q_0 tilde}, \eqref{Estimées p moins p tilde sur A_h} and the fact that $\log(\phi_\beta)'$ is uniformly bounded in $h$, so in $N$, that 
    \begin{equation}
        \bb{\Phi(\mc{A}_h)}=\mathrm{Vol}(\mc{A}_h)\times\mc{O}\pr{h^{\varrho\eta}}=\mc{O}(h^{\varrho\eta})=\mc{O}(N^{-\varrho\eta}).
    \label{Estimée log-det q_N tilde - 2}
    \end{equation}
    Next, using \eqref{Hypothèse volume des singularités épaissies} and the fact that $\widetilde{q}_0$, $q$ and $\log(\phi_{\beta})'$ are bounded uniformly in $N$, we get
    \begin{equation}
        \bb{\Phi(\mc{B}_h)}\leq C\mathrm{Vol}(\mc{B}_h)=\mc{O}(h^{\eta{\kappa_1}})=\mc{O}(N^{-\eta{\kappa_1}}).
    \label{Estimée log-det q_N tilde - 3}
    \end{equation}
   The result derives then from \eqref{Estimée log-det q_N tilde - 1}, \eqref{Démo prop 1 du théorème central - 1}, \eqref{Décomposition de Phi(T^2d)}, \eqref{Estimée log-det q_N tilde - 2} and \eqref{Estimée log-det q_N tilde - 3}.
    \end{proof}
    
    \mm 2. Let us consider $0<h^{1-2\eta}\ll\alpha\leq t\leq \beta$. We introduce the same tools as \eqref{Conditions entre alpha 1 et 2, et N - 3}, \eqref{Scaled version of q atchek} and \eqref{q_nu_tilde en version rescalée pour donner q_nu atchek} with $\alpha$ replaced by $t$. That is, we set 
    \begin{equation}\label{Scaling en fonction de t}
        t_{1}=h^{2\eta}t,\m\m\m\m\m t_{2}=t \m\m\m\m\m\text{and}\m\m\m\m\m \Tilde{h}=\frac{h}{\sqrt{t_1t_2}}=\frac{h^{1-\eta}}{t}.
    \end{equation}
    and define  
    \begin{equation}
        \atc{q}(\widetilde{x},\widetilde{\xi}):=\widetilde{q}(t_1^{1/2}\widetilde{x},t_2^{1/2}\widetilde{\xi})
    \label{Scaling de q atchek en fonction de t}
    \end{equation}
    and for all $\nu\in \N$
    \begin{equation}
        \atc{q}_\nu(\widetilde{x},\widetilde{\xi}):=(t_1t_2)^{\nu/2}\widetilde{q}_\nu(t_1^{1/2}\widetilde{x},t_2^{1/2}\widetilde{\xi})
    \label{Scaling de q_0 atchek en fonction de t}
    \end{equation}
    We also define $\atc{m}$ as \eqref{Fonction d'ordre théorème central} with $\alpha$ replaced by $t$, that is
    \begin{equation}
        \textit{\v{m}}(\widetilde{\rho}):=1+\frac{\widetilde{q}_0\pr{h^\eta t^{1/2}\widetilde{\rho}_1,t^{1/2}\widetilde{\rho}_2}}{t}\geq 1,\mm \widetilde{\rho}\in \mc{T}_{t_1,t_2}
    \label{Fonction d'ordre théorème central - version en t}
    \end{equation}
    (see Notations \ref{Notation pour le tore biscornu}) so that, as for \eqref{Le fait que q atchek soit dans l'espace de symboles rescalés} and \eqref{q_nu_tilde en version rescalée pour donner q_nu atchek}, we have
    \begin{equation}
        \frac{1}{t}\atc{q}\in S(\atc{m},t_1,t_2) \mm\text{and}\mm \frac{1}{t}\atc{q}_\nu\in S(\atc{m},t_1,t_2),\m\m \nu\in \N. 
    \end{equation}
    \mm In anticipation of the next estimates, we prove the following result. 
    \begin{lem}\label{Proposition décomposition de la trace en plusieurs termes}
        Let $\widetilde{m}$ be  an order function satisfying \eqref{Hypothèse fonction d'ordre Mildred} with $\alpha$ replaced by $t$. Let $\psi\in \mathscr{C}^\infty_c(\R)$ and consider $0<h^{1-2\eta}\ll\alpha\leq t\leq \beta$ with notations \eqref{Scaling en fonction de t}, \eqref{Scaling de q atchek en fonction de t}, \eqref{Scaling de q_0 atchek en fonction de t} and \eqref{Fonction d'ordre théorème central - version en t}. Then, under the assumptions of Theorem \ref{Théorème central}, for $a\in \{0,1\}$ the following decomposition holds:
        \begin{equation}
        \mathrm{tr}\pr{\frac{1}{t^a}\psi\pr{\frac{\atc{q}_{N,t_1,t_2}}{t}}}=N^d\Bigl(I^a_t+II^a_t+III^a_t+IV^a_t\Bigl)
        \label{Décomposition en quatre termes théorème central bis} 
    \end{equation}
    where for $k\geq 1$,
    \begin{equation}
        \begin{split}
            I^a_t &:=\frac{(h^\eta t)^d}{t^a}\int_{\mc{T}_{t_1,t_2}}\psi\pr{\frac{\atc{q}_0(\widetilde{\rho})}{t}}d\widetilde{\rho}\\
            II^a_t &:=\mc{O}_M(1)\frac{(h^\eta t)^d}{t^a}\Tilde{h}\int_{\mc{T}_{t_1,t_2}}\widetilde{\chi}\pr{\frac{\atc{q}_0(\widetilde{\rho})}{t}}d\widetilde{\rho}\\
            III^a_t &:= \mc{O}_M(1)\frac{(h^\eta t)^d}{t^a}\Tilde{h}^M\int_{\mc{T}_{t_1,t_2}}\widetilde{m}(\widetilde{\rho})d\widetilde{\rho}\\
            IV^a_t &:= \mc{O}_k(1)\frac{(h^\eta t)^d}{t^a}\pr{Nh^\eta t^{1/2}}^{-k}\int_{\mc{T}_{t_1,t_2}}\widetilde{m}(\widetilde{\rho})d\widetilde{\rho}.
        \end{split}
    \label{Décomposition en quatre termes théorème central}
    \end{equation}
    and where $\widetilde{\chi}\in \mathscr{C}^{\infty}_c(\R;[0,1])$ is such that $\widetilde{\chi}\equiv 1$ on a neighborhood of $[0,\sup\mathrm{supp}(\psi)]$.
    \end{lem}

    \begin{proof}
    By Proposition \ref{Proposition Mildred}, there exists $f\in S(\widetilde{m},t_1,t_2)$ such that 
    \begin{equation}
        \begin{split}
            \psi\pr{t^{-1}\atc{q}_{N,t_1,t_2}}&=f_{N,t_1,t_2},\\
            f\sim \sum_{\nu\in \N}\Tilde{h}^\nu f_\nu \m\m\text{in} \m\m S(\widetilde{m},& t_1,t_2), \m\m\m\m f_\nu\in S(\widetilde{m},t_1,t_2),
        \end{split}
    \label{Utilisation prop de Mildred dans le théorème central}
    \end{equation}
    where $f_0=\psi(t^{-1} \atc{q}_0)\in S(\widetilde{m},t_1,t_2)$ and for $\nu\geq 1$, $f_\nu$ satisfies \eqref{Proposition formule Hellfer-Sjös appliquée - formule 3} with $\alpha$ replaced by $t$. Furthermore, the symbol estimates of $f_\nu$ are uniform with respect to $t$, so for any fixed $M\in \N^*$, we can write
     \begin{equation*}
         f - \pr{f_0+\Tilde{h}\sum_{\nu=1}^{M-1}\Tilde{h}^{\nu-1} f_\nu}\in \Tilde{h}^M S(\widetilde{m},t_1,t_2)
    \end{equation*}
    so, since $\Tilde{h}\leq 1$ and in view of \eqref{Proposition formule Hellfer-Sjös appliquée - formule 3}, 
    \begin{equation}
        \begin{split}
            \bb{f-f_0} &\leq  \mc{O}(\Tilde{h})\sum_{\nu=1}^{M-1} \bb{f_\nu}+\mc{O}\pr{\Tilde{h}^M}\widetilde{m}\\
            &\leq \mc{O}(\Tilde{h})\sum_{\nu=1}^{2M-2}\bb{\psi^{(k)}(t^{-1}\atc{q}_0)}+\mc{O}\pr{\Tilde{h}^M}\widetilde{m}\\
            &\leq\mc{O}(\Tilde{h})\widetilde{\chi}(t^{-1}\atc{q}_0)+\mc{O}\pr{\Tilde{h}^M}\widetilde{m}
        \end{split}
    \label{Estimation f_t en fonction de f_0, de chi tilde et de m tilde}
    \end{equation}
    where $\widetilde{\chi}\in \mathscr{C}^\infty_c(\R,[0,1])$ is such that $\widetilde{\chi}\equiv 1$ on a neighborhood of $[0,\sup(\mathrm{supp}(\psi))]$. In fact, \eqref{Estimation f_t en fonction de f_0, de chi tilde et de m tilde} allows us to write
    \begin{equation}
        f=f_0+\mc{O}(\Tilde{h})\widetilde{\chi}(t^{-1}\atc{q}_0)+\mc{O}\pr{\Tilde{h}^M}\widetilde{m},
    \end{equation}
    which, thanks to \eqref{Utilisation prop de Mildred dans le théorème central} and Proposition \ref{Trace result}, leads us to the decomposition \eqref{Décomposition en quatre termes théorème central bis}.
    \end{proof}

    \mm We now estimate each term of \eqref{Décomposition en quatre termes théorème central}. In the following, we choose $k=M$ so that for any $a\in \{0,1\}$,
    \begin{equation}
        IV_t^a=\mc{O}(III_t^a).
    \end{equation}

    \begin{lem}
        Under the assumptions of Lemma \ref{Proposition décomposition de la trace en plusieurs termes}, we have
        \begin{equation}
            \begin{split}
                \int_{\alpha}^\beta I^1_tdt = \int_{\T^{2d}}\int_{\alpha}^\beta \frac{1}{t}\psi\pr{\frac{q(\rho)}{t}}dtd\rho+\mc{O}\pr{N^{-\varrho\eta
        }\alpha^{-1}}+\mc{O}\pr{N^{-\eta{\kappa_1}}\bb{\log(\alpha)}}
        \end{split}
    \label{Théorème central - terme 1 - integré}
    \end{equation}
    and 
    \begin{equation}
    \begin{split}
        I_\alpha^0 &= \int_{\T^{2d}}\psi\pr{\frac{q(\rho)}{\alpha}}d\rho+\mc{O}(N^{-\varrho\eta
        }\alpha^{-1})+\mc{O}(N^{-\eta{\kappa_1}}).
    \end{split}
\label{Théorème central - terme 1 - pas integré}
\end{equation}
    \end{lem}

    \begin{proof}
    We perform the change of variables $(\rho_1,\rho_2)=\pr{t_1^{1/2}\widetilde{\rho}_1,t_2^{1/2}\widetilde{\rho}_2}$ which yields
    \begin{equation}
        I_t^a=\int_{\T^{2d}}\frac{1}{t^a}\psi\pr{\frac{\widetilde{q}_0(\rho)}{t}}d\rho.
    \label{Definition de I_t^a - 1}
    \end{equation}
    Applying the Taylor Laplace formula gives
    \begin{equation}\label{Taylor pour le terme d'erreur dans le théorème central - 1}
        \begin{split}
            \int_{\T^{2d}}\psi&\pr{\frac{\widetilde{q}_0(\rho)}{t}}d\rho=\int_{\T^{2d}}\psi\pr{\frac{q(\rho)}{t}}d\rho+\Phi\pr{\T^{2d}},
        \end{split}
    \end{equation}
    where for all Borel set $A\subset\T^{2d}$, we denote
    \begin{equation*}
        \Phi(A):=\int_{ A}\pr{\frac{\widetilde{q}_0(\rho)-q(\rho)}{t}}\int_{s=0}^1 \psi'\pr{(1-s)\frac{\widetilde{q}_0(\rho)}{t}+s\frac{q(\rho)}{t}}dsd\rho
        .
    \end{equation*}
    As \eqref{Décomposition de Phi(T^2d)}, to estimate $\Phi\pr{\T^{2d}}$, we consider the following decomposition
    \begin{equation}
        \Phi\pr{\T^{2d}}=\Phi(\mc{B}_h)+\Phi(\mc{A}_h)
    \end{equation}
    with $\mc{B}_h$ and $\mc{A}_h$ given in \eqref{Définition de B_h}.\\
    
    On the one hand, by \eqref{Estimées de q moins q_0 tilde} and \eqref{Estimées p moins p tilde sur A_h}, we get
    \begin{equation*}
        \begin{split}
            \norme{q-\widetilde{q}_0}_{L^\infty(\mc{A}_h)} =\mc{O}(h^{\varrho\eta}),
        \end{split}
    \end{equation*}
    which, using that $\psi$ is compactly supported ensures that
    \begin{equation}\label{Estimée erreur formule de TRI - sur A_h}
        \Phi\pr{\mc{A}_h}=\frac{\mc{O}\pr{h^{\varrho\eta}}}{t}.
    \end{equation}
    \m\\
    On the other hand, setting $C:=\sup(\mathrm{supp}(\psi))$, we introduce the two sets:
\begin{align*}
    \mathscr{E}_{t}&:=\{\rho\in \T^{2d}\m|\m \min\pr{\widetilde{q}_0(\rho),\widetilde{q}(\rho)}\leq Ct\},\\
    \mathscr{F}_{t}&:=\{\rho\in \T^{2d}\m|\m \max\pr{\widetilde{q}_0(\rho),\widetilde{q}(\rho)}\leq Ct\}.
\end{align*}
Notice that $\Phi\pr{\mc{B}_h}=\Phi\pr{\mc{B}_h\cap \mathscr{E}_{t}}$ and that
\begin{equation*}
    \Phi\pr{\mc{B}_h\cap \mathscr{E}_{t}}=\Phi\pr{\mc{B}_h\cap \mathscr{E}_{t}\cap \mathscr{F}_{t}}+\Phi\pr{\mc{B}_h\cap \mathscr{E}_{t}\cap \mathscr{F}_{t}^c}.
\end{equation*}
\mm Firstly, consider $\Phi\pr{\mc{B}_h\cap \mathscr{E}_{t}\cap \mathscr{F}_{t}}$. By definition, for all $\rho\in \mc{B}_h\cap \mathscr{F}_{t}$, $s\in [0,1]$, 
\begin{equation*}
    (1-s)\frac{q(\rho)}{t}+s\frac{\widetilde{q}_0(\rho)}{t}\in \mathrm{supp}\leq C.
\end{equation*}
Furthermore, for all $\rho\in \mc{B}_h\cap \mathscr{F}_{t}$,
\begin{equation*}
    -C\leq -\frac{\widetilde{q}_0(\rho)}{t}\leq \frac{q(\rho)-\widetilde{q}_0(\rho)}{t}\leq \frac{q(\rho)}{t}\leq C.
\end{equation*}
We deduce that 
\begin{equation}\label{Une des majorations de Phi(ensemble)}
    \bb{\Phi\pr{\mc{B}_h\cap \mathscr{E}_{t}\cap \mathscr{F}_{t}}}\leq C\norme{\psi'}_{L^\infty(\R)}\mathrm{L}\pr{\mc{B}_h\cap \mathscr{F}_h}.
\end{equation}
\m\\
\mm Secondly, consider $\Phi\pr{\mc{B}_h\cap \mathscr{E}_{t}\cap \mathscr{F}_{t}^c}$. Let $\rho\in \mathscr{E}_{t}\cap \mathscr{F}_{t}^c$. Since,
\begin{equation*}
    \frac{\min(q(\rho),\widetilde{q}_0(\rho))}{t}\leq C <\frac{\max(q(\rho),\widetilde{q}_0(\rho))}{t},
\end{equation*}
there exists $s_{\rho}\in [0,1]$ such that 
\begin{equation*}
    (1-s_{\rho})\frac{q(\rho)}{t}+s_{\rho}\frac{\widetilde{q}_0(\rho)}{t}=C.
\end{equation*}
We define new sets given by,
\begin{align*}
    \mathscr{G}^h_{t,1}&:=\{\rho\in \mc{B}_h\cap \mathscr{E}_{t}\cap \mathscr{F}_{t}^c\m |\m q(\rho)>\widetilde{q}_0(\rho)\}\\
    \mathscr{G}^h_{t,2}&:=\{\rho\in \mc{B}_h\cap \mathscr{E}_{t}\cap \mathscr{F}_{t}^c\m |\m q(\rho)<\widetilde{q}_0(\rho)\}.
\end{align*}
so that $\mathscr{G}^h_{t,1}\sqcup \mathscr{G}^h_{t,2}=\mc{B}_h\cap \mathscr{E}_{t}\cap \mathscr{F}_{t}^c$ and the previous $s_{\rho}$ is unique. Its value is
\begin{equation}\label{Expression de s_rho}
    s_{\rho}=\frac{q(\rho)-Ct}{q(\rho)-\widetilde{q}_0(\rho)}.
\end{equation}
Notice that
\begin{equation*}
    \left\{\begin{array}{cc}
         & \rho\in \mathscr{G}^h_{t,1} \m\text{and}\m s<s_{\rho}\Longrightarrow \psi'\pr{(1-s)\frac{q(\rho)}{t}+s\frac{\widetilde{q}_0(\rho)}{t}}=0, \\
         &  \rho\in \mathscr{G}^h_{t,2} \m\text{and}\m s>s_{\rho}\Longrightarrow \psi'\pr{(1-s)\frac{q(\rho)}{t}+s\frac{\widetilde{q}_0(\rho)}{t}}=0.
    \end{array}\right.
\end{equation*}
Thus,
\begin{equation*}
    \begin{split}
        \Phi\pr{\mc{B}_h\cap \mathscr{E}_{t}\cap \mathscr{F}_{t}^c}&=\Phi\pr{\mathscr{G}^h_{t,1}}+\Phi\pr{\mathscr{G}^h_{t,2}}\\
        &=\int_{\mathscr{G}^h_{t,1}}\pr{\frac{q(\rho)-\widetilde{q}_0(\rho)}{t}}\int_{s_{\rho}}^1\psi'\pr{(1-s)\frac{q(\rho)}{t}+s\frac{\widetilde{q}_0(\rho)}{t}}dsd\rho\\
        &\m\m\m\m +\int_{\mathscr{G}^h_{t,2}} \pr{\frac{q(\rho)-\widetilde{q}_0(\rho)}{t}}\int_{0}^{s_{\rho}}\psi'\pr{(1-s)\frac{q(\rho)}{t}+s\frac{\widetilde{q}_0(\rho)}{t}}dsd\rho.
    \end{split}
\end{equation*}
Using \eqref{Expression de s_rho}, that $\widetilde{q}_0=\bb{\widetilde{p}}^2\geq 0$ and $q=\bb{p}^2\geq 0$, we get
\begin{equation}\label{Une des majorations de Phi(ensemble) bis}
    \begin{split}
        |\Phi&(\mc{B}_h\cap \mathscr{E}_{t}\cap \mathscr{F}_{t}^c)|\\
        &\leq \norme{\psi'}_{L^\infty(\R)}\pr{\int_{\mathscr{G}^h_{t,1}}\bb{\frac{q(\rho)-\widetilde{q}_0(\rho)}{t}}(1-s_{\rho})d\rho+\int_{\mathscr{G}^h_{t,2}}\bb{\frac{q(\rho)-\widetilde{q}_0(\rho)}{t}}s_{\rho}d\rho}\\
        &= \norme{\psi'}_{L^\infty(\R)}\pr{C\mathrm{L}\pr{\mathscr{G}^h_{t,1}}+C\mathrm{L}\pr{\mathscr{G}^h_{t,2}}-\frac{1}{t}\!\croch{\int_{\mathscr{G}^h_{t,1}}\widetilde{q}_0(\rho)d\rho+\int_{\mathscr{G}^h_{t,2}}q(\rho)d\rho}}\\
        &\leq C\norme{\psi'}_{L^\infty(\R)}\pr{\mathrm{L}\pr{\mathscr{G}^h_{t,1}}+\mathrm{L}\pr{\mathscr{G}^h_{t,2}}}\\
        &= C\norme{\psi'}_{L^\infty(\R)}\mathrm{L}\pr{\mc{B}_h\cap \mathscr{E}_{t}\cap \mathscr{F}_{t}^c}. 
    \end{split}
\end{equation}

Gathering \eqref{Une des majorations de Phi(ensemble)} and \eqref{Une des majorations de Phi(ensemble) bis}, and using \eqref{Hypothèse volume des singularités épaissies}, we get 
\begin{equation}
    \begin{split}
        \bb{\Phi\pr{\mc{B}_h}}&\leq \bb{\Phi\pr{\mc{B}_h\cap \mathscr{E}_{t}\cap \mathscr{F}_{t}^c}}+\bb{\Phi\pr{\mc{B}_h\cap \mathscr{E}_{t}\cap \mathscr{F}_{t}}}\\
        &\leq C\norme{\psi'}_{L^\infty(\R)}\pr{\mathrm{L}\pr{\mc{B}_h\cap \mathscr{F}_h}+\mathrm{L}\pr{\mc{B}_h\cap \mathscr{E}_{t}\cap \mathscr{F}_{t}^c}}\\
        &\leq C\norme{\psi'}_{L^\infty(\R)}\mathrm{L}\pr{\mc{B}_h}=\mc{O}(h^{\eta{\kappa_1}})
    \end{split}
\label{Estimée erreur formule de TRI - sur B_h}
\end{equation}
uniformly with respect to $t$.\\

Finally, remembering \eqref{Hypothèse expression de h comme l'inverse de N - 2} and combining \eqref{Definition de I_t^a - 1} and \eqref{Taylor pour le terme d'erreur dans le théorème central - 1}, with \eqref{Estimée erreur formule de TRI - sur A_h} and \eqref{Estimée erreur formule de TRI - sur B_h}, we obtain \eqref{Théorème central - terme 1 - pas integré}, but we also have 
\begin{equation*}
    \begin{split}
        \int_{t=\alpha}^\beta I^1_tdt &= \int_{\T^{2d}}\int_{t=\alpha}^\beta \frac{1}{t}\psi\pr{\frac{q(\rho)}{t}}dtd\rho+\int_{t=\alpha}^\beta\frac{1}{t}\croch{\frac{\mc{O}(h^{\varrho\eta
        })}{t}+\mc{O}(h^{\eta{\kappa_1}})}dt\\
        &= \int_{\T^{2d}}\int_{t=\alpha}^\beta \frac{1}{t}\psi\pr{\frac{q(\rho)}{t}}dtd\rho+\mc{O}\pr{N^{-\varrho\eta
        }\alpha^{-1}}+\mc{O}\pr{N^{-\eta{\kappa_1}}\bb{\log(\alpha)}}.\qedhere
    \end{split}
\end{equation*}
\end{proof}

\begin{lem}
    Under the assumptions of Lemma \ref{Proposition décomposition de la trace en plusieurs termes}, we have
        \begin{equation}
    \begin{split}
        \int_{t=\alpha}^\beta II^1_tdt &=\mc{O}_M\pr{N^{-(1-\eta)}}\int_{t=0}^{C\beta}\frac{1}{\alpha+t}d\mathscr{V}(t)\\
        &\m\m+N^{-(1-\eta)}\alpha^{-1}\Bigl(\mc{O}_M\pr{N^{-\varrho\eta}\alpha^{-1}}+\mc{O}_M\pr{N^{-\eta{\kappa_1}}\bb{\log(\alpha)}}\Bigl)
    \end{split}
\label{Théorème central - terme 2 - integré}
\end{equation}
    and 
    \begin{equation}
    \begin{split}
        II^0_\alpha =N^{-(1-\eta)}\alpha^{-1}\Bigl(\mc{O}(\alpha^{\kappa_2})+\mc{O}(N^{-\varrho\eta}\alpha^{-1})+\mc{O}(N^{-\eta{\kappa_1}})\Bigl).
    \end{split}
    \label{Théorème central - terme 2 - pas integré}
\end{equation}
\end{lem}

\begin{proof}
    Denote $C:=\sup(\mathrm{supp}(\widetilde{\chi}))$. The same change of variables as for $I_t^a$, performed to obtain \eqref{Definition de I_t^a - 1} gives
\begin{equation}
    II_t^a=\mc{O}_M(1)\Tilde{h}\int_{\T^{2d}}\frac{1}{t^a}\widetilde{\chi}\pr{\frac{\widetilde{q}_0(\rho)}{t}}d\rho.
\label{Definition de I_t^a - 2}
\end{equation}
Also, since $t\in [\alpha,\beta]$, $\frac{1}{t}\leq \frac{1}{\alpha}$. Therefore, using \eqref{Scaling en fonction de t} and applying Taylor-Laplace formula to $\widetilde{\chi}$ give 
\begin{equation}
    \begin{split}
        II_t^a&=\mc{O}_M(1)\frac{h^{1-\eta}}{t^{1+a}}\int_{\T^{2d}}\widetilde{\chi}\pr{\frac{q(\rho)}{t}}d\rho+\mc{O}_M(1)\frac{h^{1-\eta}}{t^{1+a}}\mc{R}_{II}(t)\\
        &\leq \mc{O}_M(1)\frac{h^{1-\eta}}{t^{1+a}}\int_{\T^{2d}}\widetilde{\chi}\pr{\frac{q(\rho)}{t}}d\rho+\mc{O}_M(1) \frac{h^{1-\eta}}{\alpha t^a}\mc{R}_{II}(t)
    \end{split}
\end{equation}
with 
\begin{equation*}
    \mc{R}_{II}(t):=\int_{\T^{2d}}\pr{\frac{q(\rho)-\widetilde{q}_0(\rho)}{t}}\left[\int_{s=0}^1\widetilde{\chi}'\pr{(1-s)\frac{q(\rho)}{t}+s\frac{\widetilde{q}_0(\rho)}{t}}ds\right]d\rho.
\end{equation*}
Notice that the above discussion to treat the error term of \eqref{Taylor pour le terme d'erreur dans le théorème central - 1} applies to $\mc{R}_{II}(t)$ since $\widetilde{\chi}$ is also compactly supported, independently of $h$. So 
\begin{equation}
    \mc{R}_{II}(t)=\frac{\mc{O}(h^{\varrho\eta})}{t}+\mc{O}(h^{\eta{\kappa_1}}).
\label{Expression terme R_II(t)}
\end{equation}
For the following estimates, we use the application $\mathscr{V}$ defined by \eqref{Hypothèse volume théo central} which can be rewritten as
\begin{equation*}
    \mathscr{V}(t)=\int_{\T^{2d}}\mathbb{1}_{[0,t]}(q(\rho))d\rho.
\end{equation*}
It is a right continuous and increasing function (due to the continuity of $\mathrm{L}$), which gives rise to the measure $d\mathscr{V}=q_*\mathrm{L}$, the push-forward measure of the Lebesgue measure on $\T^{2d}$ by $q$. \\

We then have from \eqref{Hypothèse volume théo central} that
\begin{equation*}
    \int_{\T^{2d}}\widetilde{\chi}\pr{\frac{q(\rho)}{\alpha}}d\rho\leq \int_{\T^{2d}}\mathbb{1}_{[0,C\alpha]}(q(\rho))d\rho=\mathscr{V}(C\alpha)=\mc{O}(\alpha^{\kappa_2}).
\end{equation*}

This ensures that \eqref{Théorème central - terme 2 - pas integré} holds. On the other hand, when $a=1$, 
\begin{equation*}
    \begin{split}
        \mc{O}_M(1)\int_{t=\alpha}^\beta\int_{\T^{2d}}&\frac{1}{t^{2}}\widetilde{\chi}\pr{\frac{q(\rho)}{t}}d\rho dt \\
        &= \mc{O}_M(1)\int_{t=\alpha}^{\beta}\int_{\T^{2d}}t^{-2}\mathbb{1}_{[0,Ct]}\pr{q(\rho)}d\rho dt\\
        &= \mc{O}_M(1) \int_{\T^{2d}}\int_{t\in\R} t^{-2}\mathbb{1}_{\left[\max\pr{\alpha,\frac{q(\rho)}{C}},\beta\right]}\pr{t}\times \mathbb{1}_{[0,C\beta]}\pr{q(\rho)}dtd\rho\\
        &= \mc{O}_M(1)\int_{\T^{2d}}\pr{\frac{1}{\max\pr{\alpha,\frac{q(\rho)}{C}}}-\dfrac{1}{\beta}}\mathbb{1}_{[0,C\beta]}\pr{q(\rho)}d\rho\\
        &= \mc{O}_M(1)\int_{\T^{2d}}\frac{1}{\alpha+q(\rho)}\mathbb{1}_{[0,C\beta]}(q(\rho))d\rho.
    \end{split}
\end{equation*}
We finally have from \eqref{Hypothèse expression de h comme l'inverse de N - 2} and \eqref{Expression terme R_II(t)} that
\begin{equation}
    \begin{split}
        \int_{t=\alpha}^\beta II^1_tdt &= \mc{O}_M(h^{1-\eta})\int_{t=0}^{C\beta}\frac{1}{\alpha+t}d\mathscr{V}(t)+\mc{O}_M(h^{1-\eta}\alpha^{-1})\int_{t=\alpha}^\beta\frac{1}{t}\mc{R}_{II}(t)dt\\
        &=\mc{O}_M\pr{N^{-(1-\eta)}}\int_{t=0}^{C\beta}\frac{1}{\alpha+t}d\mathscr{V}(t)\\
        &\mm\mm+N^{-(1-\eta)}\alpha^{-1}\pr{\mc{O}_N\pr{N^{-\varrho\eta}\alpha^{-1}}+\mc{O}_M\pr{N^{-\eta{\kappa_1}}\bb{\log(\alpha)}}}.\qedhere
    \end{split}
\label{Théorème central - terme 2 - integré}
\end{equation}
\end{proof}

\begin{lem}
    Under the assumptions of Lemma \ref{Proposition décomposition de la trace en plusieurs termes}, we have
    \begin{equation}
        III^0_{\alpha}= \mc{O}_M\pr{\pr{N^{-(1-\eta)}}^M}+\mc{O}_M\pr{N^{-\eta\min({\kappa_1},\varrho)}\alpha^{-1}\pr{N^{-(1-\eta)}\alpha^{-1}}^M},
    \label{Théorème central - terme 3 - pas intégré}
    \end{equation}
    and 
    \begin{equation}
        \begin{split}
            \int_{t=\alpha}^{\beta} & III^1_{t}dt = \mc{O}_M\pr{\pr{N^{-(1-\eta)}}^M}\int_{t=0}^{C\beta}\pr{\frac{1}{\alpha+t}}^Md\mathscr{V}(t)\\
            &\mm\mm+\mc{O}_M\pr{N^{-\eta\min({\kappa_1},\varrho)}\alpha^{-1}\pr{N^{-(1-\eta)}\alpha^{-1}}^M}+\mc{O}_M\pr{\pr{N^{-(1-\eta)}}^M}.
        \end{split}
     \label{Théorème central - terme 3 - intégré}
     \end{equation}
\end{lem}

\begin{proof}
    For $M'\in \N^*$ (to be chosen later on), we define $\widetilde{m}$ defined on $\mc{T}_{t_1,t_2}$ by 
    \begin{equation}
        \widetilde{m}(\widetilde{\rho}):=\pr{1+\mathrm{dist}_{\mc{T}_{t_1,t_2}}\pr{\widetilde{\rho},\mathrm{supp}\left[\widetilde{\chi}\pr{\dfrac{\atc{q}_0\pr{\cdot}}{t}}\right]}^2}^{-M'}
    \label{Fonction d'ordre théorème central 2}
    \end{equation}
     where for every subset $A$ of $\mc{T}_{t_1,t_2}$ and $\widetilde{\rho}\in \mc{T}_{t_1,t_2}$, $\mathrm{dist}_{\mc{T}_{t_1,t_2}}(\widetilde{\rho},A):=\inf_{\widetilde{\mu}\in A}\bb{\widetilde{\rho}-\widetilde{\mu}}_{\mc{T}_{t_1,t_2}}$ (where we recall that $\bb{\cdot}_{\mc{T}_{t_1,t_2}}$ is given in \eqref{Norme weird sur le tore biscornu}). In particular, $\widetilde{m}$ is an order function in the sense of Definition \ref{Def order function sur le tore biscornu} and it is clear that it satisfies the property \eqref{Hypothèse fonction d'ordre Mildred} whith $\alpha$ replaced by $t$. Let $C>0$ to be chosen later on. We set 
     \begin{equation*}
         d_t(\widetilde{\rho}):=\mathrm{dist}_{\mc{T}_{t_1,t_2}}\pr{\widetilde{\rho},\mathscr{E}} \m\m\m\text{with}\m\m\m \mathscr{E}:=\left\{\widetilde{\rho}\in\mc{T}_{t_1,t_2}\m\left|\m \atc{q}_0\pr{\widetilde{\rho}}\leq Ct\right.\right\}.
     \end{equation*}
     Since $\widetilde{q}_0=\bb{\widetilde{p}}$, we get that for all $\widetilde{\rho}\in \mathscr{E}$,
     \begin{equation*}
         \bb{\nabla_{\widetilde{\rho}}\pr{\frac{\atc{q}_0(\widetilde{\rho})}{t}}}\leq \mc{O}(1)\frac{\widetilde{q}_0(h^\eta t^{1/2}\widetilde{\rho}_1,t^{1/2}\widetilde{\rho}_2)}{t^{1/2}}\leq \mc{O}(1)C^{1/2}=\mc{O}(1)
     \end{equation*}
     uniformly with respect to $t$ and $\widetilde{\rho}\in \mathscr{E}$. Moreover,
     \begin{equation*}
         \nabla^2_{\widetilde{\rho}}\pr{\frac{\atc{q}_0(\widetilde{\rho})}{t}}=\mc{O}(1)
     \end{equation*}
     uniformly with respect to $t$ and $\widetilde{\rho}$ too. Similarly to \eqref{Inégalité fonction d'ordre}, this gives by Taylor expansion and the periodicity of $\widetilde{q}_0$ that for all $\widetilde{\rho}\in \R^{2d}$, $\widetilde{\mu}\in \mathscr{E}$ and $\gamma\in \Z^{2d}$
     \begin{equation*}
         \frac{\atc{q}_0(\widetilde{\rho})}{t}\leq \mc{O}(1)\pr{1+\bb{\widetilde{\rho}-\widetilde{\mu}+(t_1^{-1/2}\gamma_1,t_2^{-1/2}\gamma_2)}+\bb{\widetilde{\rho}-\widetilde{\mu}+(t_1^{-1/2}\gamma_1,t_2^{-1/2}\gamma_2)}^2}.
     \end{equation*}
     It follows that for all $\widetilde{\rho}\in \R^{2d}$, 
     \begin{equation}
         \frac{\atc{q}_0(\widetilde{\rho})}{t}\leq \mc{O}(1)\Bigl(1+d_t(\widetilde{\rho})+d_t(\widetilde{\rho})^2\Bigl)\leq \mc{O}(1)\Bigl(1+d_t(\widetilde{\rho})\Bigl)^2\leq \mc{O}(1)\Bigl(1+d_t(\widetilde{\rho})^2\Bigl).
    \label{Majoration q_0 atc par la distance}
     \end{equation}
    Choosing 
    \begin{equation*}
        C=\sup\mathrm{supp}(\widetilde{\chi})
    \end{equation*}
    implies that   
    \begin{equation}\label{Majoration de m_tilde}
        \mathrm{dist}_{\mc{T}_{t_1,t_2}}\pr{\widetilde{\rho},\mathrm{supp}\left[\widetilde{\chi}\pr{\dfrac{\atc{q}_0\pr{\cdot}}{t}}\right]}\geq d_t(\widetilde{\rho}),
    \end{equation}
    so that in view of \eqref{Majoration q_0 atc par la distance},      
     \begin{equation}\label{Majoration de m_tilde par t q_0 atchek}
         \widetilde{m}(\widetilde{\rho})\leq \pr{1+d_t(\widetilde{\rho})^2}^{-M'}\leq\left\{\begin{array}{ccccc}
              \mc{O}_{M'}(1)\pr{1+t^{-1}\atc{q}_0\pr{\widetilde{\rho}}}^{-M'} & \text{if} & \atc{q}_0\pr{\widetilde{\rho}}\leq C\beta  \\
              \mc{O}_{M'}(1)\pr{t^{-1}\atc{q}_0\pr{\widetilde{\rho}}}^{-M'} & \text{if} & \atc{q}_0\pr{\widetilde{\rho}}> C\beta
         \end{array}\right..
     \end{equation}
     \m\\
     For $III_t^a$ to be estimated, we split the integral into two parts: one over $0\leq \atc{q}_0(\widetilde{\rho})\leq C\beta$ (denoted by $III_{t,1}^a$) and the other over $C\beta<\atc{q}_0(\widetilde{\rho})$ (denoted by $III_{t,2}^a$).\\
     \m\m\m\m Using \eqref{Condition sur beta dans le théorème central}, \eqref{Majoration de m_tilde par t q_0 atchek} and \eqref{Scaling en fonction de t}, we get that 
     \begin{equation}\label{Théorème central - terme 3 - 1re partie}
         \begin{split}
             III_{t,2}^a &= \mc{O}_M(1)\frac{(h^\eta t)^d}{t^a}\Tilde{h}^M\int_{\mc{T}_{t_1,t_2}}\widetilde{m}(\widetilde{\rho})\mathbb{1}_{]C\beta,+\infty[}(\atc{q}_0(\widetilde{\rho}))d\widetilde{\rho}\\
             &\leq \mc{O}_{M,M'}(1)\frac{(h^\eta t)^d}{t^a}\Tilde{h}^M\int_{\mc{T}_{t_1,t_2}}\pr{t^{-1}\atc{q}_0(\widetilde{\rho})}^{-M'}\mathbb{1}_{]C\beta,+\infty[}(\atc{q}_0(\widetilde{\rho}))d\widetilde{\rho}\\
             &= \mc{O}_{M,M'}(1)\frac{(h^{1-\eta})^M}{t^{M+a}}\int_{\T^{2d}}\pr{t^{-1}\widetilde{q}_0(\rho)}^{-M'}\mathbb{1}_{]C\beta,+\infty[}(\widetilde{q}_0(\rho))d\rho\\
             &= \mc{O}_{M,M'}(1)t^{M'-M-a}(h^{1-\eta})^M.
         \end{split}
     \end{equation}
     Next, we choose 
     \begin{equation}\label{Choix de M' en fonction de M et a}
        M'=M+a   
     \end{equation}
     so that the estimate \eqref{Théorème central - terme 3 - 1re partie} becomes independent of $t$. Thus, from \eqref{Hypothèse expression de h comme l'inverse de N - 2}, we get
     \begin{equation}
         III^0_{\alpha,2}=\mc{O}_M\pr{\pr{N^{-(1-\eta)}}^M}\m\m\m\m\m\text{and}\m\m\m\m\m \int_{t=\alpha}^\beta III^1_{t,2}dt=\mc{O}_M\pr{\pr{N^{-(1-\eta)}}^M}.
     \label{Théorème central - terme 3b - integré et pas intégré}
     \end{equation}
     \m\m\m\m For the first term, we consider $\widehat{\chi}\in\mathscr{C}^{\infty}_c(\R)$ such that $\widehat{\chi}\equiv 1$ on $[0,C\beta]$, $0\leq \widehat{\chi}\leq 1$ and $\mathrm{supp}(\widehat{\chi})\subset[-C\beta,2C\beta]$. Then, using \eqref{Scaling en fonction de t}, \eqref{Majoration de m_tilde par t q_0 atchek} and \eqref{Choix de M' en fonction de M et a}, we have
     \begin{equation*}
         \begin{split}
             III_{t,1}^a &= \mc{O}_M(1)\frac{(h^\eta t)^d}{t^a}\Tilde{h}^M\int_{\mc{T}_{t_1,t_2}}\widetilde{m}(\widetilde{\rho})\mathbb{1}_{[0,C\beta]}(\atc{q}_0(\widetilde{\rho}))d\widetilde{\rho}\\
             &\leq \mc{O}_{M}(1)\frac{(h^\eta t)^d}{t^a}\Tilde{h}^M\int_{\mc{T}_{t_1,t_2}}\pr{1+t^{-1}\atc{q}_0(\widetilde{\rho})}^{-M'}\mathbb{1}_{[0,C\beta]}(\atc{q}_0(\widetilde{\rho}))d\widetilde{\rho}\\
             &= \mc{O}_{M}(1)\frac{(h^{1-\eta})^M}{t^{M+a}}\int_{\T^{2d}}\pr{1+t^{-1}\widetilde{q}_0(\rho)}^{-M-a}\mathbb{1}_{[0,C\beta]}(\widetilde{q}_0(\rho))d\rho\\
             &\leq \mc{O}_{M}(1)\frac{(h^{1-\eta})^M}{t^{M+a}}\int_{\T^{2d}}\frac{\widehat{\chi}(\widetilde{q}_0(\rho))}{\pr{1+t^{-1}\widetilde{q}_0(\rho)}^{M+a}}d\rho.
         \end{split}
     \end{equation*}
     Denoting $\widehat{\psi}(x):=\widehat{\chi}(x)\pr{1+t^{-1}x}^{-M-a}$, we observe that for all $x\geq 0$,
     \begin{equation*}
         \bb{\widehat{\psi}'(x)}\leq \frac{M+a}{t}.
     \end{equation*}
     But remembering from \eqref{Dvpt de Borel pour le théorème central} that $\widetilde{q}_0=\bb{\widetilde{p}}^2\geq 0$ and from the assumptions of Theorem \ref{Théorème central} that $q=\bb{p}^2\geq 0$, we obtain that for all $\rho\in \T^{2d}$, $s\in [0,1]$,
     \begin{equation}\label{Majoration dérivée de psi chapeau}
        \bb{\widehat{\psi}'\Bigl((1-s)q(\rho)+s\widetilde{q}_0(\rho)\Bigl)} \leq \mc{O}_M\pr{\frac{1}{t}}.   
     \end{equation}
     Using a similar argument as for estimating $\Phi\pr{\T^{2d}}$ in \eqref{Taylor pour le terme d'erreur dans le théorème central - 1} via \eqref{Décomposition de Phi(T^2d)}, \eqref{Estimée erreur formule de TRI - sur A_h} and \eqref{Estimée erreur formule de TRI - sur B_h}, we also get 
     \begin{equation}\label{Estimée erreur formule de TRI - pour III}
         \int_{\T^{2d}}\widehat{\psi}(\widetilde{q}_0(\rho))d\rho=\int_{\T^{2d}}\widehat{\psi}(q(\rho))d\rho+\frac{\mc{O}_M(h^{\eta\min({\kappa_1},\varrho)})}{t}.
     \end{equation}
     Furthermore, using that for all $x\geq 0$, $0\leq \widehat{\chi}(x)\leq \mathbb{1}_{[0,2C\beta]}(x)$, and the estimate \eqref{Estimée erreur formule de TRI - pour III}, we obtain
     \begin{equation}\label{Théorème central - terme 3ab - estimée préliminaire}
         \begin{split}
             III_{t,1}^a &= \mc{O}_M(1)\frac{(h^{1-\eta})^M}{t^{M+a}}\int_{\T^{2d}}\frac{\mathbb{1}_{[0,2C\beta]}(q(\rho))}{\pr{1+t^{-1}q(\rho)}^{M+a}}d\rho\\
             &\mm\mm+\mc{O}_M(h^{\eta\min({\kappa_1},\varrho)})\frac{(h^{1-\eta})^M}{t^{M+a+1}}.
         \end{split}
     \end{equation}
     On the one hand, we can directly deduce from \eqref{Théorème central - terme 3ab - estimée préliminaire} and \eqref{Hypothèse expression de h comme l'inverse de N - 2} that
     \begin{equation}
         III^0_{\alpha,1}= \mc{O}_M\pr{N^{-\eta\min({\kappa_1},\varrho)}\alpha^{-1}\pr{N^{-(1-\eta)}\alpha^{-1}}^M}.
    \label{Théorème central - terme 3a - pas intégré}
     \end{equation}
     
     On the other hand, when $a=1$, we get from \eqref{Théorème central - terme 3ab - estimée préliminaire} that
     \begin{equation*}
        \begin{split}
            \int_{t=\alpha}^\beta III^1_{t,1}dt &= \mc{O}_M\pr{(h^{(1-\eta)})^M}\int_{\T^{2d}}\mathscr{J}\!(\rho)\mathbb{1}_{[0,2C\beta]}(q(\rho))d\rho\\
            &\mm\mm +\mc{O}_M\pr{h^{\eta\min({\kappa_1},\varrho)}\alpha^{-1}\pr{h^{1-\eta}\alpha^{-1}}^M}
        \end{split}
     \end{equation*}
     where we are going to estimate 
     \begin{equation*}
         \mathscr{J}\!(\rho):=\int_{t=\alpha}^\beta\frac{1}{t^{M+1}(1+t^{-1}q(\rho))^{M+1}}dt.
     \end{equation*}
     \mm First, we consider the case when $q(\rho)\leq \alpha$. Then, $0\leq t^{-1}q(\rho)\leq 1$ for all $t\in [\alpha,\beta]$, so 
     \begin{equation}
         \mathscr{J}\!(\rho)\leq \int_{t=\alpha}^{\beta}t^{-M-1}=\mc{O}_M\pr{\alpha^{-M}}.
     \label{Estimées de J(rho) - I}
     \end{equation}
     \mm Second, when $\alpha\leq q(\rho)\leq\beta$, we split $\mathscr{J}\!(\rho)$ into two parts: one where $\alpha\leq t\leq q(\rho)$ and the other where $q(\rho)\leq t\leq \beta$. Then, we have
     \begin{equation*}
         \int_{t=\alpha}^{q(\rho)}\frac{1}{t^{M+1}(1+t^{-1}q(\rho))^{M+1}}dt\leq \int_{t=\alpha}^{q(\rho)}\frac{1}{q(\rho)^{M+1}}dt\leq q(\rho)^{-M}
     \end{equation*}
     and 
     \begin{equation*}
         \int_{t=q(\rho)}^\beta\frac{1}{t^{M+1}(1+t^{-1}q(\rho))^{M+1}}dt\leq \int_{t=q(\rho)}^\beta\frac{1}{t^{M+1}}dt=\mc{O}(1)q(\rho)^{-M}.
     \end{equation*}
     This leads us to 
     \begin{equation}
         \mathscr{J}\!(\rho)=\mc{O}(1)q(\rho)^{-M}.
     \label{Estimées de J(rho) - II}
     \end{equation}
     \mm Third, when $\beta\leq q(\rho)$, then,
     \begin{equation}
         \mathscr{J}\!(\rho)\leq \int_{t=\alpha}^{\beta}\dfrac{1}{q(\rho)^{M+1}}dt\leq \mc{O}_M(1)q(\rho)^{-M}.
     \label{Estimées de J(rho) - III}
     \end{equation}
     Finally, thanks to \eqref{Hypothèse expression de h comme l'inverse de N - 2}, \eqref{Estimées de J(rho) - I}, \eqref{Estimées de J(rho) - II}, \eqref{Estimées de J(rho) - III}, and the fact that for all $0< a\leq b$, $n\in \N$, $b^{-n}\leq \frac{2^n}{(a+b)^n}$, we can conclude that 
     \begin{equation}
        \begin{split}
            \int_{t=\alpha}^{\beta} III^1_{t,1}dt &= \mc{O}_M((h^{1-\eta})^M)\int_{\T^{2d}}\Bigl[\mathbb{1}_{[0,\alpha]}\pr{q(\rho)}\alpha^{-M}+\mathbb{1}_{]\alpha,C\beta]}\pr{q(\rho)}q(\rho)^{-M}\Bigl]d\rho\\
            &\mm\mm\mm +\mc{O}_M\pr{h^{\eta\min({\kappa_1},\varrho)}\alpha^{-1}\pr{h^{1-\eta}\alpha^{-1}}^M}\\
            &= \mc{O}_M((h^{1-\eta})^M)\int_{\T^{2d}}\dfrac{1}{\pr{\alpha+q(\rho)}^M}\mathbb{1}_{[0,C\beta]}\pr{q\pr{\rho}}d\rho\\
            &\mm\mm\mm +\mc{O}_M\pr{h^{\eta\min({\kappa_1},\varrho)}\alpha^{-1}\pr{h^{1-\eta}\alpha^{-1}}^M}\\
            &= \mc{O}_M\pr{\pr{N^{-(1-\eta)}}^M}\int_{t=0}^{C\beta}\pr{\frac{1}{\alpha+t}}^Md\mathscr{V}(t)\\
            &\mm\mm\mm +\mc{O}_M\pr{N^{-\eta\min({\kappa_1},\varrho)}\alpha^{-1}\pr{N^{-(1-\eta)}\alpha^{-1}}^M}.
        \end{split}
     \label{Théorème central - terme 3a - intégré}
     \end{equation}
     Then, gathering the equalities \eqref{Théorème central - terme 3a - intégré}, \eqref{Théorème central - terme 3a - pas intégré} and \eqref{Théorème central - terme 3b - integré et pas intégré}, we get \eqref{Théorème central - terme 3 - pas intégré} and \eqref{Théorème central - terme 3 - intégré}.\\

\mm 3. We now have all the tools to conclude the proof of Theorem \ref{Théorème central}. We need to show \eqref{Résultat trace théorème central} and \eqref{Résultat log-det théorème central}.\\

\mm Recall \eqref{Relation scaling q et q atchek} and the fact that the trace is invariant under conjugation by unitary matrices. Then, we gather Lemma \ref{Proposition décomposition de la trace en plusieurs termes}, the results \eqref{Théorème central - terme 1 - pas integré}, \eqref{Théorème central - terme 2 - pas integré} and \eqref{Théorème central - terme 3 - pas intégré} to obtain
     \begin{equation}\label{Mille et unième estimation de trace}
         \begin{split}
             \mathrm{tr}\Bigl(\psi&\pr{\frac{\widetilde{q}_N}{\alpha}}\Bigl) = \mathrm{tr}\pr{\psi\pr{\frac{\atc{q}_{N,\alpha_1,\alpha_2}}{\alpha}}}\\
             &=N^d\biggl[\int_{\T^{2d}}\psi\pr{\frac{q(\rho)}{\alpha}}d\rho+\mc{O}(N^{-\varrho\eta
        }\alpha^{-1})+\mc{O}(N^{-\eta{\kappa_1}})\\
             &\mm+N^{-(1-\eta)}\alpha^{-1}\Bigl(\mc{O}(\alpha^{\kappa_2})+\mc{O}(N^{-\varrho\eta}\alpha^{-1})+\mc{O}(N^{-\eta{\kappa_1}})\Bigl)\\
             &\mm + \mc{O}_M\pr{N^{-\eta\min({\kappa_1},\varrho)}\alpha^{-1}\pr{N^{-(1-\eta)}\alpha^{-1}}^M}+\mc{O}_M\pr{\pr{N^{-(1-\eta)}}^M}\biggl].
         \end{split}
     \end{equation}
     From the assumptions of Theorem \ref{Théorème central}, we know that $\eta\in [0,1/4]$ and ${\kappa_1}, \varrho\in \hspace{0.1cm} ]0,1]$. In particular, we have
     \begin{equation}\label{Inégalités entre eta, varrho et upsilon - Théorème central}
         1-\eta>{\kappa_1}\eta\mm\mm\text{and}\mm\mm 1-\eta>\varrho\eta.
     \end{equation}
     It remains to use \eqref{Hypothèse sur alpha - Théorème central}, \eqref{Inégalités entre eta, varrho et upsilon - Théorème central} and to gather the error terms in \eqref{Mille et unième estimation de trace} together to obtain \eqref{Résultat trace théorème central}.\\

     \mm To obtain \eqref{Résultat log-det théorème central}, we consider $\psi\in \mathscr{C}^\infty_c(\R)$ such that for all $x\geq 0$, 
     \begin{equation*}
         \psi(x)=\frac{\chi(x)-x\chi'(x)}{x+\chi(x)}.
     \end{equation*}
     where $\chi$ is as before satisfying \eqref{Inégalité sur x+chi(x)}. So, for all $x\geq 0$,
     \begin{equation*}
         \frac{d}{dt}\croch{\log\pr{x+t\chi\pr{\frac{x}{t}}}}=\frac{1}{t}\psi\pr{\frac{x}{t}}.
     \end{equation*}
     Thus, by selfadjoint functional calculus and integration over $t\in [\alpha,\beta]$, we get that
     \begin{equation}\label{Estimée log-det en function de l'intégrale de la trace}
         \croch{\log\det\pr{\widetilde{q}_N+t\chi\pr{\frac{\widetilde{q}_N}{t}}}}_{t=\alpha}^{\beta}=\int_{t=\alpha}^\beta\mathrm{tr}\pr{\frac{1}{t}\psi\pr{\frac{\widetilde{q}_N}{t}}}dt.
     \end{equation}
     Putting together \eqref{Décomposition en quatre termes théorème central bis}, \eqref{Décomposition en quatre termes théorème central}, \eqref{Estimées log-det avec beta}, \eqref{Estimée log-det en function de l'intégrale de la trace}, \eqref{Théorème central - terme 1 - integré}, \eqref{Théorème central - terme 2 - integré} and \eqref{Théorème central - terme 3 - intégré}, we obtain
     \begin{equation}
         \begin{split}
             \log\det\pr{\widetilde{q}_N+\alpha\chi\pr{\frac{\widetilde{q}_N}{\alpha}}} &= N^d\Bigl[I+\mc{O}_M(N^{-(1-\eta)})II+\mc{O}_M((N^{-(1-\eta)})^M)III\\
             &\mm\mm+ \mc{O}(N^{-\varrho\eta}\alpha^{-1})+\mc{O}\pr{N^{-\eta{\kappa_1}} \bb{\log(\alpha)}}\Bigl].
         \end{split}
     \label{Résultat intermédiaire final théorème central}
     \end{equation}
     where
     \begin{equation}
         \begin{split}
             I&= \int_{\T^{2d}}\log\pr{q(\rho)+\alpha\chi\pr{\frac{q(\rho)}{\alpha}}}d\rho\\
             II&=\int_{t=0}^{C\beta}\frac{1}{\alpha+t}d\mathscr{V}(t)\\
             III&= \int_{t=0}^{C\beta}\pr{\frac{1}{\alpha+t}}^Md\mathscr{V}(t).
         \end{split}
     \label{Dernière décomposition théorème central}
     \end{equation}
     
     To complete the proof of \eqref{Résultat log-det théorème central}, it remains to estimate the three integrals of \eqref{Dernière décomposition théorème central}. \\     
     
     \mm For $I$, set $C:=\sup\mathrm{supp}(\chi)>0$ and split it into two parts depending on whether $q(\rho)> C\alpha$ or $q(\rho)\leq C\alpha$. For the first case, we clearly have
     \begin{equation*}
         \int_{\T^{2d}}\log\pr{q(\rho)+\alpha\chi\pr{\dfrac{q(\rho)}{\alpha}}}\mathbb{1}_{]C\alpha,+\infty[}\pr{q(\rho)}d\rho= \int_{ \T^{2d}}\log\pr{q(\rho)}\mathbb{1}_{]C\alpha,+\infty[}\pr{q(\rho)}d\rho.
     \end{equation*}
     Now we treat the second case. For $\alpha$ small enough, we can assume that $C\alpha\leq 1/2$. Then, for $0<q(\rho)\leq C\alpha$,
     \begin{equation*}
         \begin{split}
             \log\pr{q(\rho)+\alpha\chi\pr{\frac{q(\rho)}{\alpha}}} &=\log(q(\rho))+\log\pr{1+\frac{\alpha}{q(\rho)}\chi\pr{\frac{q(\rho)}{\alpha}}}\\
             &= \log(q(\rho))+\mc{O}(1)\log\pr{\frac{1}{q(\rho)}}.
         \end{split}
     \end{equation*}
     where $\mc{O}(1)$ is independent of $\alpha$. We therefore have
     \begin{equation*}
        \begin{split}
            \int_{\T^{2d}}&\log\pr{q(\rho)+\alpha\chi\pr{\dfrac{q(\rho)}{\alpha}}}\mathbb{1}_{[0,C\alpha]}\pr{q(\rho)}d\rho\\
            &\m\m\m\m\m\m\m\m\m\m\m\m\m = \int_{\T^{2d}} \log\pr{q(\rho)}\mathbb{1}_{[0,C\alpha]}\pr{q(\rho)}d\rho 
            +\mc{O}(1)\int_{0}^{C\alpha} \log\pr{\frac{1}{t}}d\mathscr{V}(t)
        \end{split}
     \end{equation*}
     It remains to estimate $\int_{\T^{2d}} \log\pr{\frac{1}{t}}d\mathscr{V}(t)$. To do so, we use an integration by parts and the hypothesis \eqref{Hypothèse volume théo central} to obtain
     \begin{equation*}
         \int_{0}^{C\alpha} \log(t)d\mathscr{V}(t)=\mathscr{V}(C\alpha)\log(C\alpha)-\int_{0}^{C\alpha}\frac{1}{t}\mathscr{V}(t)dt=\mc{O}\Bigl(\alpha^{\kappa_2}\bb{\log(\alpha)}\Bigl).
     \end{equation*}
     Thus, 
     \begin{equation}
         I=\int_{\T^{2d}}\log(q(\rho))d\rho+\mc{O}\Bigl(\alpha^{\kappa_2}\bb{\log(\alpha)}\Bigl).
     \label{Estimée majoration bis 1}
     \end{equation}
     \m\m\m\m In anticipation of the estimates of $II$ and $III$ of \eqref{Dernière décomposition théorème central}, we notice that for $b\in \N$, by integration by parts, change of variables and hypothesis \eqref{Hypothèse volume théo central}
     \begin{equation*}
        \begin{split}
            \int_{0}^{C\beta}\frac{1}{(t+\alpha)^b}d\mathscr{V}(t) &= \frac{\mathscr{V}(C\beta)}{(\alpha+C\beta)^b}+\int_{0}^{C\beta}\frac{1}{(t+\alpha)^{b+1}}\mathscr{V}(t)dt\\
            &=\frac{\mathscr{V}(C\beta)}{(\alpha+C\beta)^b}+\mc{O}_M(\alpha^{{\kappa_2}-b})\int_{0}^{C\beta/\alpha}\frac{u^{\kappa_2}}{(1+u)^{b+1}}du.
        \end{split}
     \end{equation*}
     So, the second term $II$ of \eqref{Dernière décomposition théorème central} (case $b=1$) can be estimated as
     \begin{equation}
         II=\left\{\begin{array}{ccccc}
              & \mc{O}\pr{\bb{\log(\alpha)}} & \text{when} & {\kappa_2}=1\\
              & \mc{O}\pr{\alpha^{{\kappa_2}-1}} & \text{when} & 0<{\kappa_2}<1
         \end{array}\right..
     \label{Estimée majoration bis 2}
     \end{equation}
     Likewise, but for $b=M\geq 2$,      \begin{equation}
         III=\mc{O}(\alpha^{{\kappa_2}-M}).
     \label{Estimée majoration bis 3}
     \end{equation}
     Thus, with \eqref{Résultat intermédiaire final théorème central}, \eqref{Dernière décomposition théorème central}, \eqref{Estimée majoration bis 1}, \eqref{Estimée majoration bis 2} and \eqref{Estimée majoration bis 3}, we obtain \eqref{Résultat log-det théorème central}
\end{proof}

\begin{cor}
    Assume we are under the assumptions of Theorem \ref{Théorème central}. We denote $\mc{N}(\widetilde{q}_N,\alpha)$ the number of eigenvalues of $\widetilde{q}_N$ in the interval $[0,\alpha]$. We moreover assume that $\alpha$ satisfies
    \begin{equation}\label{Hypothèse ordre de grandeur h-alpha théorème central}
        N^{-\varrho\eta}\alpha^{-1}\ll\alpha^{{\kappa_2}}\ll 1\mm\mm\text{and}\mm\mm N^{-{\kappa_1}\eta}\ll \alpha^{\kappa_2}\ll 1.
    \end{equation}
    Then, it holds,
    \begin{equation}\label{Nombre de valeurs propres plus petites que alpha}
        \mc{N}(\widetilde{q}_N,\alpha)=\mc{O}(N^d\alpha^{\kappa_2}).
    \end{equation}
\label{Nombre de valeurs propres plus petites que alpha - résultat}
\end{cor}

\begin{proof}
    Let $\Phi\in \mathscr{C}^\infty_c(\R)$ be such that $\Phi\equiv 1$ on $[0,1]$, $\Phi\equiv 0$ on $\R\setminus[-1,2]$ and $0\leq \Phi\leq 1$. The selfadjoint functional calculus gives us the equality 
    \begin{equation*}
        \sigma\pr{\Phi\pr{\alpha^{-1}\widetilde{q}_N}}=\Phi\pr{\sigma\pr{\alpha^{-1}\widetilde{q}_N}},
    \end{equation*}
    which allows us to write
    \begin{equation}
        \mathrm{tr}\croch{\Phi\pr{\frac{\widetilde{q}_N}{\alpha}}}=\sum_{\lambda\in \sigma(\widetilde{q}_N)}\Phi\pr{\frac{\lambda}{\alpha}}\geq \#\Bigl[\sigma(\widetilde{q}_N)\cap[0,\alpha] \Bigl]=\mc{N}(\widetilde{q}_N,\alpha).
    \label{Nombre de valeurs propres 1}
    \end{equation}
    However, since for all $\rho\in \T^{2d}$, $q(\rho)\geq 0$, 
    \begin{equation}
        \begin{split}
            \int_{\T^{2d}} \Phi\pr{\dfrac{q_0(\rho)}{\alpha}}d\rho\leq \int_{\T^{2d}}\mathbb{1}_{\left[0,2\alpha\right]}\pr{q_0(\rho)}d\rho=\mathscr{V}(2\alpha)=\mc{O}(\alpha^{\kappa_2}),
        \end{split}
    \label{Nombre de valeurs propres 2}
    \end{equation}
    the last equality coming from \eqref{Hypothèse volume théo central}. Given the estimate \eqref{Résultat trace théorème central} under the assumption \eqref{Hypothèse ordre de grandeur h-alpha théorème central}, it remains to combine \eqref{Nombre de valeurs propres 1} and \eqref{Nombre de valeurs propres 2} to obtain \eqref{Nombre de valeurs propres plus petites que alpha}. 
\end{proof}

\begin{rem}
    Let $\mc{K}$ be a fixed compact subset of $\C$. In view of its proof, if we set $q=\bb{p-z}^2$ and $\widetilde{q}=\overline{(\widetilde{p}-z)}\sharp_h(\widetilde{p}-z)$, and assume that \eqref{Hypothèse volume théo central} holds uniformly with respect to $z$ (to be compared to Assumption \ref{Hypothèse théorème final - volume préimage}), then Theorem \ref{Théorème central} still holds, and with error estimates uniform with respect to $z$. In particular, with the same notations, Corollary \ref{Nombre de valeurs propres plus petites que alpha - résultat} still holds too and 
    \begin{equation*}
        \mc{N}(\widetilde{q}_N,\alpha)=\mc{O}(N^d\alpha^{\kappa_2})
    \end{equation*}
    is uniform with respect to $z$.
\label{Nombre de valeurs propres plus petites que alpha indep de z}
\end{rem}

\section{Grushin problem}\label{Grushin}
We begin by recalling some elementary facts about Grushin problems. See \cite{sjoeandzworski} for more details. The idea is to set up a problem of the form
\begin{equation*}
    \mc{P}(z):=\begin{pmatrix}
        P(z)&R_-\\
        R_+&0
    \end{pmatrix}:\mc{H}_1\oplus \mc{H}_-\longrightarrow\mc{H}_2\oplus\mc{H}_+,
\end{equation*}
where $P(z):\mc{H}_1\to \mc{H}_2$ is to be examined, depending on some parameter $z\in \C$, where $\mc{H}_\pm$, $\mc{H}_j$, $j=1,2$ are complex Hilbert spaces and $R_\pm$ are suitably chosen operators so that the matrix of operators $\mc{P}(z)$ is invertible. Such a $\mc{P}(z)$ is called a \textit{Grushin problem}. In that case, we will write its inverse as follows:
\begin{equation*}
    \begin{pmatrix}
            P(z)&R_-\\
            R_+&0
        \end{pmatrix}^{-1}=\begin{pmatrix}
            E(z)&E_+(z)\\
            E_-(z)&E_{-+}(z)
        \end{pmatrix}.
\end{equation*}
The key argument relies on the Schur complement formula: if $\dim(\mc{H}_+)=\dim(\mc{H}_-)<+\infty$, then $P(z)$ is invertible if and only if $E_{-+}(z)$ is invertible. Furthermore, we have the following identities:
\begin{equation}
    P^{-1}(z)=E(z)-E_+(z)E^{-1}_{-+}(z)E_-(z)\m\m\m\text{and}\m\m\m E^{-1}_{-+}(z)=-R_+P^{-1}(z)R_-.
\label{Relations avec complément de Schur}
\end{equation}
In that case, since the study of $P(z)$ relies on that of $E_{-+}(z)$, we aim to find finite dimensional spaces $\mc{H}_\pm$ (in which case $E_{-+}(z)$ is a finite dimensional matrix).  

\subsection{A Grushin Problem for the unperturbed operator}

Let $N\in \N^*$ and $h=(2\pi N)^{-1}$ as in \eqref{Hypothèse expression de h comme l'inverse de N - 2} be the semiclassical parameter. Let $\varrho\in \hspace{0.1cm} ]0,1]$ and let us take $p\in \mathscr{C}^{0,\varrho}_{\mathrm{pw}}S_d$ (see Definition \ref{Définition espace d'intérêt}) with $\mathscr{U}_p$ the set of potential singularities of $p$ satisfying Assumptions \ref{Hypothèse théorème final - volume singularités} and \ref{Hypothèse théorème final - volume préimage}.
We restrict $p$ to $]0,1]^d_x\times \T^d_\xi$ and extend it by $\Z^d$ periodicity to $\R^d_x\times \T^d_\xi$ as in \eqref{Developpement en série de Fourier par rapport à xi}. Take $\widetilde{p}\in S_{\eta,0}(1,1,1)$ defined in \eqref{Definition de p tilde} with $0<\eta<1/4$. Let $\mc{K}$ be a compact subset of $\C$.\\

Let $\alpha$ be such that 
\begin{equation}
    N^{-\varrho\eta}\alpha^{-1}\ll \alpha^{\kappa_2}\ll1\mm\mm\text{and}\mm\mm N^{-{\kappa_1}\eta}\ll\alpha^{\kappa_2}\ll1.
\label{Hypothèse ordre de grandeur entre alpha et N - 0}
\end{equation}
By Lemma \ref{Lemme dimension et base de l'espace de travail}, we can identify $\mc{H}^d_{h,1,1}\cong \C^{N^d}$. Thanks to Proposition \ref{Chris and Zworski adapté} and \eqref{Majoration de la norme de la matrice de Toeplitz}, we know that $\widetilde{p}_N$ (see Notations \ref{Notations pour les matrices des pseudo}) is a bounded operator $\ell^2\to \ell^2$, uniformly in $N$.\\

Let $z\in \C$ and $\widetilde{q}:=\overline{(\widetilde{p}-z)}\#(\widetilde{p}-z)$. Notice in that case that the principal symbol of $\widetilde{q}$ is $\bb{\widetilde{p}-z}^2$. Since $Q:=\widetilde{q}_N=(\widetilde{p}_N-z)^*(\widetilde{p}_N-z)$ is positive (so selfadjoint), we can take an orthonormal basis of eigenfunctions $e_1,...,e_{N^d}\in \mc{H}^d_{h,1,1}$ associated with the eingenvalues that we will denote
\begin{equation}\label{Notation pour les valeurs singulières ordonnées}
    0\leq t_1^2\leq...\leq t_{N^d}^2.
\end{equation}
Moreover, $Q':=(\widetilde{p}_N-z)(\widetilde{p}_N-z)^*$ and $Q$ have the same spectrum. Indeed, let $N_0:=\dim\Bigl(\mathrm{Ker}(\widetilde{p}_N-z)^*\Bigl)$ and let $f_1,...,f_{N_0}$ be an orthonormal basis of $\mathrm{Ker}(\widetilde{p}_N-z)^*$. For all $j>N_0$, we set $f_j:=t_j^{-1}(\widetilde{p}_N-z)e_j$ (which is well-defined since $t_j> 0$). Then, for all $j\in \disc{1}{N^d}$,
\begin{equation*}
    (\widetilde{p}_N-z)^*f_j=t_je_j\m\m\m\text{and}\m\m\m (\widetilde{p}_N-z)e_j=t_jf_j
\end{equation*}
and we see that $f_1,...,f_{N^d}$ is an orthonormal basis of eigenvalues of $Q'$ associated with the eigenvalues \eqref{Notation pour les valeurs singulières ordonnées}.\\

By \eqref{Hypothèse ordre de grandeur entre alpha et N - 0}, Corollary \ref{Nombre de valeurs propres plus petites que alpha - résultat} and Remark \ref{Nombre de valeurs propres plus petites que alpha indep de z}, we know that $\mathfrak{m}:=\mc{N}\pr{Q,\alpha}$ fulfills the following condition:
\begin{equation}
    \mathfrak{m}=\mc{O}(N^d\alpha^{\kappa_2})
\label{Nombre de valeurs propres Problème de Grushin}
\end{equation}
uniformly with respect to $z\in \mc{K}$. Assume until further notice that
\begin{equation}\label{Hypothèse partie 6 sur le fait que M>0}
    \mathfrak{m}>0.
\end{equation}
In this case, it satisfies 
\begin{equation}\label{Valeurs propres inférieures à alpha}
    0\leq t_1^2\leq \dots\leq t_\mathfrak{m}^2\leq \alpha< t_{\mathfrak{m}+1}^2\leq \dots \leq t_{N^d}^2.
\end{equation}
We denote $(\delta_i)_{i\in\disc{1}{\mathfrak{m}}}$ an orthonormal basis of $\C^\mathfrak{m}$.\\

\m\m\m\m We set
\begin{equation}
    R_+:\left\{\begin{array}{ccccc}
          \mc{H}^d_{h,1,1} & \to & \C^\mathfrak{m} \\
          u & \mapsto & \sum_{j=1}^\mathfrak{m} \ps{u}{e_j}\delta_j
    \end{array}\right.\m\m\text{and}\m\m R_-:\left\{\begin{array}{ccccc}
          \C^\mathfrak{m} & \to & \mc{H}^d_{h,1,1} \\
          u_- & \mapsto & \sum_{j=1}^\mathfrak{m} \ps{u_-}{\delta_j}f_j
    \end{array}\right..
\label{Opérateurs R+ et R- pb de Grushin}
\end{equation}
We will show that the operator 
\begin{equation}
    \mc{P}(z):=\begin{pmatrix}
        \widetilde{p}_N-z&R_-(z)\\
        R_+(z)&0
    \end{pmatrix}:\mc{H}^d_{h,1,1}\times \C^\mathfrak{m}\longrightarrow\mc{H}^d_{h,1,1}\times\C^\mathfrak{m}.
\label{Operateur_Pb_Grushin}
\end{equation}
is bijective. In which case, we will denote its inverse 
\begin{equation}
    \mc{E}(z):=\begin{pmatrix}
        E(z)&E_+(z)\\
        E_-(z)&E_0(z)
    \end{pmatrix}.
\label{Problème de Grushin et inverse}
\end{equation}
\m\m\m\m Given $(v,v_+)\in \mc{H}^d_{h,1,1}\times\C^\mathfrak{m}$, we look for a solution $(u,u_-)\in \mc{H}^d_{h,1,1}\times\C^\mathfrak{m}$ such that 
\begin{equation}
    \mc{P}(z)\begin{pmatrix}
        u\\
        u_-
    \end{pmatrix}=\begin{pmatrix}
        v\\
        v_+
    \end{pmatrix}.
\label{Inversibilité problème de Grushin}
\end{equation}
We write it of the form
\begin{align*}
    v=\sum_{j=1}^{N}v(j)f_j\m\m\m\m\m\m&\text{and}\m\m\m\m\m\m v_+=\sum_{j=1}^\mathfrak{m}v_+(j)\delta_j\\
    u=\sum_{j=1}^{N}u(j)e_j\m\m\m\m\m\m&\text{and}\m\m\m\m\m\m u_-=\sum_{j=1}^\mathfrak{m}u_-(j)\delta_j\m.
\end{align*}
Solving \eqref{Inversibilité problème de Grushin} is then equivalent to solving
\begin{equation}
   \left\{\begin{array}{cc}
        & t_ju(j)=v(j) \m\m\m\text{for all}\m\m\m j\in\disc{\mathfrak{m}+1}{N^d} \\
        & \begin{pmatrix}
            t_j&1\\
            1&0
        \end{pmatrix}\begin{pmatrix}
            u(j)\\
            u_-(j)
        \end{pmatrix}=\begin{pmatrix}
            v(j)\\
            v_+(j)
        \end{pmatrix} \m\m\m\text{for all}\m\m\m j\in\disc{1}{\mathfrak{m}}
   \end{array}\right..
\label{Inversibilité problème de Grushin équivalente}
\end{equation}
Then, it is enough to set $u(j)=t_j^{-1}v(j)$ for all $j\in \disc{\mathfrak{m}+1}{N^d}$ and for all $j\in \disc{1}{\mathfrak{m}}$,
\begin{equation*}
    \begin{pmatrix}
            u(j)\\
            u_-(j)
    \end{pmatrix}=\begin{pmatrix}
            0&1\\
            1&-t_j
        \end{pmatrix}\begin{pmatrix}
            v(j)\\
            v_+(j)
        \end{pmatrix}.
\end{equation*}
This shows that \eqref{Operateur_Pb_Grushin} is bijective. Furthermore, in view of \eqref{Problème de Grushin et inverse}, we get
\begin{equation}
    \begin{array}{ccccc}
         & E(z)(\cdot)=\displaystyle\sum_{j=\mathfrak{m}+1}^{N^d}\dfrac{1}{t_j}\ps{\cdot}{f_j}e_j, &\m\m\m & E_+(z)(\cdot)=\displaystyle\sum_{j=1}^\mathfrak{m}\ps{\cdot}{\delta_j}e_j, \\
         & E_-(z)(\cdot)=\displaystyle\sum_{j=1}^\mathfrak{m}\ps{\cdot}{f_j}\delta_j, &\m\m\m & E_{-+}(z)(\cdot)=-\displaystyle\sum_{j=1}^\mathfrak{m} t_j\ps{\cdot}{\delta_j}\delta_j.
    \end{array}
\label{Expression des opérateurs de l'inverse au Problème de Grushin}
\end{equation}
This shows that 
\begin{equation}
    \norme{E(z)}\leq \alpha^{-1/2},\m\m\m\m \norme{E_-(z)}=\norme{E_+(z)}=1 ,\m\m\m\m \norme{E_{-+}(z)}\leq \alpha^{1/2}.
\label{Norme des opérateurs de l'inverse au pb de Grushin}
\end{equation}

\subsection{log-determinant estimates for \eqref{Operateur_Pb_Grushin}}

Notice that $\max(\alpha,\cdot)$ is continuous on $\R$. Thanks to \eqref{Inversibilité problème de Grushin équivalente} and the selfadjoint functional calculus, we have
\begin{equation}
    \bb{\det\pr{\mc{P}(z)}}^2=\prod_{j=\mathfrak{m}+1}^{N^d}t_i^2=\alpha^{-\mathfrak{m}}\prod_{j=1}^{N^d}\max(\alpha,t_i^2)=\alpha^{-\mathfrak{m}}\det\pr{\max(\alpha,Q)}.
\label{Expression det de mathscal(P)(z) en fonction de alpha et M}
\end{equation}
The Schur complement formula, \eqref{Operateur_Pb_Grushin} and \eqref{Problème de Grushin et inverse} give that
\begin{equation}
    \log\bb{\det(\widetilde{p}_N-z)}=\log\bb{\det\pr{\mc{P}(z)}}+\log\bb{\det\pr{E_0(z)}}.
\label{Utilisation de la conséquence de Schur opé non perturbé}
\end{equation}

Consider $\chi\in \mathscr{C}^{\infty}_c(\R)$ such that $\mathrm{supp}(\chi)\subset [-1,2]$, $0\leq \chi\leq 1$ and $\chi\equiv 1$ over $[0,1]$. Then, for all $x\geq 0$,
\begin{equation*}
    0<x+\frac{\alpha}{4}\chi\pr{\frac{4x}{\alpha}}\leq \max(\alpha,x)\leq x+\alpha\chi\pr{\frac{x}{\alpha}}.
\end{equation*}
Using the selfadjoint functional calculus, \eqref{Résultat log-det théorème central} of Theorem \ref{Théorème central} twice, \eqref{Nombre de valeurs propres Problème de Grushin} and \eqref{Expression det de mathscal(P)(z) en fonction de alpha et M}, we obtain
\begin{equation}
    \begin{split}
        \log\pr{ \bb{\det\pr{\mc{P}(z)}}^2} &= \log\det\pr{\max(\alpha,Q)}+\mathfrak{m}\log\pr{\frac{1}{\alpha}}\\
        &= N^d\biggl(\int_{\T^2}\log\pr{\bb{p(\rho)-z}^2}d\rho+\mc{O}\pr{N^{-\eta{\kappa_1}}\bb{\log(\alpha)}}\\
        &\mm\mm +\mc{O}\pr{\alpha^{\kappa_2}\bb{\log(\alpha)}}+\mc{O}\pr{N^{-\varrho\eta}\alpha^{-1}}\biggl).
    \end{split}
\label{log-det matrice pb de Grushin non perturbé}
\end{equation}

\begin{rem}
    In the case $\mathfrak{m}=0$, which means that \eqref{Hypothèse partie 6 sur le fait que M>0} does not hold anymore, the Grushin problem \eqref{Operateur_Pb_Grushin} simply reads $\mc{P}(z)=\widetilde{p}_N-z$. This operator is bijective since all the singular values of $\widetilde{p}_N-z$ are greater than $\sqrt{\alpha}>0$ by definition of $\mathfrak{m}$. In particular, we have, 
    \begin{equation*}
        \begin{split}
            \log\pr{ \bb{\det\pr{\mc{P}(z)}}^2} &= \log\det(Q)= \log\det\pr{\textbf{1}_{\alpha}(Q)}\\
            &= N^d\biggl(\int_{\T^2}\log\pr{\bb{p(\rho)-z}^2}d\rho+\mc{O}\pr{N^{-\eta{\kappa_1}}\bb{\log(\alpha)}}\\
            &\mm\mm +\mc{O}\pr{\alpha^{\kappa_2}\bb{\log(\alpha)}}+\mc{O}\pr{N^{-\varrho\eta}\alpha^{-1}}\biggl),
        \end{split}
    \end{equation*}
    which coincides with \eqref{log-det matrice pb de Grushin non perturbé}.
\end{rem}

\section{Estimates of logarithmic potential}
\label{log_pot}
\label{Principal setting} We start by recalling the definition of logarithmic potentials (see for example \cite[Section 7.2]{sjoestrand2019toeplitzbandmatricessmall}, or \cite[Section 4.1]{Bordenchafai_potentielloga}). Let us denote $\mathscr{P}(\C)$ the set of probability measures $\mu$ on $\C$ such that 
\begin{equation*}
    \int_{\C}\log\pr{1+\bb{x}}d\mu(x)<+\infty
\end{equation*}
If $\mu\in \mathscr{P}(\C)$, we define the \textit{logarithmic potential} of $\mu$ as the function 
\begin{equation}\label{Définition du potentiel loga pour une mesure donnée}
    \phi_\mu(z):=\int_{\C}\log\bb{\omega-z}d\mu(\omega), \mm z\in \C.
\end{equation}

Let $N\in \N^*$ and $h=(2\pi N)^{-1}$ as \eqref{Hypothèse expression de h comme l'inverse de N - 2} the semiclassical parameter. Let $\varrho\in \hspace{0.1cm} ]0,1]$ and let us take $p\in \mathscr{C}^{0,\varrho}_{\mathrm{pw}}S_d$ (see Definition \ref{Définition espace d'intérêt}) with $\mathscr{U}_p$ the set of potential singularities of $p$ satisfying Assumptions \ref{Hypothèse théorème final - volume singularités} and \ref{Hypothèse théorème final - volume préimage}.
We restrict $p$ to $]0,1]^d_x\times \T^d_\xi$ and extend it by $\Z^d$ periodicity to $\R^d_x\times \T^d_\xi$ as \eqref{Developpement en série de Fourier par rapport à xi}. Take $\widetilde{p}\in S_{\eta,0}(1,1,1)$ defined by \eqref{Definition de p tilde} with $0<\eta<1/4$. We set 
\begin{equation}
    \mu:=p_*\mathrm{L},
\end{equation}
where $\mathrm{L}$ denotes the Lebesgue measure on $[0,1]^d_x\times \T^d_\xi$, so that in this case, 
\begin{equation*}
    \phi_\mu(z)=\int_{[0,1]^d_x\times \T^d_\xi}\log\bb{p(\rho)-z}d\rho.
\end{equation*}

Let $\alpha$ be such that 
\begin{equation}
    N^{-\varrho\eta}\alpha^{-1}\ll \alpha^{\kappa_2}\ll1\mm\mm\text{and}\mm\mm N^{-{\kappa_1}\eta}\ll\alpha^{\kappa_2}\ll1.
\label{Hypothèse ordre de grandeur entre alpha et N}
\end{equation}

Let $Q_N$ be a random $N^d\times N^d$ complex matrix satisfying Assumption \ref{Hypothèse théorème final - matrice aléatoire} for some $\kappa_3>0$.\\

Let $\varepsilon_0>0$ be a positive real number to be chosen later on. From the Markov's inequality and the first point of Assumption \ref{Hypothèse théorème final - matrice aléatoire} on $Q_N$, we have
\begin{equation}
    \Prob{\norme{Q_N}>CN^{\kappa_3+\varepsilon_0}}\leq \frac{1}{C}\m \frac{\mathbb{E}\pr{\norme{Q_N}}}{N^{\kappa_3+\varepsilon_0}}=\mc{O}(N^{-\varepsilon_0}).
\label{Apparition de epsilon_0 dans la proba de la norme de la matrice aléatoire}
\end{equation}
Until further notice, we are going to consider $Q_N$ as a deterministic matrix and assume that
\begin{equation}
    \norme{Q_N}\leq CN^{\kappa_3+\varepsilon_0}.
\label{Majoration de la norme de la matrice aléatoire}
\end{equation}
Our main purpose is to study the eigenvalue distribution of 
\begin{equation}
    M^\delta_N(p):=M_N(p)+\delta Q_N
\label{Notation de la mtrice d'intérêt perturbée}
\end{equation}
where we recall that $M_N(p)$ is defined in \eqref{Coef infinite matrix quantization} and \eqref{Définition de la matrice avec la nouvelle procédure de quantif}, and with
\begin{equation}
    \delta:=N^{-(\kappa_3+\delta_0)}
\label{Ecriture de delta en décroissance polynomiale}
\end{equation}
for some $\delta_0>0$. We assume that
\begin{equation}
    \delta\ll N^{-(\kappa_3+\varepsilon_0)}\alpha^{1/2}.
\label{Hypothèse sur delta matrice générale}
\end{equation}
In particular, in view of \eqref{Majoration de la norme de la matrice aléatoire}, assumption \eqref{Hypothèse sur delta matrice générale} implies that if $\delta_0\geq \varepsilon_0$, then
\begin{equation}\label{Contrainte sur delta, alpha et la norme de la matrice aléatoire}
    0\leq \delta\alpha^{-1/2}\norme{Q_N}\leq 1/2. 
\end{equation}

\mm From now on, we will denote for every $N^d\times N^d$ matrix $A$ its singular values $s_j(A)$, $j\in \disc{1}{N^d}$ such that
\begin{equation}\label{Notation pour les valeurs singulières ordonnées bis}
    s_{N^d}(A)\leq s_{N^d-1}(A)\leq \dots \leq s_1(A)=\norme{A}.
\end{equation}
To be linked with notation \eqref{Notation pour les valeurs singulières ordonnées}, we have that for all $j\in \disc{1}{N^d}$,
\begin{equation*}
    s_j(A)=t_{N^d+1-j}(A).
\end{equation*}
\mm To study the eigenvalues of \eqref{Notation de la mtrice d'intérêt perturbée}, we will work with the logarithmic potential associated with the empirical spectral measure $\mu_N$ of $M^\delta_N(p)$ defined in \eqref{Mesure empirique - théorème final}, that is, 
\begin{equation}
    \mu_N:=\frac{1}{N^d}\sum_{\lambda\in \sigma\pr{M^\delta_N(p)}}\delta_\lambda.
\label{Expression mesure empirique de P perturbé via nouvelle quantif}
\end{equation}
Notice that 
\begin{equation}\label{Expression simplifiée du potentiel logarithmique en fonction du déterminant}
    \phi_{\mu_N}(z)=\frac{1}{N^d}\log\bb{\det\pr{M^\delta_N(p)-z}}=\frac{1}{N^d}\sum_{j=1}^{N^d}\log\pr{s_j\pr{M^\delta_N(p)-z}}.
\end{equation}
\mm Our aim is to compare $\phi_\mu$ and $\phi_{\mu_N}$. To this end, we introduce 
\begin{equation}
    \prescript{\delta}{}{\widetilde{p}}_N:=\widetilde{p}_N+\delta Q_N,
\label{Notation de la mtrice d'intérêt perturbée régularisée}
\end{equation}
and the corresponding empirical spectral measure
\begin{equation}
    \widetilde{\mu}_N:=\frac{1}{N^d}\sum_{\lambda\in \sigma\pr{\prescript{\delta}{}{\widetilde{p}}_N}}\delta_\lambda.
\label{Expression mesure empirique de P tilde perturbé}
\end{equation}
As for \eqref{Expression simplifiée du potentiel logarithmique en fonction du déterminant}, observe that 
\begin{equation}\label{Expression simplifiée du potentiel logarithmique en fonction du déterminant - 2}
    \phi_{\widetilde{\mu}_N}(z)=\frac{1}{N^d}\log\bb{\det\pr{\prescript{\delta}{}{\widetilde{p}_N}-z}}=\frac{1}{N^d}\sum_{j=1}^{N^d}\log\pr{s_j\pr{\prescript{\delta}{}{\widetilde{p}_N}-z}}.
\end{equation}
In Section \ref{Comparaison entre les potentiels loga - 1}, we will compare $\phi_{\widetilde{\mu}_N}$ with $\phi_\mu$ and in Section \ref{Comparaison entre les potentiels loga - 2}, $\phi_{\mu_N}$ with $\phi_{\widetilde{\mu}_N}$.

\subsection{Comparison of $\phi_{\widetilde{\mu}_N}$ and $\phi_\mu$}\label{Comparaison entre les potentiels loga - 1}
\begin{prop}
Let $\mc{K}\subset\subset \C$ and $\beta>0$. Let $\phi_{\mu_N}$ be as in \eqref{Expression simplifiée du potentiel logarithmique en fonction du déterminant} and let $\phi_{\widetilde{\mu}_N}$ be as in \eqref{Expression simplifiée du potentiel logarithmique en fonction du déterminant - 2}. Then, for all $z\in \mc{K}$, 
\begin{equation}
    \begin{split}
        \phi_{\widetilde{\mu}_N}(z) &\leq \phi_{\mu}(z)+\mc{O}\pr{N^{-\eta{\kappa_1}}\bb{\log(\alpha)}}+\mc{O}\pr{\alpha^{\kappa_2}\bb{\log(\alpha)}}\\
        &\m\m\m\m\m\m\m\m\m\m +\mc{O}\pr{N^{-\varrho\eta}\alpha^{-1}}+\mc{O}\pr{\delta\alpha^{-1/2}N^{\kappa_3+\varepsilon_0}}
    \end{split}
\label{Estimée log-det supérieure, cas gaussien -potentiel log}
\end{equation}
with probability $\geq 1-\mc{O}(N^{-\varepsilon_0})$ and
\begin{equation}
    \begin{split}
        \phi_{\widetilde{\mu}_N}(z) &\geq \phi_{\mu}(z)+\mc{O}\pr{N^{-\eta{\kappa_1}}\bb{\log(\alpha)}}+\mc{O}\pr{\alpha^{\kappa_2}\bb{\log(\alpha)}}\\
        &\m\m\m\m\m\m\m\m\m\m +\mc{O}\pr{N^{-\varrho\eta}\alpha^{-1}}+\mc{O}\pr{\alpha^{{\kappa_2}}\log(\delta N^{-\beta})}
    \end{split}
\label{Estimée log-det inférieure, cas gaussien -potentiel log}
\end{equation}
with probability $\geq 1-o(1)$.
\label{Différence des potentiels loga entre mu_N tilde et mu}
\end{prop}

\begin{proof}
    1. We start by setting up a Grushin problem for $\prescript{\delta}{}{\widetilde{p}}_N$. Let $\mc{P}(z)$, $R_+$ and $R_-$ be as in \eqref{Opérateurs R+ et R- pb de Grushin} and \eqref{Operateur_Pb_Grushin}. For all $z\in \mc{K}$, we set
\begin{equation*}
    \mc{P}^\delta(z):=\begin{pmatrix}
        \prescript{\delta}{}{\widetilde{p}}_N-z&R_-\\
        R_+&0
    \end{pmatrix}:\C^{N^d}\times \C^\mathfrak{m}\longrightarrow \C^{N^d}\times \C^\mathfrak{m}.
\end{equation*}
Denote
\begin{equation*}
    B(z):=\begin{pmatrix}
    \delta Q_NE(z)&\delta Q_NE_+(z)\\
    0&0
\end{pmatrix}.
\end{equation*}
Since $\mc{P}(z)^{-1}=\mc{E}(z)$, see \eqref{Problème de Grushin et inverse}, we have
\begin{equation}
    \mc{P}^{\delta}(z)\mc{E}(z)=\mc{P}(z)\mc{E}(z)+B(z)=I+B(z).
\label{Multiplication P_delta avec E}
\end{equation}
By \eqref{Norme des opérateurs de l'inverse au pb de Grushin}, \eqref{Hypothèse sur delta matrice générale} and \eqref{Contrainte sur delta, alpha et la norme de la matrice aléatoire}
\begin{equation*}
    \norme{B(z)}\leq \norme{\delta Q_NE(z)}+\norme{\delta Q_NE_+(z)}\leq\frac{1}{2}(1+\alpha^{1/2})<1
\end{equation*}
since $\alpha\ll 1$. So \eqref{Multiplication P_delta avec E} is bijective, and 
\begin{equation}
    \pr{\mc{P}^{\delta}(z)\mc{E}(z)}^{-1}=I+\sum_{n=1}^{+\infty}(-1)^nB(z)^n.
\end{equation}
Since $\mc{E}(z)$ is bijective, so is $\mc{P}^\delta(z)$. More precisely, 
\begin{equation}\label{Expression inverse du problème de Grushin perturbé}
    \begin{split}
        \mc{P}^{\delta}&(z)^{-1}=\mc{E}(z)\pr{\mc{P}^{\delta}(z)\mc{E}(z)}^{-1}\\
        &=\begin{pmatrix}
                E(I+\delta Q_NE(z))^{-1}& E_+(z)-E(z)(I+\delta Q_NE(z))^{-1}\delta Q_NE_+(z)\\
                E_-(z)(I+\delta Q_NE(z))^{-1}&E_{-+}(z)-E_-(z)(I+\delta Q_NE(z))^{-1}\delta Q_NE_+(z)
        \end{pmatrix}\\
        &=:\begin{pmatrix}
            E^{\delta}(z)&E_+^{\delta}(z)\\
            E_-^{\delta}(z)&E_{-+}^{\delta}(z)
        \end{pmatrix}=:\mc{E}^{\delta}(z)\m.
    \end{split}
\end{equation}
Let us estimate the entries of $\mc{E}^\delta(z)$. For this, notice first that by \eqref{Norme des opérateurs de l'inverse au pb de Grushin} and \eqref{Contrainte sur delta, alpha et la norme de la matrice aléatoire},
\begin{equation*}
    \norme{(I+\delta Q_NE(z))^{-1}}\leq \dfrac{1}{1-\delta\norme{Q_NE(z)}}\leq \dfrac{1}{1-\delta\norme{Q_N}\norme{E(z)}}\leq 2.
\end{equation*}
Therefore, from \eqref{Norme des opérateurs de l'inverse au pb de Grushin}, \eqref{Hypothèse sur delta matrice générale} and \eqref{Expression inverse du problème de Grushin perturbé},
\begin{equation}
    \begin{split}
        &\norme{E^{\delta}(z)} = \norme{E(I+\delta Q_NE(z))^{-1}}\leq 2\alpha^{-1/2}\\
        &\norme{E_+^{\delta}(z)} = \norme{E_+(z)-E(z)(I+\delta Q_NE(z))^{-1}\delta Q_NE_+(z)}\leq 1+2\cdot\frac{1}{2}=2\\
        &\norme{E_-^{\delta}(z)} = \norme{E_-(z)(I+\delta Q_NE(z))^{-1}}\leq 2\\
        &\norme{E_{-+}^{\delta}(z)} = \norme{E_{-+}(z)-E_-(z)(I+\delta Q_NE(z))^{-1}\delta Q_NE_+(z)}\leq 2\alpha^{1/2}.
    \end{split}
\label{Estimée norme des opérateurs de l'inverse de l'opérateur perturbé}
\end{equation}

\mm Similarly as for \eqref{Utilisation de la conséquence de Schur opé non perturbé},  
\begin{equation}
    \log\bb{\det\pr{\prescript{\delta}{}{\widetilde{p}}_N-z}}=\log\bb{\det\pr{\mc{P}^{\delta}(z)}}+\log\bb{\det\pr{E_0^{\delta}(z)}}.
\label{Décomposition log-det de l'opérateur souhaité}
\end{equation}
Notice that $\delta\mapsto\mc{P}^\delta(z)^*\mc{P}^\delta(z)$ is of class $\mathscr{C}^1$. Under the assumption of bijectiveness of $\mc{P}^\delta(z)$, we have by a direct calculation that
\begin{equation}
    \dfrac{d}{d\delta}\Bigl[\log\det\pr{\mc{P}^\delta(z)^*\mc{P}^\delta(z)}\Bigl]=2\Re\pr{\mathrm{tr}\pr{E^{\delta}(z)Q_N}}.
\label{Dérivée de log-det en fonction de la trace}
\end{equation}
Since $\det\pr{\mc{P}^\delta(z)^*\mc{P}^\delta(z)}=\bb{\det\pr{\mc{P}^{\delta}(z)}}^2$, integrating \eqref{Dérivée de log-det en fonction de la trace} provides
\begin{align*}
    \bb{\log\bb{\det\pr{\mc{P}^{\delta}(z)}}-\log\bb{\det\pr{\mc{P}(z)}}}=\bb{\int_{0}^{\delta} \Re\Bigl(\mathrm{tr}\pr{E^{\tau}(z)Q_N}\Bigl)d\tau}.
\end{align*}
From the equality $\mc{E}(z)^{-1}\mc{E}^{\delta}(z)=\mc{P}(z)\mc{P}^{\delta}(z)^{-1}$, we have
\begin{equation}
    \bb{\log\bb{\det\pr{\mc{P}^{\delta}(z)}}-\log\bb{\det\pr{\mc{P}(z)}}}=\bb{\log\bb{\det\pr{\mc{E}^{\delta}(z)}}-\log\bb{\det\pr{\mc{E}(z)}}}.
\label{Diff des P égal diff des E}
\end{equation}

Hence, the combination of the trace norm properties and \eqref{Norme des opérateurs de l'inverse au pb de Grushin} leads to
\begin{equation}
    \begin{split}
        \bb{\log\bb{\det\pr{\mc{E}^{\delta}(z)}}-\log\bb{\det\pr{\mc{E}(z)}}}&=\bb{\int_{0}^{\delta} \Re\Bigl(\mathrm{tr}\pr{E^{\tau}(z)Q_N}\Bigl)d\tau}\\
        &\leq \int_{0}^{\delta}\mathrm{tr}\Bigl(\bb{E^{\tau}(z)Q_N}\Bigl)d\tau\\
        &\leq \delta\ \alpha^{-1/2}
        \norme{Q_N}_{\text{tr}}\\
        &\leq\mc{O}(\delta\alpha^{-1/2}N^{d}\norme{Q_N}),
    \end{split}
\label{Estimée de la différenece des log-det de E rond}
\end{equation}
where the last inequality comes from the fact that $\norme{Q}_{\mathrm{tr}}\leq N^{d}\norme{Q}$ for all $Q\in \C^{N^d\times N^d}$. So, using \eqref{log-det matrice pb de Grushin non perturbé}, \eqref{Apparition de epsilon_0 dans la proba de la norme de la matrice aléatoire}, \eqref{Majoration de la norme de la matrice aléatoire}, \eqref{Diff des P égal diff des E} and \eqref{Estimée de la différenece des log-det de E rond} provides
\begin{equation}
    \begin{split}
        \log\pr{ \bb{\det\pr{\mc{P}^\delta(z)}}} &= N^d\biggl(\int_{\T^{2d}}\log\pr{\bb{p(\rho)-z}}d\rho+\mc{O}\pr{N^{-\eta{\kappa_1}}\bb{\log(\alpha)}}\\
        &+\mc{O}\pr{\alpha^{\kappa_2}\bb{\log(\alpha)}} +\mc{O}\pr{N^{-\varrho\eta}\alpha^{-1}}+\mc{O}\pr{\delta\alpha^{-1/2}N^{\kappa_3+\varepsilon_0}}\biggl)
    \end{split}
\label{Estimée finale log-det pb de Grushin perturbé}
\end{equation}
with probability $\geq 1-\mc{O}(N^{-\varepsilon_0})$. \\

\m\m\m\m For the estimate of \eqref{Estimée log-det supérieure, cas gaussien -potentiel log} to be complete, it remains, in view of \eqref{Décomposition log-det de l'opérateur souhaité} and \eqref{Estimée finale log-det pb de Grushin perturbé}, to estimate $\log\bb{\det\pr{E_{-+}^{\delta}(z)}}$. By \eqref{Estimée norme des opérateurs de l'inverse de l'opérateur perturbé} and \eqref{Nombre de valeurs propres Problème de Grushin}, we have
\begin{equation}
    \log\bb{\det\pr{E_{-+}^{\delta}(z)}}\leq\mc{O}\pr{\mathfrak{m}\left|\log(\alpha)\right|}
\label{Estimée finale log-det de l'opérateur E_0 perturbé - majoration}
\end{equation}
with probability $\geq 1-\mc{O}(N^{-\varepsilon_0})$.\\
\m\\
Combining \eqref{Décomposition log-det de l'opérateur souhaité}, \eqref{Estimée finale log-det pb de Grushin perturbé}, \eqref{Estimée finale log-det de l'opérateur E_0 perturbé - majoration} and \eqref{Nombre de valeurs propres Problème de Grushin}, we obtain 
\begin{equation}
    \begin{split}
        \log\bb{\det\Bigl(\prescript{\delta}{}{\widetilde{p}}_N-z\Bigl)}&\leq N^d\biggl(\int_{\T^{2d}}\log\pr{\bb{p(\rho)-z}}d\rho+\mc{O}\pr{N^{-\eta{\kappa_1}}\bb{\log(\alpha)}}\\
        &\m\m\m\m\m\m+\mc{O}\pr{\alpha^{\kappa_2}\bb{\log(\alpha)}} +\mc{O}\pr{N^{-\varrho\eta}\alpha^{-1}}+\mc{O}\pr{\delta\alpha^{-1/2}N^{\kappa_3+\varepsilon_0}}\biggl)
    \end{split}
\label{Estimée log-det supérieure, cas gaussien}
\end{equation}
with probability $\geq 1-\mc{O}(N^{-\varepsilon_0})$, which gives \eqref{Estimée log-det supérieure, cas gaussien -potentiel log}.\\
\m\\

\mm 2. We now turn to the proof of \eqref{Estimée log-det inférieure, cas gaussien -potentiel log}. From Proposition \ref{Chris and Zworski adapté} and \eqref{Ecriture de delta en décroissance polynomiale}, we have
\begin{equation*}
    \norme{\delta^{-1}\pr{\widetilde{p}_N-z}}=\mc{O}(N^{\kappa_3+\delta_0}).
\end{equation*}
From the second point of Assumption \ref{Hypothèse théorème final - matrice aléatoire} on $Q_N$ (that is, the anti-concentration bound), there exists $\beta>0$ depending on $\kappa_3+\delta_0$, such that
\begin{equation}
    \begin{split}
        \Prob{s_{N^d}\pr{\widetilde{p}_N-z+\delta Q_N}\leq \delta N^{-\beta}} = \Prob{s_{N^d}\pr{\delta^{-1}(\widetilde{p}_N-z)+Q_N}\leq N^{-\beta}}=o(1).
    \end{split}
\label{Utilisation de l'hypothèse d'anti-concentration}
\end{equation}
To shorten the notations, we will write 
\begin{equation}\label{Notation de la mtrice d'intérêt perturbée régularisée - z}
    \prescript{\delta}{}{}\widetilde{p}_N(z):=\prescript{\delta}{}{}\widetilde{p}_N-z
\end{equation}
and we recall that in that case, the number $\mathfrak{m}$ of singular values of $\prescript{\delta}{}{}\widetilde{p}_N(z)$ smaller than $\alpha$ is given in \eqref{Nombre de valeurs propres Problème de Grushin}, uniformly with respect to $z\in \mc{K}$.\\

Suppose first that $E_{-+}^\delta(z)$ is bijective. Then so is $\prescript{\delta}{}{}\widetilde{p}_N(z)$ and, from \eqref{Relations avec complément de Schur} and the multiplicative Ky-Fan inequalities (see \cite{Gohberg_Krein}), we have
\begin{equation*}
    \begin{split}
        s_\mathfrak{m}\pr{E_{-+}^\delta(z)} &= \frac{1}{s_1\pr{E^\delta_{-+}(z)^{-1}}}=\frac{1}{s_1\pr{-R_-\prescript{\delta}{}{}\widetilde{p}_N(z)^{-1}R_+}}\\
        &\geq \frac{1}{\norme{R_-}\norme{R_+}s_1\pr{\prescript{\delta}{}{}\widetilde{p}_N(z)^{-1}}}\\
        &\geq \frac{s_{N^d}\pr{\prescript{\delta}{}{}\widetilde{p}_N(z)}}{\norme{R_-}\norme{R_+}}.
    \end{split}
\end{equation*}
Yet, since, $\norme{R_-}\leq 1$ and $\norme{R_+}\leq 1$, this previous inequality implies
\begin{equation}\label{Inégalité avec les valeurs singulières}
    s_{N^d}\pr{\prescript{\delta}{}{}\widetilde{p}_N(z)}\leq s_\mathfrak{m}\pr{E_0^\delta(z)}.
\end{equation}
Using a denseness argument and the continuity of $s_{N^d}$ and $s_\mathfrak{m}$, we get that \eqref{Inégalité avec les valeurs singulières} still holds even if $E_{-+}^\delta(z)$ is not bijective.\\

It can therefore be deduced from \eqref{Utilisation de l'hypothèse d'anti-concentration} and \eqref{Inégalité avec les valeurs singulières} that, 
\begin{equation}\label{Majoration pour des estimées de proba avec le problème de Grushin}
    \Prob{s_\mathfrak{m}\Bigl(E_{-+}^\delta(z)\Bigl)\leq \delta N^{-\beta}}\leq \Prob{s_{N^d}\Bigl(\prescript{\delta}{}{}\widetilde{p}_N(z)\Bigl)\leq \delta N^{-\beta}}= o(1).
\end{equation}

In the case $s_\mathfrak{m}\pr{E_0^\delta(z)}> \delta N^{-\beta}$, 
\begin{equation}
    \begin{split}
        \log\bb{\det\pr{E_{-+}^\delta(z)}} &= \sum_{j=1}^\mathfrak{m}\log\pr{s_j(E_{-+}^\delta(z))}\\
        &\geq \mathfrak{m}\log\pr{s_\mathfrak{m}(E_{-+}^\delta(z))}\\
        &\geq \mathfrak{m}\log(\delta N^{-\beta})
    \end{split}
\label{Estimée finale log-det de l'opérateur E_0 perturbé - minoration}
\end{equation}
and thanks to the union bound, from \eqref{Apparition de epsilon_0 dans la proba de la norme de la matrice aléatoire} and \eqref{Majoration pour des estimées de proba avec le problème de Grushin}, this event happens with probability
\begin{equation}
    \begin{split}
        \mathbb{P}\Bigl(\acc{s_\mathfrak{m}\pr{E_{-+}^\delta(z)}> \delta N^{-\beta}}&\cap \acc{\norme{Q_N}\leq CN^{\kappa_3+\varepsilon_0}}\Bigl)\\
        &\geq 1-o(1)-\mc{O}(N^{-\varepsilon_0})=1-o(1).
    \end{split}
\label{Proba minoration de log det de E_0 delta}
\end{equation}
Hence, gathering \eqref{Décomposition log-det de l'opérateur souhaité}, \eqref{Estimée finale log-det pb de Grushin perturbé}, \eqref{Estimée finale log-det de l'opérateur E_0 perturbé - minoration} and \eqref{Proba minoration de log det de E_0 delta}, we have
\begin{equation}
    \begin{split}
        \log\bb{\det(\prescript{\delta}{}{}\widetilde{p}_N-z)} &\geq N^d\biggl(\int_{\T^{2d}}\log\pr{\bb{p(\rho)-z}}d\rho+\mc{O}\pr{N^{-\eta{\kappa_1}}\bb{\log(\alpha)}}\\
        &+\mc{O}\pr{\alpha^{\kappa_2}\bb{\log(\alpha)}} +\mc{O}\pr{N^{-\varrho\eta}\alpha^{-1}}+\mc{O}\pr{\mathfrak{m}N^{-d}\log(\delta N^{-\beta})}\biggl)
    \end{split}
\label{Estimée log-det inférieure, cas gaussien}
\end{equation}
with probability $\geq 1-o(1)$, wich remembering of the expression \eqref{Nombre de valeurs propres Problème de Grushin} of $\mathfrak{m}$, leads to \eqref{Estimée log-det inférieure, cas gaussien -potentiel log}.
\end{proof}

\subsection{Comparison of $\phi_{\mu_N}$ and $\phi_{\widetilde{\mu}_N}$}\label{Comparaison entre les potentiels loga - 2}

Recall \eqref{Définition de la matrice avec la nouvelle procédure de quantif}, Notations \ref{Notations pour les matrices des pseudo} and the setting described at the beginning of Section \ref{Principal setting}. Let us start by giving a useful decomposition of $M_N(p)-\widetilde{p}_N$.

\begin{prop}\label{Expression P_N en fonction de P_N tilde}
     Let $\eta,\varrho,p,\widetilde{p}$ be given at the beginning of Section \ref{Principal setting}. For $N\gg 1$, there exist $\vartheta>0$ (independent of $N$) and two matrices $T_N$ and $B_N$ such that 
    \begin{itemize}
        \item $M_N(p)=\widetilde{p}_N+T_N+N^{-\varrho\eta/2}B_N$,
        \item $\norme{B_N}=\mc{O}(1)$,
        \item $\mathrm{rank}(T_N)=\mc{O}(N^{d-\vartheta})$, 
        \item $\norme{T_N}=\mc{O}(1)$.
    \end{itemize} 
\end{prop}

\begin{rem}
    This Proposition is the key step for the proof of Corollary \ref{Corollaire du théorème final}. In fact, if $\vartheta>0$, $B_N$ and $T_N$ are given in Proposition \ref{Expression P_N en fonction de P_N tilde}, and if $R_N$ satisfies \eqref{Hypothèse matrice de perturbation} with $0< \kappa_4\leq d$, then 
    \begin{itemize}
        \item $M_N(p)+T_N=\widetilde{p}_N+\widetilde{T}_N+N^{-\varrho\eta/2}B_N$, with $\widetilde{T}_N=T_N+R_N$,
        \item $\norme{B_N}=\mc{O}(1)$,
        \item $\mathrm{rank}(\widetilde{T}_N)=\mc{O}(N^{d-\vartheta_0})$, with $\vartheta_0=\min(\vartheta,\kappa_4)>0$, 
        \item $\norme{\widetilde{T}_N}=\mc{O}(1)$.
    \end{itemize} 
    Thus, considering $M_N(p)$ or $M_N(p)+R_N$ in this case won't change the rest of the proof of Theorem \ref{Théorème final}.
\end{rem}

\begin{proof}[Proof of Proposition \ref{Expression P_N en fonction de P_N tilde}]
    1. We begin by introducing a new symbol $\widehat{p}$ such that 
    \begin{equation}\label{Egalité entre quantif standard de p tilde et quantif de Weyl de p chapeau}
        \mathrm{Op}_{h}(\widetilde{p})=\mathrm{Op}_h^w(\widehat{p})
    \end{equation}
    (see \cite[Theorem 4.13]{Zworski} or \cite[Theorem 2.7.1]{martinez}). Indeed, by \cite{Zworski,martinez}, 
    \begin{equation*}
        \widehat{p}(x,\xi)=e^{-\frac{ih}{2}\ps{D_x}{D_\xi}}\widetilde{p}(x,\xi)\in S_{\eta,0}(1),
    \end{equation*}
    which can also be written as
    \begin{equation}
        \widehat{p}(x,\xi)=\frac{1}{(2\pi h)^d}\int_{\R^d}\int_{\R^d}e^{\frac{i}{h}\ps{x_1}{\xi_1}}\widetilde{p}(x+x_1,\xi+\xi_1)dx_1d\xi_1.
    \label{Expression de p chapeau}
    \end{equation}
    In particular, thanks to \eqref{Expression de p chapeau}, we obtain the $\Z^{2d}$ periodicity of $\widehat{p}$ from that of $\widetilde{p}$, which ensures that $\widehat{p}\in S_{\eta,0}(1,1,1)$. Furthermore, the exponential operator gives the following first order Borel development 
    \begin{equation}\label{Dvpt de Borel de p chapeau à l'ordre 1}
        \widehat{p}=\widetilde{p}+h^{1-\eta}\widetilde{r}
    \end{equation}
    with $\widetilde{r}\in S_{\eta,0}(1)$. Notice that in Notations \ref{Notations pour les matrices des pseudo}, equality \eqref{Egalité entre quantif standard de p tilde et quantif de Weyl de p chapeau} provides
    \begin{equation}\label{Egalité entre p tilde N 1 et p chapeau N}
        \widetilde{p}_N^1=\widehat{p}_N.
    \end{equation}
    To prove the result, consider the splitting 
    \begin{equation}
        M_N(p)=\Bigl(M_N(p)-M_N(\widetilde{p})\Bigl)+\Bigl(M_N(\widetilde{p})-\widetilde{p}^1_N\Bigl)+\Bigl(\widehat{p}_N-\widetilde{p}_N\Bigl)+\widetilde{p}_N
    \label{Expression P_N en fonction de P_N tilde - 0}
    \end{equation}
    acting on $\C^{N^d}$, as described in \eqref{Définition de la matrice avec la nouvelle procédure de quantif} and in Notations \ref{Notations pour les matrices des pseudo}.\\

    \mm 2. We first consider the term $\widehat{p}_N-\widetilde{p}_N$ in \eqref{Expression P_N en fonction de P_N tilde - 0}. Taking the Weyl quantization of \eqref{Dvpt de Borel de p chapeau à l'ordre 1} gives 
    \begin{equation}
        \widehat{p}_N-\widetilde{p}_N=N^{-(1-\eta)}\widetilde{r}_N,
    \label{Expression P_N en fonction de P_N tilde - 1 ter}
    \end{equation}
     where, from Proposition \ref{Chris and Zworski adapté}, $\norme{\widetilde{r}_N}=\mc{O}(1)$.\\

     \mm 3. Next consider the term $M_N(\widetilde{p})-\widetilde{p}^1_N$. From Lemma \ref{Coefficients of the Toeplitz matrix}, Remark \ref{Expression des coefficients des matrices ancienne quantif, pour les valeurs de t qui nous intéressent}, \eqref{Coef infinite matrix quantization} and \eqref{Définition de la matrice avec la nouvelle procédure de quantif}, we have that for all $s,j\in \disc{1}{N}^d$,
    \begin{equation}
        M_N(\widetilde{p})_{s,j}-\pr{\widetilde{p}^1_N}_{s,j}=\widetilde{p}_{s-j}\pr{\frac{s}{N}}-\sum_{n,r\in \Z^d}c_{n,j-s+rN}(\widetilde{p})e^{\frac{2i\pi}{N}\ps{n}{s}}.
    \label{P atchek moins P hat - 1}
    \end{equation}
    By \eqref{Coeffs de Fourier en les deux variables} and \eqref{Developpement en série de Fourier par rapport à xi - version mollifiée}, for all $n,m\in \Z^d$,
    \begin{equation*}
        \begin{split}
            c_{n,m}(\widetilde{p}) &= \int_{\T^{2d}}\widetilde{p}(x,\xi)e^{-2i\pi\pr{\ps{x}{n}+\ps{\xi}{m}}}dxd\xi\\
            &=\int_{\T^d}\widetilde{p}_m(x)e^{-2i\pi\ps{x}{n}}dx=:c_n(\widetilde{p}_{m}).
        \end{split}
    \end{equation*}
    Furthermore, for all $m\in \Z^d$, since $\widetilde{p}_{m}$ is $\Z^d$ periodic, we also have
    \begin{equation*}
        \widetilde{p}_{m}(x)=\sum_{n\in \Z^d}c_n(\widetilde{p}_{m})e^{2i\pi\ps{n}{x}},
    \end{equation*}
    so \eqref{P atchek moins P hat - 1} becomes
    \begin{equation}
        M_N(\widetilde{p})_{s,j}-\pr{\widetilde{p}^1_N}_{s,j}=-\sum_{r\in \Z^d\setminus\{0\}}\widetilde{p}_{s-j+rN}\pr{\frac{s}{N}}.
    \label{P atchek moins P hat - 2}
    \end{equation}
    Following the same proof as for Lemma \ref{Décroissance poly des coef de Fourier sur le tore biscornu}, and using that $\widetilde{p}\in S_{\eta,0}(1,1,1)$, we obtain that for all $\nu\in \Z^d$, $x\in [0,1]^d$, $K\in \N$,
    \begin{equation}
        \widetilde{p}_{\nu}(x)=\mc{O}_K(1)\jap{\nu}^{-K}\sum_{\bb{\beta}\leq K}\norme{D_{\xi}^\beta\widetilde{p}(x,\cdot)}_{L^1(\T^d)}=\mc{O}_K(1)\!\jap{\nu}^{-K}
    \label{P atchek moins P hat - 3}
    \end{equation}
    where $\mc{O}_K(1)$ is uniform with respect to $x$.\\

    Let $0<a_N< N$ depending on $N$, to be chosen later on.
    Using Lemma \ref{Lemme nul 2} in the Appendix, for all $r\in \Z^d\setminus\{0\}$, $s,j\in \disc{1}{N}^d$ such that $s-j\in \disc{-(N-a_N)}{N-a_N}^d$, and $K\in \N$, we have
    \begin{equation*}
        \widetilde{p}_{s-j+rN}\pr{\frac{s}{N}}=\mc{O}_K(a_N^{-K})\bb{r}^{-K}.
    \end{equation*}
    Then, for $K\geq d+1$, we obtain from \eqref{P atchek moins P hat - 2} and the previous estimate that 
    \begin{equation}
        M_N(\widetilde{p})_{s,j}-\pr{\widetilde{p}^1_N}_{s,j}=\mc{O}_K(a_N^{-K}),
    \label{Estimée matrice A}
    \end{equation}
    when $s-j\in \disc{-(N-a_N)}{N-a_N}^d$. We are led to write
    \begin{equation}
        M_N(\widetilde{p})-\widetilde{p}^1_N=A+B,
    \label{P atchek moins P hat bis - 1}
    \end{equation}
    where 
    \begin{equation*}
        A:=\pr{\pr{M_N(\widetilde{p})_{s,j}-\pr{\widetilde{p}^1_N}_{s,j}}\mathbb{1}_{\disc{-(N-a_N)}{N-a_N}^d}(s-j)}_{s,j}
    \end{equation*}
    and 
    \begin{equation}\label{P atchek moins P hat bis - 1 bis}
        B:=M_N(\widetilde{p})-\widetilde{p}^1_N-A.
    \end{equation}
    In view of \eqref{Estimée matrice A}, we obtain 
    \begin{equation}
        \norme{A}\leq \norme{A}_{HS}=\mc{O}_K\pr{N^{d}a_N^{-K}}.
    \label{P atchek moins P hat bis - 2}
    \end{equation}
    \mm Furthermore, we notice that the number of non-zero coefficients of $B$ is at most
    \begin{equation}
        \#\pr{\disc{-(N-1)}{N-1}^d}-\#\pr{\disc{-(N-a_N)}{N-a_N}^d}=\mc{O}\pr{N^{d-1}a_N},
    \label{P atchek moins P hat - 3 bis}
    \end{equation}
    so given \eqref{P atchek moins P hat bis - 1 bis}, this implies that $\mathrm{rank}(B)=\mc{O}(N^{d-1}a_N)$. For all $r\in \Z^d\setminus \{0\}$, $s,j\in\disc{1}{N}^d$, 
    \begin{equation}
        \norme{s-j+rN}\geq N\norme{r}-\norme{s-j}\geq N\pr{\norme{r}-\sqrt{d}}\geq \norme{r}-\sqrt{d}.
    \label{P atchek moins P hat - 4}
    \end{equation}
    Then, putting 
    \begin{equation}\label{Valeur de K comme étant d+1}
        K=d+1
    \end{equation}
    in \eqref{P atchek moins P hat - 3}, it follows from \eqref{P atchek moins P hat - 2}, \eqref{P atchek moins P hat bis - 1} and \eqref{P atchek moins P hat - 4}, for all $s,j\in \disc{1}{N}^d$, 
    \begin{equation}
        \begin{split}
            \bb{M_N(\widetilde{p})_{s,j}-\pr{\widetilde{p}^1_N}_{s,j}} &\leq  \sum_{\substack{r\in \Z^d\setminus\{0\}\\ \bb{r}\leq \sqrt{d}}}\bb{\widetilde{p}_{s-j+rN}\pr{\frac{s}{N}}}+\sum_{\substack{r\in \Z^d\setminus\{0\}\\ \bb{r}> \sqrt{d}}}\bb{\widetilde{p}_{s-j+rN}\pr{\frac{s}{N}}}\\
            &\leq \sum_{\substack{r\in \Z^d\setminus\{0\}\\ \bb{r}\leq \sqrt{d}}}\mc{O}_d(1)+\mc{O}_d(1)\underbrace{\sum_{\substack{r\in \Z^d\setminus\{0\}\\ \bb{r}> \sqrt{d}}}\pr{1+(\bb{r}-\sqrt{d})^2}^{-(d+1)/2}}_{<+\infty}\\
            &=\mc{O}_d(1).
        \end{split}
    \label{P atchek moins P hat - 5}
    \end{equation}
    In particular, we deduce from \eqref{P atchek moins P hat bis - 1 bis} and \eqref{P atchek moins P hat - 5} that for all $s,j\in \disc{1}{N}^d$ 
    \begin{equation}\label{P atchek moins P hat - 5 bis}
        \bb{B_{s,j}}\leq \mc{O}_d(1).
    \end{equation}
    Finally, using \eqref{P atchek moins P hat - 3 bis}, \eqref{Valeur de K comme étant d+1} and \eqref{P atchek moins P hat - 5 bis} together leads to
    \begin{equation}
        \norme{B}\leq \norme{B}_{\mathrm{HS}}=\mc{O}(N^{d-1}a_N).
    \label{P atchek moins P hat bis - 3}
    \end{equation}

    We choose $a_N$ such that $a_N=o(N)$ and $N^da_N^{-K}\to 0$ when $N\to +\infty$. For example, we can set $a_N=N^{(K+d)/2K}$.
    In that case, from \eqref{P atchek moins P hat bis - 1}, \eqref{P atchek moins P hat bis - 2} and \eqref{P atchek moins P hat bis - 3}, we obtain
    \begin{equation}
        M_N(\widetilde{p})-\widetilde{p}^1_N=N^{-1/2}A_N+B,
    \label{Expression P_N en fonction de P_N tilde - 3 bis}
    \end{equation}
    with 
    \begin{equation}\label{Norme de la matrice A_N - 1}
        \norme{A_N}=\mc{O}(1)
    \end{equation}
    and 
    \begin{equation}\label{Norme et rang de la matrice B - 1}
        \mathrm{rank}(B)=\mc{O}\pr{N^{d-\frac{1}{2d+2}}},\mm \norme{B}=\mc{O}\pr{N^{d-\frac{1}{2d+2}}}.
    \end{equation}
     
    \mm 4. Now, we consider the term $M_N(p)-M_N(\widetilde{p})$. For this, we fix $0<a<1$ to be chosen later on and introduce the truncated symbols
    \begin{equation}\label{Symboles tronqués}
        \begin{split}
            p^N(x,\xi)&:= \sum_{\nu\in \disc{-N^a}{N^a}^d} p_\nu(x)e^{2i\pi\ps{\nu}{\xi}},\\
            \widetilde{p}^N(x,\xi)&:= \sum_{\nu\in \disc{-N^a}{N^a}^d} \widetilde{p}_\nu(x)e^{2i\pi\ps{\nu}{\xi}}.
        \end{split}
    \end{equation}
    Likewise, we decompose the term to be estimated in the following way:
    \begin{equation}
            M_N(p)-M_N(\widetilde{p})
            =I+II+III,
    \label{Expression P_N en fonction de P_N tilde - 1}
    \end{equation}
    where 
    \begin{equation}\label{Expression P_N en fonction de P_N tilde - 1 bis}
        \begin{split}
            I&=M_N(p)-M_N\pr{p^N}\\
            II&=M_N\pr{p^N}-M_N\pr{\widetilde{p}^N}\\
            III&=M_N\pr{\widetilde{p}^N}-M_N(\widetilde{p}).
        \end{split}
    \end{equation}
    \m\\
    \mm 4.1) We start by estimating the term $II$ of this decomposition. By \eqref{Expression P_N en fonction de P_N tilde - 1 bis}, \eqref{Coef infinite matrix quantization} and \eqref{Définition de la matrice avec la nouvelle procédure de quantif}, we find that for all $s,j\in \disc{1}{N}^d$,
    \begin{equation}\label{Ensemble des indices, middle term}
        M_N\pr{p^N}_{s,j}-M_N\pr{\widetilde{p}^N}_{s,j}=\mathbb{1}_{\disc{-N^a}{N^a}^d}(s-j)\croch{p_{s-j}\pr{\frac{s}{N}}-\widetilde{p}_{s-j}\pr{\frac{s}{N}}}.
    \end{equation}
    Let $U_j$ denote the open sets of $[0,1]^d$ defined in \eqref{Propriété des ouverts dans la def des fonctions de l'espace d'intérêt}. Then for all $m\in \Z^d$ and $j$, we have that
    \begin{equation}
        p_m\in \mathscr{C}^{0,\varrho}\pr{\overline{U_j}}.
    \label{Les coef de Fourier sont Hölder par morceaux}
    \end{equation}
    Also recall the definition of $\psi$ satisfying \eqref{Proprios de la fonction mollifiante - 1} and \eqref{Proprios de la fonction mollifiante - 2}. Then, notice from \eqref{Interversion convolution-intégrale pour l'expression of p_m tilde} that for all $x\in [0,1]^d$, $m\in \Z^d$,
    \begin{equation}
        \begin{split}
            p_m(x)-\widetilde{p}_m(x) &= p_m(x)-\frac{1}{h^{d\eta}}\int_{\T^{d}}\pr{\int_{\R^d}p(z,\xi)\psi\pr{\frac{x-z}{h^\eta}}dz}e^{-2i\pi\ps{m}{\xi}}d\xi\\
            &= p_m(x)-\int_{\R^d}p_m(x-h^\eta y)\psi(y)dy\\
            &= \int_{\mathrm{supp}(\psi)}\Bigl[p_m(x)-p_m(x-h^\eta y)\Bigl]\psi(y)dy.
        \end{split}
    \label{Différence entre les coef de Fourier de p et de p tilde}
    \end{equation}
    If $x\in [0,1]^d\setminus \Bigl(\mathscr{U}_p+B(0,2h^\eta)\Bigl)$, for all $y\in \mathrm{supp}(\psi)\subset B(0,1)$ (see \eqref{Proprios de la fonction mollifiante - 1}), $x$ and $x-h^\eta y$ are in the same open set $U_j$ which in view of \eqref{Les coef de Fourier sont Hölder par morceaux} and \eqref{Différence entre les coef de Fourier de p et de p tilde}, leads to
    \begin{equation}\label{Estimée p_m - p_m_tilde}
        p_m(x)-\widetilde{p}_m(x)=\mc{O}(h^{\varrho\eta})
    \end{equation}
    uniformly with respect to $x$. \\
    \mm Similarly to \eqref{P atchek moins P hat bis - 1}, let us write 
    \begin{equation}\label{Décomposition matrices sous la forme A+B}
        M_N\pr{p^N}-M_N\pr{\widetilde{p}^N}=A+B
    \end{equation}
    where 
    \begin{equation*}
        B:=\pr{\pr{M_N\pr{p^N}_{s,j}-M_N\pr{\widetilde{p}^N}_{s,j}}\mathbb{1}_{\mathscr{U}_p+B(0,2h^\eta)}\pr{\frac{s}{N}}}_{s,j}
    \end{equation*}
    and
    \begin{equation}\label{Décomposition matrices sous la forme A+B - Expression de B}
        A:= M_N\pr{p^N}-M_N\pr{\widetilde{p}^N}-B.
    \end{equation}
    \mm Thanks to \eqref{Coef infinite matrix quantization}, \eqref{Définition de la matrice avec la nouvelle procédure de quantif}, \eqref{Hypothèse expression de h comme l'inverse de N - 2}, \eqref{Estimée p_m - p_m_tilde} and \eqref{Décomposition matrices sous la forme A+B - Expression de B}, for all $s,j\in \disc{1}{N}^d$,  
    \begin{equation}
        A_{s,j}= \mc{O}(N^{-\varrho\eta}).
    \label{Expression P_N en fonction de P_N tilde - 2}
    \end{equation}
    Remembering \eqref{Ensemble des indices, middle term}, for all $s,j\in \disc{1}{N}^d$ if $s-j\notin\disc{-N^a}{N^a}^d$, then
    \begin{equation}\label{Condition de nullité des coefficients de A}
        A_{s,j}=0.
    \end{equation}
    So, from \eqref{Condition de nullité des coefficients de A}, \eqref{Expression P_N en fonction de P_N tilde - 2} and Lemma \ref{Norme matrice B_N en dimension d}, we get that
    \begin{equation}\label{Norme matrice A, middle term}
        \norme{A}=\mc{O}\pr{N^{\frac{ad}{2}-\varrho\eta}}.
    \end{equation}
    \mm Moreover, in view of Assumption \ref{Hypothèse théorème final - volume singularités}, 
    \begin{equation}
        \#\acc{s\in \disc{1}{N}^d\m \left|\m s/N\in  \mathscr{U}_p+B(0,2h^\eta)\right.}=\mc{O}(N^{d-\eta{\kappa_1}}),
    \label{Expression P_N en fonction de P_N tilde - 3}
    \end{equation}
    and \eqref{Coeffs de Fourier en une seule variable}, \eqref{Proprios de la fonction mollifiante - 1} and \eqref{Différence entre les coef de Fourier de p et de p tilde} give that for all $x\in [0,1]^d$, 
    \begin{equation}\label{Majoration module de p_m - p_m_tilde}
        \bb{p_m(x)-\widetilde{p}_m(x)}\leq \mc{O}(1)
    \end{equation}
    uniformly in $x$. We therefore have from \eqref{Majoration module de p_m - p_m_tilde} that for all $s,j\in \disc{1}{N}^d$,
    \begin{equation}\label{Majoration des coefficients de la matrice B}
        \bb{B_{s,j}}\leq \mc{O}(1).
    \end{equation}
    So, from \eqref{Majoration des coefficients de la matrice B}, \eqref{Ensemble des indices, middle term}, Lemma \ref{Norme matrice B_N en dimension d} and \eqref{Expression P_N en fonction de P_N tilde - 3}, we have,
    \begin{equation}\label{norme et rang de la matrice B - middle term}
        \norme{B}=\mc{O}\pr{N^{\frac{ad}{2}}},\mm\mathrm{rank}(B)= \mc{O}(N^{d-\eta{\kappa_1}}).
    \end{equation}

    \mm 4.2) Next, we turn to the term $I$ of  \eqref{Expression P_N en fonction de P_N tilde - 1}. First, we have for all $s,j\in \disc{1}{N}^d$,
    \begin{equation*}
        M_N(p)_{s,j}-M_N\pr{p^N}_{s,j}=\pr{1-\mathbb{1}_{\disc{-N^a}{N^a}^d}(s-j)}p_{s-j}\pr{\frac{s}{N}}.
    \end{equation*}
    For all $s,j\in \disc{1}{N}^d$, if $s-j\notin\disc{-N^a}{N^a}^d$, $\norme{s-j}\geq N^a$. So by \eqref{Décroissance des coef de Fourier par rapport à l'indice}, for all $k\in \N$, there exists $C_k>0$ independent of $s$ and $j$ such that
    \begin{equation*}
        \bb{p_{s-j}\pr{\frac{s}{N}}}\leq C_kN^{-ak}.
    \end{equation*}
    Thus, for all $k\in \N$, 
    \begin{equation}\label{Expression P_N en fonction de P_N tilde - 6}
        \norme{M_N(p)-M_N\pr{p^N}}=\mc{O}_k(N^{d-ak}).
    \end{equation}
    
    \mm 4.3) The same reasoning but for $\widetilde{p}$ instead of $p$ and the use of \eqref{P atchek moins P hat - 3} leads to 
    \begin{equation}
        \norme{M_N\pr{\widetilde{p}^N}-M_N(\widetilde{p})}=\mc{O}_k(N^{d-ak})
    \label{Expression P_N en fonction de P_N tilde - 7}
    \end{equation}
    for all $k\in \N$.\\

    Finally, we choose $a:=\frac{\varrho\eta}{d}\in \hspace{0.1cm} ]0,1[$ and $k$ large enough. Hence, gathering \eqref{Expression P_N en fonction de P_N tilde - 0},
    \eqref{Expression P_N en fonction de P_N tilde - 1 ter},
    \eqref{Expression P_N en fonction de P_N tilde - 3 bis},
    \eqref{Norme de la matrice A_N - 1}, \eqref{Norme et rang de la matrice B - 1}, 
    \eqref{Expression P_N en fonction de P_N tilde - 1}, 
    \eqref{Décomposition matrices sous la forme A+B}, 
    \eqref{Norme matrice A, middle term},
    \eqref{norme et rang de la matrice B - middle term}, \eqref{Expression P_N en fonction de P_N tilde - 6} and \eqref{Expression P_N en fonction de P_N tilde - 7}, we obtain that 
    \begin{equation}\label{Décomposition de M_N(p) en fonction de p_N_tilde - version nulle}
        M_N(p)=\widetilde{p}_N+N^{-\varrho\eta/2}B_N+T_N
    \end{equation}
    where 
    \begin{itemize}
        \item $\norme{B_N}=\mc{O}(1)$,
        \item $\mathrm{rank}(T_N)=\mc{O}\pr{N^{d-\min\pr{\eta{\kappa_1},\frac{1}{2d+2}}}}$,
        \item $\norme{T_N}=\mc{O}\pr{N^{d-\frac{1}{2d+2}}}$.
    \end{itemize}
    However, in view of \eqref{Décomposition de M_N(p) en fonction de p_N_tilde - version nulle}, Remark \ref{Remarque sur le fait que les matrices via la nouvelle procédure de quantif sont bornées} and \eqref{Majoration de la norme de la matrice de Toeplitz}, we have
    \begin{equation*}
        \norme{T_N}\leq \norme{M_N(p)}+\norme{\widetilde{p}_N}+\mc{O}\pr{N^{-\varrho\eta/2}}\leq \mc{O}(1),
    \end{equation*}
    which leads to the stated result with
    \begin{equation*}
        \vartheta= \min\pr{\eta{\kappa_1},\frac{1}{2d+2}}>0.\qedhere
    \end{equation*}
\end{proof}

\mm We can now give the comparison of $\phi_{\mu_N}$ and $\phi_{\widetilde{\mu}_N}$.
\begin{prop}
    Let $\mc{K}\subset\subset \C$. Let $\phi_{\mu_N}$ be as in \eqref{Expression simplifiée du potentiel logarithmique en fonction du déterminant} and $\phi_{\widetilde{\mu}_N}$ be as in \eqref{Expression simplifiée du potentiel logarithmique en fonction du déterminant - 2}. Let $\beta>0$ be as in Proposition \ref{Différence des potentiels loga entre mu_N tilde et mu}. Then, there exists $C_0>0$ such that for all $z\in \mc{K}$, 
    \begin{equation}
        \begin{split}
            \phi_{\mu_N}(z) &\leq \phi_{\widetilde{\mu}_N}(z)+\mc{O}(N^{-\vartheta}\log(\delta N^{-\beta}))+\mc{O}(N^{-\vartheta})\\
            &+\mc{O}\pr{N^{-\vartheta}\log\pr{1+C_0N^{\beta-\varrho\eta/2}\delta^{-1}}}+\log\pr{1+2C_0N^{-\varrho\eta/2}\alpha^{-1/2}}\\
            &+\mc{O}\pr{\alpha^{{\kappa_2}}\log\pr{1+C_0N^{\beta-\varrho\eta/2}\delta^{-1}}}
        \end{split}
    \label{Estimation potentiel log matrice tronquée par celui de la matrice du symbole périodique - majoration}
    \end{equation}
    with probability $1-o(1)$, and
    \begin{equation}
        \begin{split}
            \phi_{\mu_N}(z)&\geq \phi_{\widetilde{\mu}_N}(z)+\mc{O}(N^{-\vartheta})+\mc{O}(\alpha^{{\kappa_2}})+\mc{O}(\alpha^{-1/2}N^{-\varrho\eta/2})\\
            &+\mc{O}(N^{-\vartheta}\log(\delta N^{-\beta}))+\mc{O}(\alpha^{{\kappa_2}}\log(\delta N^{-\beta}))
        \end{split}
    \label{Estimation potentiel log matrice tronquée par celui de la matrice du symbole périodique - minoration}
    \end{equation}
    with probability $\geq 1-o(1)$.
\label{Différence des potentiels logaritmiques}
\end{prop}

\begin{proof}
    From Proposition \ref{Chris and Zworski adapté}, \eqref{Majoration de la norme de la matrice de Toeplitz}, \eqref{Apparition de epsilon_0 dans la proba de la norme de la matrice aléatoire}, \eqref{Ecriture de delta en décroissance polynomiale} and \eqref{Notation de la mtrice d'intérêt perturbée régularisée}, we know that
    \begin{equation}
       \mathbb{P}\pr{\norme{\prescript{\delta}{}{}\widetilde{p}_N}\leq \mc{O}(1)} \geq 1-\mc{O}(N^{-\varepsilon_0})
    \end{equation}
    where $\mc{O}(1)$ is uniform in $N$. So, in what follows, we can and will assume that  
    \begin{equation}\label{Estimée de la norme de p_N_tilde_delta uniformément en N avec très bonne proba}
        \norme{\prescript{\delta}{}{}\widetilde{p}_N}\leq \mc{O}(1)
    \end{equation}
    uniformly in $N$.\\
    
    The key of the proof is to use the decomposition 
    \begin{equation*}
        M_N(p)=\widetilde{p}_N+T_N+N^{-\varrho\eta/2}B_N
    \end{equation*}
    given by Proposition \ref{Expression P_N en fonction de P_N tilde}, with $T_N$ associated with $\vartheta>0$. Let $z\in \mc{K}$. In view of \eqref{Notation de la mtrice d'intérêt perturbée}, \eqref{Notation de la mtrice d'intérêt perturbée régularisée} and \eqref{Notation de la mtrice d'intérêt perturbée régularisée - z}, we have in particular that 
    \begin{equation}\label{Décomposition de M_N(p)_delta en fonction de p_N_tilde_delta dans la proposition qui donne la décomposition de M_N(p) en fonction de p_N_tilde}
        M_N^\delta(p)-z=\prescript{\delta}{}{}\widetilde{p}_N(z)+T_N+N^{-\varrho\eta/2}B_N.
    \end{equation}
    
    \mm For the rest of the proof, we will write 
    \begin{equation}\label{Rang de la matrice R_N dans la proposition qui donne la décomposition de M_N(p) en fonction de p_N_tilde}
        K:=\mathrm{rank}(T_N).
    \end{equation}
    Denote the following events
    \begin{equation}
        \begin{split}
            A_1 \m:\m&s_{N^d}\Bigl(\prescript{\delta}{}{}\widetilde{p}_N(z)\Bigl)>\delta N^{-\beta}\\
            A_2\m:\m&s_{N^d}\Bigl(\prescript{\delta}{}{}\widetilde{p}_N(z)+T_N\Bigl))>\delta N^{-\beta}\\
            A_3\m:\m &s_{N^d}\Bigl(\prescript{\delta}{}{}\widetilde{p}_N(z)+N^{-\varrho\eta/2}B_N\Bigl)>\delta N^{-\beta}\\
            A_4\m:\m &s_{N^d}\Bigl(\prescript{\delta}{}{}\widetilde{p}_N(z)+T_N+N^{-\varrho\eta/2}B_N\Bigl)>\delta N^{-\beta}.
        \end{split}
    \label{Différence des potentiels logaritmiques - 1}
    \end{equation}
    We also remember from Proposition \ref{Chris and Zworski adapté}, Proposition \ref{Expression P_N en fonction de P_N tilde} and \eqref{Majoration de la norme de la matrice de Toeplitz} that
    \begin{equation}\label{Majoration norme de p_N_tilde et de B_N dans la deuxième proposition pour l'estimée de la différence des potentiels loga}
        \norme{\widetilde{p}_N}, \m\norme{B_N}=\mc{O}(1)\leq C_0
    \end{equation}
     uniformly in $N>0$, for some constant $C_0>0$. From \eqref{Majoration de la norme de la matrice aléatoire} and the second point of Assumption \ref{Hypothèse théorème final - matrice aléatoire} on $Q_N$, we have that for all $j\in \disc{1}{4}$,
     \begin{equation}\label{Proba globale des quatre évenements dans la preuve du théorème final}
         \Prob{A_j}\geq 1-o(1).
     \end{equation}
     Notice here that since we only have finitely many events, we can and will assume that $\beta>0$ is the same for each event. By the union bound, \eqref{Apparition de epsilon_0 dans la proba de la norme de la matrice aléatoire} and \eqref{Proba globale des quatre évenements dans la preuve du théorème final}, we have that 
     \begin{equation}\label{Probabilité via l'union bound}
         \mathbb{P}\pr{\bigcap_{j=1}^4A_j\cap \acc{\norme{Q_N}\leq CN^{\kappa_3+\varepsilon_0}}}\geq 1-o(1),
     \end{equation}
     and in what follows, we assume that the event 
    \begin{equation*}
        \bigcap_{j=1}^4A_j\cap \acc{\norme{Q_N}\leq CN^{\kappa_3+\varepsilon_0}}
    \end{equation*}
    holds.\\
    \mm We recall that in view of \eqref{Rang de la matrice R_N dans la proposition qui donne la décomposition de M_N(p) en fonction de p_N_tilde}, Allahverdiev's theorem (\cite{Gohberg_Krein}) give that for every $N^d\times N^d$ matrix $A$, 
    \begin{align}\label{Allahverdiev - 1}
        & s_{j+K}(A)\leq s_j(A+T_N)\mm \text{for} \mm 1\leq j\leq N^d+1-K ,\\
        \label{Allahverdiev - 2}
        & s_j(A+T_N)\leq s_{j-K}(A) \mm \text{for} \mm K\leq j\leq N^d.
    \end{align}
    Since the singular values are ordered, we also recall from \eqref{Notation pour les valeurs singulières ordonnées bis} that for all $j\in \disc{1}{N^d}$ and every $N^d\times N^d$ matrix $A$,
    \begin{equation}\label{Majoration et minoration des valeurs singulières par la plus grande et la plus petite}
        s_{N^d}(A)\leq s_j(A)\leq s_1(A)=\norme{A}.
    \end{equation}
    
    \mm 1. We begin with the upper bound \eqref{Estimation potentiel log matrice tronquée par celui de la matrice du symbole périodique - majoration}. In view of \eqref{Majoration norme de p_N_tilde et de B_N dans la deuxième proposition pour l'estimée de la différence des potentiels loga} and \eqref{Décomposition de M_N(p)_delta en fonction de p_N_tilde_delta dans la proposition qui donne la décomposition de M_N(p) en fonction de p_N_tilde}, using the additive Ky-Fan inequalities (see \cite[Corollary 2.2]{Gohberg_Krein}), we have  
    \begin{equation*}
        s_j\Bigl(M_N^\delta(p)-z\Bigl)\leq s_j\Bigl(\prescript{\delta}{}{}\widetilde{p}_N(z)+T_N\Bigl)+\frac{C_0}{N^{\varrho\eta/2}}
    \end{equation*}
    for all $j\in \disc{1}{N^d}$. In particular, from \eqref{Expression simplifiée du potentiel logarithmique en fonction du déterminant}, we get
    \begin{equation}\label{Estimation potentiel log matrice tronquée par celui de la matrice du symbole périodique - majoration - 1}
        \phi_{\mu_N}(z)\leq \frac{1}{N^d}\sum_{j=1}^{N^d}\log\pr{s_j\Bigl(\prescript{\delta}{}{}\widetilde{p}_N(z)+T_N\Bigl)+\frac{C_0}{N^{\varrho\eta/2}}}.
    \end{equation}
    Notice that for all $j\in \disc{1}{N^d}$,
    \begin{equation}
        \begin{split}
            \log&\pr{s_j(\prescript{\delta}{}{}\widetilde{p}_N(z)+T_N)+\frac{C_0}{N^{\varrho\eta/2}}}\\
            &\mm\mm\mm=\log\pr{s_j(\prescript{\delta}{}{}\widetilde{p}_N(z)+T_N)}+\log\pr{1+\frac{C_0N^{-\varrho\eta/2}}{s_j(\prescript{\delta}{}{}\widetilde{p}_N(z)+T_N)}}.
        \end{split}
    \label{Majoration et minoration pour la diff des potentiels - 1}
    \end{equation}

    Consider the first term on the right hand side of \eqref{Majoration et minoration pour la diff des potentiels - 1}. By \eqref{Allahverdiev - 2} and change of variables, we get
    \begin{equation}\label{Estimation potentiel log matrice tronquée par celui de la matrice du symbole périodique - majoration - 2}
        \begin{split}
            \sum_{j=1}^{N^d}&\log\pr{s_j(\prescript{\delta}{}{}\widetilde{p}_N(z)+T_N)} \\
            &= \sum_{j=1}^{K}\log\pr{s_j(\prescript{\delta}{}{}\widetilde{p}_N(z)+T_N)}+\sum_{j=K+1}^{N^d}\log\pr{s_j(\prescript{\delta}{}{}\widetilde{p}_N(z)+T_N)}\\
            &\leq \sum_{j=1}^{K}\log\pr{s_j(\prescript{\delta}{}{}\widetilde{p}_N(z)+T_N)}+\sum_{j=K+1}^{N^d}\log\pr{s_{j-K}(\prescript{\delta}{}{}\widetilde{p}_N(z))}\\
            &\leq \sum_{j=1}^{K}\log\pr{s_j(\prescript{\delta}{}{}\widetilde{p}_N(z)+T_N)}+\sum_{j=1}^{N^d-K}\log\pr{s_j(\prescript{\delta}{}{}\widetilde{p}_N(z))}\\
            &\leq  N^d\phi_{\widetilde{\mu}_N}(z)-\sum_{j=N^d-K+1}^{N^d}\log\pr{s_j(\prescript{\delta}{}{}\widetilde{p}_N(z))}+\sum_{j=1}^{K}\log\pr{s_j(\prescript{\delta}{}{}\widetilde{p}_N(z)+T_N)}.
        \end{split}
    \end{equation}
   So, from \eqref{Expression simplifiée du potentiel logarithmique en fonction du déterminant - 2} and \eqref{Majoration et minoration des valeurs singulières par la plus grande et la plus petite}, \eqref{Estimation potentiel log matrice tronquée par celui de la matrice du symbole périodique - majoration - 2} becomes
   \begin{equation}\label{Estimation potentiel log matrice tronquée par celui de la matrice du symbole périodique - majoration - 3}
       \begin{split}
           \sum_{j=1}^{N^d}&\log\pr{s_j(\prescript{\delta}{}{}\widetilde{p}_N(z)+T_N)}\\
            &\leq N^d\phi_{\widetilde{\mu}_N}(z)-K\log\pr{s_{N^d}(\prescript{\delta}{}{}\widetilde{p}_N(z))}+K\log\pr{\norme{\prescript{\delta}{}{}\widetilde{p}_N(z)+T_N}}.
       \end{split}
   \end{equation}
   By \eqref{Estimée de la norme de p_N_tilde_delta uniformément en N avec très bonne proba} and Proposition \ref{Expression P_N en fonction de P_N tilde}, we have
    \begin{equation}\label{Estimation potentiel log matrice tronquée par celui de la matrice du symbole périodique - majoration - 4}
        \norme{\prescript{\delta}{}{}\widetilde{p}_N(z)+T_N}\leq \mc{O}(1).
    \end{equation}
    Since we assumed that the event $A_1$ in \eqref{Différence des potentiels logaritmiques - 1} holds, it follows from \eqref{Estimation potentiel log matrice tronquée par celui de la matrice du symbole périodique - majoration - 3} and \eqref{Estimation potentiel log matrice tronquée par celui de la matrice du symbole périodique - majoration - 4} that
    \begin{equation}
        \sum_{j=1}^{N^d}\log\pr{s_j(\prescript{\delta}{}{}\widetilde{p}_N(z)+T_N)} \leq N^d\phi_{\widetilde{\mu}_N}(z)-K\log(\delta N^{-\beta})+\mc{O}(K).
    \label{Majoration et minoration pour la diff des potentiels - 2}
    \end{equation}
    
    Next, we turn to the second term on the right hand side of \eqref{Majoration et minoration pour la diff des potentiels - 1}. Using \eqref{Allahverdiev - 1} and a change of variables, we have
    \begin{equation}
        \begin{split}
            \sum_{j=1}^{N^d} &\log\pr{1+\frac{C_0N^{-\varrho\eta/2}}{s_j(\prescript{\delta}{}{}\widetilde{p}_N(z)+T_N)}}\\
            &= \sum_{j=1}^{N^d-K} \log\pr{1+\frac{C_0N^{-\varrho\eta/2}}{s_j(\prescript{\delta}{}{}\widetilde{p}_N(z)+T_N)}}+\sum_{j=N^d-K+1}^{N^d} \log\pr{1+\frac{C_0N^{-\varrho\eta/2}}{s_j(\prescript{\delta}{}{}\widetilde{p}_N(z)+T_N)}}\\
            &\leq\sum_{j=K+1}^{N^d} \log\pr{1+\frac{C_0N^{-\varrho\eta/2}}{s_j(\prescript{\delta}{}{}\widetilde{p}_N(z))}}+\sum_{j=N^d-K+1}^{N^d} \log\pr{1+\frac{C_0N^{-\varrho\eta/2}}{s_j(\prescript{\delta}{}{}\widetilde{p}_N(z)+T_N)}}.
        \end{split}
    \label{Majoration et minoration pour la diff des potentiels - 3}
    \end{equation}
    By \eqref{Majoration et minoration des valeurs singulières par la plus grande et la plus petite} and since $A_2$ in \eqref{Différence des potentiels logaritmiques - 1} holds, \eqref{Majoration et minoration pour la diff des potentiels - 3} becomes
    \begin{equation}\label{Estimation potentiel log matrice tronquée par celui de la matrice du symbole périodique - majoration - 5}
        \begin{split}
            \sum_{j=1}^{N^d} &\log\pr{1+\frac{C_0N^{-\varrho\eta/2}}{s_j(\prescript{\delta}{}{}\widetilde{p}_N(z)+T_N)}}\\
            &\leq\sum_{j=K+1}^{N^d} \log\pr{1+\frac{C_0N^{-\varrho\eta/2}}{s_j(\prescript{\delta}{}{}\widetilde{p}_N(z))}}+K\log\pr{1+\frac{C_0N^{-\varrho\eta/2}}{s_{N^d}(\prescript{\delta}{}{}\widetilde{p}_N(z)+T_N)}}\\
            &\leq \sum_{j=K+1}^{N^d} \log\pr{1+\frac{C_0N^{-\varrho\eta/2}}{s_j(\prescript{\delta}{}{}\widetilde{p}_N(z))}}+K\log\pr{1+C_0N^{\beta-\varrho\eta/2}\delta^{-1}}.
        \end{split}
    \end{equation}

    It remains to estimate the first term on the right hand side of \eqref{Estimation potentiel log matrice tronquée par celui de la matrice du symbole périodique - majoration - 5}. For this, we recall that the number of singular values of $\widetilde{p}_N-z$ which are smaller than $\sqrt{\alpha}$ is $\mathfrak{m}$ given by \eqref{Nombre de valeurs propres Problème de Grushin} (which is uniform with respect to $z\in \mc{K}$) and we use this to decompose the sum. Moreover, by Ky-Fan inequalities, we get that 
    \begin{equation}\label{Estimation potentiel log matrice tronquée par celui de la matrice du symbole périodique - majoration - 6}
        s_j(\prescript{\delta}{}{}\widetilde{p}_N(z))\geq s_j(\widetilde{p}_N(z))-\delta\norme{Q_N}\geq \alpha^{1/2}/2, \mm 1\leq j\leq N^d-\mathfrak{m}.
    \end{equation}
    It follows from \eqref{Estimation potentiel log matrice tronquée par celui de la matrice du symbole périodique - majoration - 6}, \eqref{Majoration et minoration des valeurs singulières par la plus grande et la plus petite} and the fact that $A_1$ in \eqref{Différence des potentiels logaritmiques - 1} holds that 
    \begin{equation}
        \begin{split}
            \sum_{j=K+1}^{N^d}& \log\pr{1+\frac{C_0N^{-\varrho\eta/2}}{s_j(\prescript{\delta}{}{}\widetilde{p}_N(z))}} \\
            &= \pr{\sum_{j=K+1}^{N^d-\mathfrak{m}}+\sum_{N^d-\mathfrak{m}+1}^{N^d}}\log\pr{1+\frac{C_0N^{-\varrho\eta/2}}{s_j(\prescript{\delta}{}{}\widetilde{p}_N(z))}}\\
            &\leq \sum_{j=K+1}^{N^d-\mathfrak{m}} \log\pr{1+\frac{C_0N^{-\varrho\eta/2}}{s_j(\prescript{\delta}{}{}\widetilde{p}_N(z))}}+ \sum_{j=N^d-\mathfrak{m}+1}^{N^d} \log\pr{1+\frac{C_0N^{-\varrho\eta/2}}{s_{N^d}(\prescript{\delta}{}{}\widetilde{p}_N(z))}}\\
            &\leq N^d \log\pr{1+2C_0N^{-\varrho\eta/2}\alpha^{-1/2}}+\mathfrak{m}\log\pr{1+\frac{C_0N^{-\varrho\eta/2}}{s_{N^d}(\prescript{\delta}{}{}\widetilde{p}_N(z))}}\\
            &\leq N^d \log\pr{1+2C_0N^{-\varrho\eta/2}\alpha^{-1/2}}+\mathfrak{m}\log\pr{1+C_0N^{\beta-\varrho\eta/2}\delta^{-1}}.
        \end{split}
    \label{Majoration et minoration pour la diff des potentiels - 4}
    \end{equation}
    Finally, gathering \eqref{Probabilité via l'union bound}, \eqref{Majoration et minoration pour la diff des potentiels - 1}, \eqref{Majoration et minoration pour la diff des potentiels - 2}, \eqref{Estimation potentiel log matrice tronquée par celui de la matrice du symbole périodique - majoration - 5} and \eqref{Majoration et minoration pour la diff des potentiels - 4} and remembering from Proposition \ref{Expression P_N en fonction de P_N tilde} that, $K=\mc{O}(N^{d-\vartheta})$, we obtain the inequality \eqref{Estimation potentiel log matrice tronquée par celui de la matrice du symbole périodique - majoration} with probability $1\geq o(1)$.\\

    \mm 2. Next, we turn to the lower bound \eqref{Estimation potentiel log matrice tronquée par celui de la matrice du symbole périodique - minoration}. Recall that we work under the assumptions that the events in \eqref{Différence des potentiels logaritmiques - 1} hold. We have from \eqref{Décomposition de M_N(p)_delta en fonction de p_N_tilde_delta dans la proposition qui donne la décomposition de M_N(p) en fonction de p_N_tilde}, \eqref{Allahverdiev - 2} and \eqref{Majoration et minoration des valeurs singulières par la plus grande et la plus petite} that
    \begin{equation}
        \begin{split}
            &\sum_{j=1}^{N^d}\log\pr{s_j(M_N^\delta(p)-z)} = \sum_{j=1}^{N^d}\log\pr{s_j(\prescript{\delta}{}{}\widetilde{p}_N(z)+T_N+N^{-\varrho\eta/2}B_N)}\\
            &\m= \pr{\sum_{j=1}^{N^d-K}+\sum_{j=N^d-K+1}^{N^d}}\log\pr{s_j(\prescript{\delta}{}{}\widetilde{p}_N(z)+T_N+N^{-\varrho\eta/2}B_N)}\\
            &\m\geq \sum_{j=K+1}^{N^d}\log\pr{s_j(\prescript{\delta}{}{}\widetilde{p}_N(z)+N^{-\varrho\eta/2}B_N)}+K\log\pr{s_{N^d}(\prescript{\delta}{}{}\widetilde{p}_N(z)+T_N+N^{-\varrho\eta/2}B_N)}\\
            &\m\geq \sum_{j=K+1}^{N^d}\log\pr{s_j(\prescript{\delta}{}{}\widetilde{p}_N(z)+N^{-\varrho\eta/2}B_N)}+ K\log\pr{\delta N^{-\beta}},
        \end{split}
    \label{Majoration et minoration pour la diff des potentiels - 5}
    \end{equation}
    where in the last inequality, we used that event $A_4$ in \eqref{Différence des potentiels logaritmiques - 1} holds. Let us now estimate the first term on the right hand side of \eqref{Majoration et minoration pour la diff des potentiels - 5}. To do so, first notice that in view of \eqref{Majoration norme de p_N_tilde et de B_N dans la deuxième proposition pour l'estimée de la différence des potentiels loga} and by the additive Ky-Fan inequalities, we have  
    \begin{equation}\label{Estimation potentiel log matrice tronquée par celui de la matrice du symbole périodique - minoration - 1}
        s_j\Bigl(\prescript{\delta}{}{}\widetilde{p}_N(z)\Bigl)-\frac{C_0}{N^{\varrho\eta/2}}\leq s_j\Bigl(\prescript{\delta}{}{}\widetilde{p}_N(z)+N^{-\varrho\eta/2}B_N\Bigl)
    \end{equation}
    for all $j\in \disc{1}{N^d}$. Inequality \eqref{Estimation potentiel log matrice tronquée par celui de la matrice du symbole périodique - minoration - 1} therefore gives 
    \begin{equation}\label{Estimation potentiel log matrice tronquée par celui de la matrice du symbole périodique - minoration - 2}
        \begin{split}
            &\sum_{j=K+1}^{N^d}\log\pr{s_j(\prescript{\delta}{}{}\widetilde{p}_N(z)+N^{-\varrho\eta/2}B_N)} \\
            &\m= \sum_{j=K+1}^{N^d-\mathfrak{m}}\log\pr{s_j(\prescript{\delta}{}{}\widetilde{p}_N(z)+N^{-\varrho\eta/2}B_N)}+\sum_{j=N^d-\mathfrak{m}+1}^{N^d}\log\pr{s_j(\prescript{\delta}{}{}\widetilde{p}_N(z)+N^{-\varrho\eta/2}B_N)}\\
            &\m\geq \sum_{j=K+1}^{N^d-\mathfrak{m}}\log\pr{s_j(\prescript{\delta}{}{}\widetilde{p}_N(z))-\frac{C_0}{N^{\varrho\eta/2}}}+\sum_{j=N^d-\mathfrak{m}+1}^{N^d}\log\pr{s_j(\prescript{\delta}{}{}\widetilde{p}_N(z)+N^{-\varrho\eta/2}B_N)}.
        \end{split}
    \end{equation}
    Furthermore, for all $x>y>0$, we have
    \begin{equation}\label{Inégalité nulle sur le log}
        \log(x-y)\geq \log(x)-\frac{y}{x-y}.
    \end{equation}
    Notice also that by Ky-Fan inequalities, \eqref{Valeurs propres inférieures à alpha}, \eqref{Ecriture de delta en décroissance polynomiale}, \eqref{Hypothèse sur delta matrice générale}, \eqref{Contrainte sur delta, alpha et la norme de la matrice aléatoire} and Proposition \ref{Expression P_N en fonction de P_N tilde}, we have for all $j\in \disc{1}{N^d-\mathfrak{m}}$,
    \begin{equation}\label{Estimation potentiel log matrice tronquée par celui de la matrice du symbole périodique - minoration - 3}
        \begin{split}
            s_j(\prescript{\delta}{}{}\widetilde{p}_N(z))-\frac{C_0}{N^{\varrho\eta/2}} &\geq  s_j(\widetilde{p}_N(z))-\delta\norme{Q_N}-\mc{O}(N^{-\varrho\eta/2})\\
            &\geq \alpha^{1/2}/2-\mc{O}(N^{-\varrho\eta/2})\geq \frac{1}{C}\alpha^{1/2}
        \end{split}
    \end{equation}
    for some constant $C>0$ and for $N$ large enough since $0<N^{-\varrho\eta}\alpha^{-1}\ll 1$, see \eqref{Hypothèse ordre de grandeur entre alpha et N}. So, from \eqref{Inégalité nulle sur le log} and \eqref{Estimation potentiel log matrice tronquée par celui de la matrice du symbole périodique - minoration - 3} we have for all $j\in \disc{1}{N^d-\mathfrak{m}}$,
    \begin{equation}\label{Estimation potentiel log matrice tronquée par celui de la matrice du symbole périodique - minoration - 4}
        \begin{split}
            \log\pr{s_j(\prescript{\delta}{}{}\widetilde{p}_N(z))-\frac{C_0}{N^{\varrho\eta/2}}} &\geq \log\pr{s_j(\prescript{\delta}{}{}\widetilde{p}_N(z))}-\frac{C_0N^{-\varrho\eta/2}}{s_j(\prescript{\delta}{}{}\widetilde{p}_N(z))-C_0N^{-\varrho\eta/2}}\\
            &\geq \log\pr{s_j(\prescript{\delta}{}{}\widetilde{p}_N(z))}- \widetilde{C}N^{-\varrho\eta/2}\alpha^{-1/2}
        \end{split}
    \end{equation}
    where $\widetilde{C}=CC_0$. So, from \eqref{Majoration et minoration des valeurs singulières par la plus grande et la plus petite}, \eqref{Estimation potentiel log matrice tronquée par celui de la matrice du symbole périodique - minoration - 4}, \eqref{Estimée de la norme de p_N_tilde_delta uniformément en N avec très bonne proba} and the fact that $A_3$ in \eqref{Différence des potentiels logaritmiques - 1} holds, \eqref{Estimation potentiel log matrice tronquée par celui de la matrice du symbole périodique - minoration - 2} becomes
    \begin{equation}\label{Estimation potentiel log matrice tronquée par celui de la matrice du symbole périodique - minoration - 5}
        \begin{split}
            \sum_{j=K+1}^{N^d}&\log\pr{s_j(\prescript{\delta}{}{}\widetilde{p}_N(z)+N^{-\varrho\eta/2}B_N)} \\
            &\geq \sum_{j=K+1}^{N^d-\mathfrak{m}}\log\pr{s_j(\prescript{\delta}{}{}\widetilde{p}_N(z))}+\mc{O}(\alpha^{-1/2}N^{d-\varrho\eta/2})+\mathfrak{m}\log(\delta N^{-\beta})\\
            &\geq N^d\phi_{\widetilde{\mu}_N}(z)-(K+\mathfrak{m})\log(s_1(\prescript{\delta}{}{}\widetilde{p}_N(z)))+\mc{O}(\alpha^{-1/2}N^{d-\varrho\eta/2})+\mathfrak{m}\log(\delta N^{-\beta})\\
            &\geq N^d\phi_{\widetilde{\mu}_N}(z)+\mc{O}(K)+\mc{O}(\mathfrak{m})+\mc{O}(\alpha^{-1/2}N^{d-\varrho\eta/2})+\mathfrak{m}\log(\delta N^{-\beta}).
        \end{split}
    \end{equation}
    Remembering that $K=\mc{O}(N^{d-\vartheta})$, that $\mathfrak{m}$ is given in \eqref{Nombre de valeurs propres Problème de Grushin},  \eqref{Majoration et minoration pour la diff des potentiels - 5} and \eqref{Estimation potentiel log matrice tronquée par celui de la matrice du symbole périodique - minoration - 5}, we obtain \eqref{Estimation potentiel log matrice tronquée par celui de la matrice du symbole périodique - minoration} with probability $\geq 1-o(1)$.
\end{proof}

\section{Proof of Theorem \ref{Théorème final}}\label{Démo du théorème final}
\mm We begin by recalling a criterion for the weak convergence of measures. This is a standard result (see e.g. \cite[Theorem 2.8.3]{taotopics}), and we present here a version adapted to our case. Its proof is given for the reader convenience in the Appendix.

Recall the notation \eqref{Définition du potentiel loga pour une mesure donnée} and $\mathscr{P}(\C)$ given at the beginning of Section \ref{Principal setting}.

\begin{thm}
    Let $(\mu_n)_{n\in \N}$ be a sequence of random probability measures and let $\mu\in \mathscr{P}(\C)$ be deterministic. Suppose that:
    \begin{enumerate}
        \item there exists a compact set $D$ such that $\mathrm{supp}(\mu)\subset D$ and
        \begin{equation*}
            \mathbb{P}\Bigl(\mathrm{supp}\pr{\mu_n}\subset D\Bigl)=1-o(1)\m\m\m\text{when}\m\m\m\m n\to +\infty,
        \end{equation*}
        \item for almost every $z\in \C$, $\phi_{\mu_n}(z)$ converges in probability to $\phi_{\mu}(z)$.
    \end{enumerate}
    Then, $\mu_n$ converges weakly in probability to $\mu$.
\label{Théorème cvgence du potentiel loga implique celle de la mesure empirique}
\end{thm}

    \mm We now turn to the proof of Theorem \ref{Théorème final}. Let $\varrho\in \hspace{0.1cm} ]0,1]$ and let $p\in \mathscr{C}^{0,\varrho}_{\mathrm{pw}}S_d$ satisfy Assumptions \ref{Hypothèse théorème final - volume singularités} and \ref{Hypothèse théorème final - volume préimage} with parameters ${\kappa_1}\in \hspace{0.1cm} ]0,1]$ and ${\kappa_2}\in \hspace{0.1cm} ]0,1]$ respectively. As at the beginning of Section \ref{Principal setting}, restrict $p$ to $]0,1]^d_x\times \T^d_\xi$ and extend it by $\Z^d$ periodicity to $\R^d_x\times \T^d_\xi$. For $\eta=1/6\in \hspace{0.1cm} ]0,1/4[$, take $\widetilde{p}$ defined by \eqref{Definition de p tilde}. Let $Q_N$ be an $N^d\times N^d$ random matrix satisfying Assumption \ref{Hypothèse théorème final - matrice aléatoire} with $\kappa_3>0$. Let $\delta_0>0$ and set 
    \begin{equation}\label{Hypothèse expression de delta théorème final bis}
        \delta:=N^{-(\kappa_3+\delta_0)}.
    \end{equation}

    We consider as well for $N\gg1$, the semiclassical parameter 
    \begin{equation*}
        h=\frac{1}{2\pi N}
    \end{equation*}
    and we set 
    \begin{equation}
        \alpha:=N^{-\omega}\mm\mm\text{where}\mm\mm \omega:=\min(\delta_0,\varrho\eta/3,{\kappa_1}\eta/2)>0.
    \label{Choix de alpha + notation de omega}
    \end{equation}
    Choosing $\alpha$ in this way ensures that $h\ll h^{2\eta}\alpha$. Indeed, for $\eta=1/6$, since $0<\varrho\leq 1$,
    \begin{equation*}
        2\eta+\frac{\varrho\eta}{2}\leq \frac{5\eta}{2}<1,
    \end{equation*}
    so \eqref{Conditions entre alpha 1 et 2, et N - 2} holds. Furthermore, expression \eqref{Choix de alpha + notation de omega} of $\alpha$ allows us as well to have
    \begin{equation*}
        N^{-\eta\min(\varrho,{\kappa_1})}\ll \alpha\ll 1,
    \end{equation*}
    so that Theorem \ref{Théorème central} applies. Moreover, notice that $\omega(1+{\kappa_2})-\varrho\eta<0$ and $\omega{\kappa_2}-{\kappa_1}\eta<0$, which implies that
    \begin{equation*}
        N^{-\varrho\eta}\alpha^{-1}\ll \alpha^{\kappa_2}\ll1\mm\mm\text{and}\mm\mm N^{-{\kappa_1}\eta}\ll\alpha^{\kappa_2}\ll1,
    \end{equation*}
    meaning that Corollary \ref{Nombre de valeurs propres plus petites que alpha - résultat} applies too.\\
    
    We set for $\varepsilon_0$ intervening in \eqref{Apparition de epsilon_0 dans la proba de la norme de la matrice aléatoire},
    \begin{equation}\label{Choix de epsilon_0}
        \varepsilon_0:=\frac{\delta_0}{4}
    \end{equation}
    which in view of \eqref{Hypothèse expression de delta théorème final bis}, gives
    \begin{equation*}
        \delta N^{\kappa_3+\varepsilon_0} \alpha^{-1/2}=N^{-\delta_0+\varepsilon_0+\min(\delta_0,\varrho\eta/2)/2}=\mc{O}(N^{-\delta_0/4}) \longrightarrow 0 \mm\text{as}\mm N\to +\infty.
    \end{equation*}
    In particular, \eqref{Ecriture de delta en décroissance polynomiale} holds.\\

    \mm We then show that the hypotheses of Theorem \ref{Théorème cvgence du potentiel loga implique celle de la mesure empirique} are fulfilled.\\

    \mm \textit{(1)} Since $p\in \mathscr{C}^{0,\varrho}_{\mathrm{pw}}S_d$, we know from Remark \ref{Remarque sur le fait que les matrices via la nouvelle procédure de quantif sont bornées} that there exists a constant $C>0$ independent of $N$ such that $\norme{M_N(p)}\leq C$. Then, with probability $\geq 1-\mc{O}(N^{-\varepsilon_0})$, 
    \begin{equation*}
        \begin{split}
            \norme{M_N^\delta(p)} \leq \norme{M_N(p)}+\delta \norme{Q_N}\leq C+\alpha^{1/2}/2\leq C+1.
        \end{split}
    \end{equation*}
    In particular, 
    \begin{equation}
        \sigma\pr{M_N^\delta(p)}\subset D\pr{0,\norme{M_N^\delta(p)}}\subset \overline{D(0,C+1)}
    \label{Démo du théorème final - 1}
    \end{equation}
    and $\overline{D(0,C+1)}$ is a compact set of $\C$. In view of \eqref{Démo du théorème final - 1} and \eqref{Expression mesure empirique de P perturbé via nouvelle quantif}, we have
    \begin{equation}
        \mathbb{P}\Bigl(\mathrm{supp}(\mu_N)\subset \overline{D(0,C+1)}\Bigl)\geq 1-\mc{O}(N^{-\varepsilon_0}).
    \label{Démo du théorème final - 2}
    \end{equation}\\

    \mm \textit{(2)} Let $z\in \C$ and $\mc{K}=\overline{D(z,1)}$. From Proposition \ref{Différence des potentiels loga entre mu_N tilde et mu}, we know that \eqref{Estimée log-det supérieure, cas gaussien -potentiel log} holds with probability $\geq 1-\mc{O}(N^{-\varepsilon_0})$ and \eqref{Estimée log-det inférieure, cas gaussien -potentiel log} with probability $\geq 1-o(1)$. By the union bound, this leads to 
    \begin{equation}
    \begin{split}
        \bb{\phi_{\widetilde{\mu}_N}(z)-\phi_{\mu}(z)} &\leq \mc{O}\pr{N^{-\eta{\kappa_1}}\bb{\log(\alpha)}}+\mc{O}\pr{\alpha^{\kappa_2}\bb{\log(\alpha)}}+\mc{O}\pr{N^{-\varrho\eta}\alpha^{-1}}\\
        &\m\m\m\m\m\m\m\m\m\m +\mc{O}\pr{\delta\alpha^{-1/2}N^{\kappa_3+\varepsilon_0}}+\mc{O}\pr{\alpha^{{\kappa_2}}\log(\delta N^{-\beta})}
    \end{split}
\label{Différence des potentiels loga entre mu_N tilde et mu - absolute value}
\end{equation}
with probability $\geq 1-o(1)$.\\

Furthermore, from Proposition \ref{Différence des potentiels logaritmiques}, we also have that \eqref{Estimation potentiel log matrice tronquée par celui de la matrice du symbole périodique - majoration} holds with probability $\geq 1-o(1)$ and \eqref{Estimation potentiel log matrice tronquée par celui de la matrice du symbole périodique - minoration} with probability $\geq 1-o(1)$, which implies by the union bound that 
    \begin{equation}
        \begin{split}
            \bb{\phi_{\mu_N}(z)-\phi_{\widetilde{\mu}_N}(z)} &\leq \mc{O}\pr{N^{-\vartheta}\log\pr{1+C_0N^{\beta-\varrho\eta/2}\delta^{-1}}} \\
            &+ \mc{O}(N^{-\vartheta}\log(\delta N^{-\beta}))+\log\pr{1+2C_0N^{-\varrho\eta/2}\alpha^{-1/2}}\\
            &+\mc{O}\pr{\alpha^{{\kappa_2}}\log\pr{1+C_0N^{\beta-\varrho\eta/2}\delta^{-1}}}+\mc{O}(\alpha^{-1/2}N^{-\varrho\eta/2})\\
            &+\mc{O}(\alpha^{{\kappa_2}}\log(\delta N^{-\beta}))+\mc{O}(N^{-\vartheta})
        \end{split}
    \label{Estimation potentiel log matrice tronquée par celui de la matrice du symbole périodique - absolute value}
    \end{equation}
    with probability $\geq 1-o(1)$.

\m\\
Through the choice \eqref{Choix de alpha + notation de omega} of $\alpha$, \eqref{Choix de epsilon_0} of $\varepsilon_0$ and \eqref{Hypothèse expression de delta théorème final bis} of $\delta$, all the terms on the right hand side of \eqref{Différence des potentiels loga entre mu_N tilde et mu - absolute value} and \eqref{Estimation potentiel log matrice tronquée par celui de la matrice du symbole périodique - absolute value} are vanishing at $+\infty$. Indeed, after simplification, these terms are
\begin{align*}
          \mc{O}\pr{N^{-\eta{\kappa_1}}\bb{\log(\alpha)}} &= \mc{O}(N^{-\eta{\kappa_1}}\log(N)) \\
        \mc{O}(\alpha^{\kappa_2}\bb{\log(\alpha)}) &= \mc{O}(N^{-\omega{\kappa_2}}\log(N))\\
        \mc{O}\pr{N^{-\varrho\eta}\alpha^{-1}} &= \mc{O}(N^{-\varrho\eta/2})\\
        \mc{O}\pr{\delta\alpha^{-1/2}N^{\kappa_3+\varepsilon_0}} &=\mc{O}(N^{-\varepsilon/4})\\
        \mc{O}\pr{\alpha^{{\kappa_2}}\log(\delta N^{-\beta})} &= \mc{O}(N^{-\omega{\kappa_2}}\log(N))\\
        \mc{O}(N^{-\vartheta}\log(\delta N^{-\beta})) &= \mc{O}(N^{-\vartheta}\log(N))\\
        \mc{O}\pr{N^{-\vartheta}\log\pr{1+C_0N^{\beta-\varrho\eta/2}\delta^{-1}}}&\leq \mc{O}(N^{-\vartheta}\log(N))\\
        \log\pr{1+2C_0N^{-\varrho\eta/2}\alpha^{-1/2}} &= \mc{O}(N^{-\varrho\eta/4})\\
        \mc{O}\pr{\alpha^{{\kappa_2}}\log\pr{1+C_0N^{\beta-\varrho\eta/2}\delta^{-1}}} &\leq \mc{O}(N^{-\omega{\kappa_2}}\log(N))\\
        \mc{O}(\alpha^{-1/2}N^{-\varrho\eta/2}) &= \mc{O}(N^{-\varrho\eta/4})
\end{align*}
where we recall that $\omega$ was defined in \eqref{Choix de alpha + notation de omega}. \\

Let $\tau>0$. To simplify the notations, let us denote $E_1$ the right hand side of \eqref{Différence des potentiels loga entre mu_N tilde et mu - absolute value}, which vanishes when $N\to +\infty$. So, for $N$ large enough, it is smaller than $\tau$ and we can write for $N$ large enough that
\begin{equation*}
    \mathbb{P}\Bigl(\bb{\phi_{\widetilde{\mu}_N}(z)-\phi_{\mu}(z)}\leq \tau\Bigl) \geq \mathbb{P}\Bigl(\bb{\phi_{\widetilde{\mu}_N}(z)-\phi_{\mu}(z)}\leq E_1\Bigl)\geq 1-o(1).
\end{equation*}
This ensures that 
\begin{equation}
    \phi_{\widetilde{\mu}_N}(z)\overset{\mathbb{P}}{\underset{N\to +\infty}{\longrightarrow}}\phi_{\mu}(z).
\label{Convergence des potentiels loga - 3}
\end{equation}
Likewise, denote $E_2$ the right hand side of \eqref{Estimation potentiel log matrice tronquée par celui de la matrice du symbole périodique - absolute value} which vanishes at $+\infty$ too. Therefore, for $N$ large enough, 
\begin{equation*}
    \mathbb{P}\Bigl(\bb{\phi_{\mu_N}(z)-\phi_{\widetilde{\mu}_N}(z)}\leq \tau\Bigl)\geq  \mathbb{P}\Bigl(\bb{\phi_{\mu_N}(z)-\phi_{\widetilde{\mu}_N}(z)}\leq E_2\Bigl)\geq 1-o(1)
\end{equation*}
which leads to
\begin{equation}
    \phi_{\mu_N}(z)-\phi_{\widetilde{\mu}_N}(z)\overset{\mathbb{P}}{\underset{N\to +\infty}{\longrightarrow}} 0.
\label{Convergence des potentiels loga - 4}
\end{equation}
Thus, we deduce from \eqref{Convergence des potentiels loga - 3} and \eqref{Convergence des potentiels loga - 4} that
\begin{equation}
    \phi_{\mu_N}(z)\overset{\mathbb{P}}{\underset{N\to +\infty}{\longrightarrow}}\phi_{\mu}(z)
\label{Démo du théorème final - 3}
\end{equation}
and this holds for all $z\in \C$. \\

Finally, it remains to apply Theorem \ref{Théorème cvgence du potentiel loga implique celle de la mesure empirique} with \eqref{Démo du théorème final - 2} and \eqref{Démo du théorème final - 3} to conclude the proof of Theorem \ref{Théorème final}.

\appendix
\section{Proof of estimates for Theorem \ref{Théorème central}}

\begin{lem}\label{Lemme nul 2}
    Let $a\in \disc{-(N-a_N)}{N-a_N}$. Then, for all $b\in \Z$, 
    \begin{equation*}
        \bb{a+Nb}\geq \frac{a_N}{2}\bb{b}.
    \end{equation*}
\end{lem}

\begin{proof}
    We assume that there exists a constant $N>C>0$ such that for all $b\in \Z$, $\bb{
    a+Nb}\geq C\bb{b}$. Equivalently, for all $b\in \Z$,
    \begin{equation}
        (N^2-C^2)b^2+2Nba+a^2\geq 0.
    \label{Inégalité lemme nul}
    \end{equation}
    Let us study this inequation on $\R$. The left hand side vanishes when $b=b_1$ or $b_2$ where
    \begin{equation*}
        b_1=-\frac{Na+C\bb{a}}{N^2-C^2} \m\m\m\m\m\m \text{and}\m\m\m\m\m\m b_2=-\frac{Na-C\bb{a}}{N^2-C^2}
    \end{equation*}
    which depends on the sign of $a$. If $a\geq 0$, then
    \begin{equation*}
        b_1=-\frac{a}{N-C} \m\m\m\m\m\m \text{and}\m\m\m\m\m\m b_2=-\frac{a}{N+C}.
    \end{equation*}
    In particular, $b_1\leq b_2\leq 0$, so for \eqref{Inégalité lemme nul} to be true for all $b\in \Z$, it is sufficient that $-1<b_1$, which is equivalent to $a<N-C$. Since $a\leq N-a_N$, it suffices to take $C=a_N/2$.\\
    We obtain the same conclusion if $a\leq 0$.
\end{proof}

\begin{lem}
    The norm of the $N^d\times N^d$ matrix 
    \begin{equation*}
        B_N:=\pr{\mathbb{1}_{\disc{-N^a}{N^a}^d}(s-j)}_{s,j}
    \end{equation*}
    satisfies:
    \begin{equation*}
        \norme{B_N}=\mc{O}\pr{N^{\frac{ad}{2}}}.
    \end{equation*}
\label{Norme matrice B_N en dimension d}
\end{lem}

\begin{proof}
    For $u\in \C^{\disc{1}{N}^d}$, we have from the Minkowski's inequality that
    \begin{equation*}
        \begin{split}
            \norme{B_Nu}_2 &= \pr{\sum_{n\in \disc{1}{N}^d}\bb{\sum_{\nu\in \disc{1}{N}^d}(B_N)_{n,\nu}u_\nu}^2}^{1/2}\leq \sum_{\nu\in \disc{1}{N}^d}\pr{\bb{u_\nu}^2\underbrace{\sum_{n\in \disc{1}{N}^d}\bb{(B_N)_{n,\nu}}^2}_{=:\mc{N}_\nu}}^{1/2}.
        \end{split}
    \end{equation*}
    Here, since all the entries of $B_N$ are in $\{0,1\}$, we obtain that 
    \begin{equation*}
        \mc{N}_\nu=\#\acc{n\in \disc{1}{N}^d\m\left|\m \forall j\in \disc{1}{N}, \m \nu_j-N^a\leq n_j\leq \nu_j+N^a\right.}\leq (2N^a+1)^{d}
    \end{equation*}
    and this leads to the result.
\end{proof}

\section{Proof of Theorem \ref{Théorème cvgence du potentiel loga implique celle de la mesure empirique}}.
\begin{lem}
    Let $(\mu_n)$ be a sequence of $\mathscr{P}(\C)$ and $\mu\in \mathscr{P}(\C)$ satisfying: there exists a compact $K$ of $\C$ such that
    \begin{equation}
        \Prob{\mathrm{supp}\pr{\mu_n}\not\subset K}=o(1)\m\m\m\text{when}\m\m\m n\to +\infty.
    \label{Hypothèse support des mesures contenu dans un compact avec bonne proba}
    \end{equation}
    Then,
    \begin{equation*}
        \mu_n\overset{\mathbb{P}}{\rightharpoonup} \mu \Longleftrightarrow \mu_n\overset{\mathbb{P}}{\to} \mu \m\m\m\text{in}\m\m\m \mc{D}'(\C).
    \end{equation*}
\label{Lemme équivalence cvg faible et au sens distri des mesures de proba}
\end{lem}

\begin{proof}
    In this proof, we will denote $\mathscr{C}_b(\C;\R)$ the set of all the bounded and continuous functions $f:\C\to \R$. To simplify the notations, for all measure $\mu$ and function $f\in L^1(\C,\mu)$, we denote 
    \begin{equation*}
        \ps{\mu}{f}:=\int_\C f(x)\mu(dx).
    \end{equation*}
    Let us recall that $\mu_n\overset{\mathbb{P}}{\rightharpoonup} \mu$ means that for all $f\in \mathscr{C}_b(\C;\R)$, for all $\varepsilon>0$
    \begin{equation}
        \mathbb{P}\Bigl(\bb{\ps{\mu_n}{f}-\ps{\mu}{f}}\geq \varepsilon\Bigl)=o(1)\m\m\m\text{when}\m\m\m n\to +\infty.
    \label{Def convergence faible en proba des mesures}
    \end{equation}
    Likewise, $\mu_n\overset{\mathbb{P}}{\to} \mu $ in $\mc{D}'(\C)$ means the same thing as \eqref{Def convergence faible en proba des mesures} with $f\in \mathscr{C}^\infty_c(\C,\R)$ instead.\\
    
    \m\m\m\m The necessary condition is obvious. We now turn to the sufficient one. We begin by reducing to the case of compactly supported continuous functions. Let $f\in \mathscr{C}_b(\C;\R)$. Let $K$ be a compact set of $\C$ given in \eqref{Hypothèse support des mesures contenu dans un compact avec bonne proba}. After having possibly enlarged $K$, we may assume that $\mathrm{supp}(\mu)\subset K$. Let $\widetilde{K}$ be another compact set such that
    \begin{equation*}
        \overline{K+D(0,1)}\subset\mathring{\widetilde{K}}\subset \widetilde{K}\m
    \end{equation*}
    and $\chi\in \mathscr{C}^\infty_c(\C;\R)$ such that $\chi\equiv 1$ on $K$, $\chi\equiv 0$ on $\widetilde{K}^c$ and $0\leq \chi\leq 1$. So,
    \begin{equation*}
        \begin{split}
             \Prob{\bb{\ps{\mu_n}{f}-\ps{\mu}{f}}\geq \varepsilon} &= \Prob{\acc{\bb{\ps{\mu_n}{f}-\ps{\mu}{f}}\geq \varepsilon}\cap \acc{\mathrm{supp}(\mu_n)\subset K}}\\
             &\m\m\m\m + \Prob{\acc{\bb{\ps{\mu_n}{f}-\ps{\mu}{f}}\geq \varepsilon}\cap \acc{\mathrm{supp}(\mu_n)\not\subset K}}\\
             &\leq \Prob{\acc{\bb{\ps{\mu_n}{\chi f}-\ps{\mu}{\chi f}}\geq \varepsilon}\cap \acc{\mathrm{supp}(\mu_n)\subset K}}\\
             &\m\m\m\m +\Prob{\mathrm{supp}(\mu_n)\not\subset K}\\
             &\leq \Prob{\bb{\ps{\mu_n}{\chi f}-\ps{\mu}{\chi f}}\geq \varepsilon}+o(1),
        \end{split}
    \end{equation*}
    where the second inequality comes from the hypothesis \eqref{Hypothèse support des mesures contenu dans un compact avec bonne proba}. It remains to show \eqref{Def convergence faible en proba des mesures} for functions in $\mathscr{C}_c(\mathring{\widetilde{K}};\R)$. \\

    \m\m\m\m Let us take $f\in \mathscr{C}_c(\mathring{\widetilde{K}};\R)$. Then, by density, there exists $g\in \mathscr{C}^\infty_c(\mathring{\widetilde{K}};\R)$ such that $\norme{f-g}_{L^\infty(\C)}\leq \frac{\varepsilon}{3}$. Therefore, almost surely,
    \begin{equation*}
        \begin{split}
            \bb{\ps{\mu_n}{f}-\ps{\mu}{f}} &\leq \bb{\ps{\mu_n}{f-g}}+\bb{\ps{\mu_n}{g}-\ps{\mu}{g}}+\bb{\ps{\mu}{f-g}}\\
            &\leq \frac{2\varepsilon}{3}+\bb{\ps{\mu_n}{g}-\ps{\mu}{g}}.
        \end{split}
    \end{equation*}
    In that case, since we assumed that $\mu_n\overset{\mathbb{P}}{\to} \mu$ in $\mc{D}'(\C)$, we have
    \begin{equation*}
        \begin{split}
            \Prob{\bb{\ps{\mu_n}{f}-\ps{\mu}{f}}\geq \varepsilon} &\leq 1-\Prob{\bb{\ps{\mu_n}{g}-\ps{\mu}{g}}< \frac{\varepsilon}{3}}\\
            &= \Prob{\bb{\ps{\mu_n}{g}-\ps{\mu}{g}}\geq  \frac{\varepsilon}{3}}\underset{n\to +\infty}{\longrightarrow} 0,
        \end{split}
    \end{equation*}
    which completes the proof.
\end{proof}

\begin{proof}[Proof of Theorem \ref{Théorème cvgence du potentiel loga implique celle de la mesure empirique}]
    1. Firstly, we show that for all compact sets $K$ of $\C$, for all $M\geq 0$,
    \begin{equation}
        \int_{K}\min\pr{\bb{\phi_{\mu_n}(z)-\phi_{\mu}(z)},M}\mathrm{L}(dz)\underset{n\to +\infty}{\longrightarrow} 0\m\m\m\m \text{in probability}.
    \label{Théorème cvgence du potentiel loga implique celle de la mesure empirique - 1}
    \end{equation}
    
    To this end, take $K$ a compact set, $M\geq 0$ and $\varepsilon>0$. By Markov's inequality and Fubini's theorem, 
    \begin{equation*}
        \begin{split}
            &\underbrace{\Prob{\bb{\int_{K}\min\pr{\bb{\phi_{\mu_n}(z)-\phi_{\mu}(z)},M}\mathrm{L}(dz)}\geq \varepsilon}}_{=:\mc{P}_n^\varepsilon}\\
            &\mm\mm\mm\mm\mm\mm\leq \frac{1}{\varepsilon}\mathbb{E}\pr{\int_{K}\min\pr{\bb{\phi_{\mu_n}(z)-\phi_{\mu}(z)},M}\mathrm{L}(dz)}\\
            &\mm\mm\mm\mm\mm\mm\leq \frac{1}{\varepsilon}\int_{K}\underbrace{\mathbb{E}\croch{\min\pr{\bb{\phi_{\mu_n}(z)-\phi_{\mu}(z)},M}}}_{=:I_n(z)}\mathrm{L}(dz).
        \end{split}
    \end{equation*}
    Furthermore, since almost surely for every $z\in K$, $\min\pr{\bb{\phi_{\mu_n}(z)-\phi_{\mu}(z)},M}\leq M$, we get that $I_n(z)\leq M$ for all $z\in K$.\\

    Let $\delta>0$ and write:
    \begin{equation*}
        \begin{split}
            I_n(z)&=\pr{\int_{\bb{\phi_{\mu_n}(z)-\phi_\mu(z)}<\delta}+\int_{\bb{\phi_{\mu_n}(z)-\phi_\mu(z)}\geq\delta}} \min\pr{\bb{\phi_{\mu_n}(z)-\phi_{\mu}(z)},M}d\mathbb{P}\\
            &=: I_n^1(z)+I_n^2(z).
        \end{split}
    \end{equation*}
    On the one hand, for all $z\in K$, $I_n^1(z)\leq \delta\Prob{\bb{\phi_{\mu_n}(z)-\phi_\mu(z)}<\delta}\leq \delta$, and on the other hand, for all $z\in K$, by assumption \textit{(2)}, 
    \begin{equation*}
        I_n^2(z)\leq M\Prob{\bb{\phi_{\mu_n}(z)-\phi_\mu(z)}\geq \delta}\underset{n\to +\infty}{\longrightarrow}0.
    \end{equation*}
    Thus, using the dominated convergence theorem, we get
    \begin{equation*}
        \limsup_{n\to +\infty}\mc{P}_n^\varepsilon\leq \frac{\delta\mathrm{L}(K)}{\varepsilon}.
    \end{equation*}
    Since this inequality holds for all $\delta>0$, it implies that
    \begin{equation}
        \lim_{n\to +\infty}\mc{P}_n^\varepsilon=\limsup_{n\to +\infty}\mc{P}_n^\varepsilon=0,
    \label{Théorème cvgence du potentiel loga implique celle de la mesure empirique - 1.5}
    \end{equation}
    which leads to the property \eqref{Théorème cvgence du potentiel loga implique celle de la mesure empirique - 1}.\\

    \m\m\m\m 2. For $K$ a compact subset of $\C$, let $M>0$ be such that $K,D\subset D(0,M)$. Then, for all $w\in D$, 
    \begin{equation*}
        \begin{split}
            \int_K |\log&\pr{\bb{z-w}}|^2\mathrm{L}(dz)\\
            &\leq \int_{D(w,1)} \bb{\log\pr{\bb{z-w}}}^2\mathrm{L}(dz)+ \int_{K\setminus D(w,1)} \bb{\log\pr{\bb{z-w}}}^2\mathrm{L}(dz)\\
            &\leq 2\pi\int_{0}^1 r\log^2(r)dr+\mathrm{L}(K)\log^2(2M)=:\widetilde{C}<+\infty,
        \end{split}
    \end{equation*}
    and the constant $\widetilde{C}$ does not depend on $w\in D$. Using the Minkowski's inequality, we have
    \begin{equation}
        \begin{split}
            \norme{\phi_{\mu_n}}_{L^2(K)}&:=\pr{\int_{K}\bb{\int_D\log\pr{\bb{z-w}}d\mu_n(w)}^2\mathrm{L}(dz)}^{1/2}\\
            &\leq \int_D\pr{\int_K\bb{\log\pr{\bb{z-w}}}^2\mathrm{L}(dz)}^{1/2}d\mu_n(w)\leq \widetilde{C}^{1/2}\m
        \end{split}
    \label{Théorème cvgence du potentiel loga implique celle de la mesure empirique - 2}
    \end{equation}
    with probability $\geq 1-o(1)$, and the constant does not depend on $n$. Since we assumed that $\mathrm{supp}(\mu)\subset D$, we also find that
    \begin{equation*}
        \norme{\phi_\mu}_{L^2(K)}\leq \widetilde{C}^{1/2}.
    \end{equation*}

    \m\m\m\m 3. Almost surely, for all $M>0$, 
    \begin{equation}
        \begin{split}
            &\norme{\phi_{\mu_n}-\phi_{\mu}}_{L^1(K)}\\
            &\mm\mm=\int_{\bb{\phi_{\mu_n}-\phi_{\mu}}>M}\bb{\phi_{\mu_n}-\phi_{\mu}}d\mathrm{L}+\int_{\bb{\phi_{\mu_n}-\phi_{\mu}}\leq M} \bb{\phi_{\mu_n}-\phi_{\mu}}d\mathrm{L}\\
            &\mm\mm=:I_n^1+I_n^2.
        \end{split}
    \label{Théorème cvgence du potentiel loga implique celle de la mesure empirique - 3}
    \end{equation}
    On the one hand, using the Cauchy-Schwarz's and Markov's inequalities, almost surely,
    \begin{equation*}
        \begin{split}
            I_n^1 &\leq \mathrm{L}\!\pr{\acc{\bb{\phi_{\mu_n}-\phi_{\mu}}>M}}^{1/2}\times \norme{\phi_{\mu_n}-\phi_{\mu}}_{L^2(K)}\\
            &\leq \frac{\norme{\phi_{\mu_n}-\phi_{\mu}}^{1/2}_{L^1(K)}}{M^{1/2}}\times \norme{\phi_{\mu_n}-\phi_{\mu}}_{L^2(K)}.
        \end{split}
    \end{equation*}
    And in view of \eqref{Théorème cvgence du potentiel loga implique celle de la mesure empirique - 2}, this previous inequality and the Cauchy-Schwarz's one lead to: with probability $\geq 1-o(1)$, for all $M>0$, 
    \begin{equation}
        I_n^1\leq \frac{\atc{C}}{M^{1/2}}.
    \label{Théorème cvgence du potentiel loga implique celle de la mesure empirique - 4}
    \end{equation}
    For the other integral, almost surely, we have
    \begin{equation}
        I_n^2\leq \int_{K}\min\pr{\bb{\phi_{\mu_n}(z)-\phi_{\mu}(z)},M}\mathrm{L}(dz)=:J_{n,M}.
    \label{Théorème cvgence du potentiel loga implique celle de la mesure empirique - 5}
    \end{equation}
    In that case, in view of \eqref{Théorème cvgence du potentiel loga implique celle de la mesure empirique - 3}, \eqref{Théorème cvgence du potentiel loga implique celle de la mesure empirique - 4} and \eqref{Théorème cvgence du potentiel loga implique celle de la mesure empirique - 5}, for all $M>0$, we have
    \begin{equation}
        \begin{split}
            \Prob{\norme{\phi_{\mu_n}-\phi_{\mu}}_{L^1(K)}\leq \frac{\atc{C}}{M^{1/2}}+J_{n,M}}&\geq \Prob{\acc{I_n^1\leq \frac{\atc{C}}{M^{1/2}}}\bigcap\acc{I_n^2\leq J_{n,M}}}\\
            &\geq 1-\Prob{I_n^1> \frac{\atc{C}}{M^{1/2}}}-\underbrace{\Prob{I_n^2>J_{n,M}}}_{=0}\\
            &\geq \Prob{I_n^1\leq  \frac{\atc{C}}{M^{1/2}}}\geq 1-o(1).
        \end{split}
    \label{Théorème cvgence du potentiel loga implique celle de la mesure empirique - 6}
    \end{equation}
    Let $\varepsilon>0$ and $M>0$ such that $\frac{\atc{C}}{M^{1/2}}\leq \frac{\varepsilon}{2}$. Then, thanks to \eqref{Théorème cvgence du potentiel loga implique celle de la mesure empirique - 6},
    \begin{equation*}
        \begin{split}
            \mathbb{P}\biggl(&\norme{\phi_{\mu_n}-\phi_{\mu}}_{L^1(K)} \leq \varepsilon\biggl)\\
            &\mm\mm\geq \Prob{\acc{J_{n,M}\leq \frac{\varepsilon}{2}}\bigcap \acc{\norme{\phi_{\mu_n}-\phi_{\mu}}_{L^1(K)}\leq \frac{\atc{C}}{M^{1/2}}+J_{n,M}}}\\
            &\mm\mm\geq 1-\mc{P}_n^{\varepsilon/2}-\Prob{\norme{\phi_{\mu_n}-\phi_{\mu}}_{L^1(K)}> \frac{\atc{C}}{M^{1/2}}+J_{n,M}}\\
            &\mm\mm\geq 1-o(1)-\mc{P}_n^{\varepsilon/2}.
        \end{split}
    \end{equation*}
    It remains to use \eqref{Théorème cvgence du potentiel loga implique celle de la mesure empirique - 1.5} to conclude that
    \begin{equation}
        \norme{\phi_{\mu_n}-\phi_{\mu}}_{L^1(K)}\underset{n\to +\infty}{\longrightarrow} 0 \m\m\m\text{in probability}.
    \label{Théorème cvgence du potentiel loga implique celle de la mesure empirique - 7}
    \end{equation}
    Using Hölder's inequality, \eqref{Théorème cvgence du potentiel loga implique celle de la mesure empirique - 7} also provides that 
    \begin{equation}
        \Delta \phi_{\mu_n}\overset{\mathbb{P}}{\underset{n\to +\infty}{\longrightarrow}}\Delta \phi_\mu\m\m\m\text{in}\m\m\m \mc{D}'(\C)
    \label{Théorème cvgence du potentiel loga implique celle de la mesure empirique - 8}
    \end{equation}
    which means that for all $\varphi\in \mathscr{C}^\infty_c(\C,\R)$, $\varepsilon>0$,
    \begin{equation*}
        \mathbb{P}\Bigl(\bb{\ps{\Delta\phi_{\mu_n}}{\varphi}-\ps{\Delta\phi_\mu}{\varphi}}\geq \varepsilon\Bigl)\underset{n\to +\infty}{\longrightarrow}0.
    \end{equation*}
    Here, $\ps{\cdot}{\cdot}$ denotes the duality brackets $\mc{D}'(\C)\times \mc{D}(\C)$. Remembering that $\Delta\phi_{\mu}=2\pi \mu$ in $\mc{D}'(\C)$ and that almost surely, for all $n\in \N^*$,
    \begin{equation*}
        \Delta\phi_{\mu_n}=2\pi \mu_n\m\m\m\text{in}\m\m\m \mc{D}'(\C),
    \end{equation*}
    it remains to use \eqref{Théorème cvgence du potentiel loga implique celle de la mesure empirique - 8} and Lemma \ref{Lemme équivalence cvg faible et au sens distri des mesures de proba} to conclude.
\end{proof}

\mm We now give a criterion to satisfy Assumption \ref{Hypothèse théorème final - volume préimage}.
\begin{prop}\label{Généralisation BPZ - 1}
    Let $p=p(\xi)$ be of the form
    \begin{equation*}
        p(\xi)=\sum_{k\in \Lambda}p_ke^{2i\pi \ps{k}{\xi}}, \quad\xi\in \T^d,
    \end{equation*}
    where $\Lambda=\prod_{j=1}^d\disc{-n_j}{m_j}$, $n_j,m_j\geq 0$ and $p_k\in \C$, $k\in \Lambda$. We assume that 
    \begin{equation}\label{Généralisation BPZ - Hypothèse - 1}
        \exists \hspace{0.1cm }k\in \Lambda\setminus \{0\}, \m p_k\neq 0.
    \end{equation}
    Let us denote 
    \begin{equation*}
        S:=\sum_{j=1}^d (n_j+m_j)>0, \mm \Xi:=\prod_{j=1}^d(n_j+m_j+1) \mm\text{and}\mm m:=\sum_{\substack{k\in \Lambda\\ k\neq 0}}\bb{p_k}>0.
    \end{equation*}
    Then, for all $z\in \C$, $t\geq 0$, 
    \begin{equation}\label{Généralisation BPZ - Hypothèse - 2}
        \mathrm{L}\pr{\acc{\xi \in \T\m ;\m \bb{p(\xi)-z}^2\leq t}}\leq 14d \pr{\frac{\sqrt{\Xi}}{m}\sqrt{t}}^{\frac{1}{S}}.
    \end{equation}
\end{prop}

\begin{proof}
    Let $z\in \C$. Applying Fontes-Merz's inequality to $p-z$ (see \cite[Theorem 2]{Fontes_Merz}) ensures that for all measurable sets $E$ of $\T^d$, 
    \begin{equation}\label{Généralisation BPZ - preuve - 1}
        \norme{p-z}_{L^\infty(\T)}\leq \pr{\frac{14d}{\mathrm{L}(E)}}^{S}\norme{p-z}_{L^\infty(E)}.
    \end{equation}
    So, since $S>0$, \eqref{Généralisation BPZ - preuve - 1} can equivalently be rewritten as
    \begin{equation}\label{Généralisation BPZ - preuve - 2}
        \mathrm{L}(E)\leq  14d\hspace{0.1cm}\pr{\frac{\norme{p-z}_{L^\infty(E)}}{\norme{p-z}_{L^\infty(\T^d)}}}^{\frac{1}{S}}.
    \end{equation}
    Notice that $(2\pi k)_{k\in \Lambda}$ is a family of elements that are pairwise distinct. This implies that $(e^{2i\pi \ps{k}{\cdot}})_{k\in \Lambda}$ is an orthonormal family of $L^2(\T^d)$ for the usual inner product. So, 
    \begin{equation}\label{Généralisation BPZ - preuve - 3}
        \norme{p-z}_{L^\infty(\T^d)}^2\geq \norme{p-z}_{L^2(\T^d)}^2=\sum_{\substack{k\in \Lambda\\k\neq 0}}\bb{p_k}^2+\bb{p_0-z}^2.
    \end{equation}
    By Cauchy-Schwarz' inequality, we have, 
    \begin{equation}\label{Généralisation BPZ - preuve - 4}
        \sum_{\substack{k\in \Lambda\\k\neq 0}}\bb{p_k}^2+\bb{p_0-z}^2\geq \sum_{\substack{k\in \Lambda\\k\neq 0}}\bb{p_k}^2\geq \pr{\frac{m}{\sqrt{\Xi}}}^2.
    \end{equation}
    Denoting 
    \begin{equation*}
        E_{z,t}:=\acc{\xi\in \T^d\m; \m \bb{p(\xi)-z}^2\leq t},
    \end{equation*}
    we have that 
    \begin{equation}\label{Généralisation BPZ - preuve - 5}
        \norme{p-z}_{L^\infty(E_{z,t})}\leq \sqrt{t}.
    \end{equation}
    The continuity of $p$ provides the closedness of $E_{z,t}$. In particular, $E_{z,t}$ is a Borel subset of $\T^d$. Then, using \eqref{Généralisation BPZ - preuve - 2} with $E=E_{z,t}$, \eqref{Généralisation BPZ - preuve - 3}, \eqref{Généralisation BPZ - preuve - 4} and \eqref{Généralisation BPZ - preuve - 5}, we obtain \eqref{Généralisation BPZ - Hypothèse - 2}.
\end{proof}

\begin{prop}\label{Généralisation BPZ quat}
    Let $p=p(x,\xi)$ be of the form
    \begin{equation}\label{Généralisation BPZ quat - Hypothèse - 0}
        p(x,\xi)=\sum_{k\in \Lambda}p_k(x)e^{2i\pi \ps{k}{\xi}}, \quad(x,\xi)\in [0,1]^d\times \T^d,
    \end{equation}
    where $\Lambda=\prod_{j=1}^d\disc{-n_j}{m_j}$, $n_j,m_j\geq 0$, and $p_k:[0,1]^d\to \C$, $k\in \Lambda$ is measurable. Let us denote 
    \begin{equation*}
        S:=\sum_{j=1}^d (n_j+m_j)>0\mm\m\m\text{and}\mm\m\m \Xi:=\prod_{j=1}^d(n_j+m_j+1).
    \end{equation*}
    We assume that: there exists $m>0$ such that for almost every $x\in [0,1]^d$, 
    \begin{equation}\label{Généralisation BPZ quat - Hypothèse - 2}
        \sum_{\substack{k\in \Lambda\\k\neq 0}}^{N_+}\bb{p_k(x)}\geq m>0.
    \end{equation}
    Then, for all $z\in \C$, $t\geq 0$, 
    \begin{equation}\label{Généralisation BPZ quat - Hypothèse - 3}
        \mathrm{L}\pr{\acc{(x,\xi)\in [0,1]^d\times \T^d\m;\m \bb{p(x,\xi)-z}^2\leq t}}\leq Ct^\kappa
    \end{equation}
    where
    \begin{equation}\label{Généralisation BPZ quat - Hypothèse - 4}
        \kappa=\frac{1}{2S}\mm\text{and}\mm C= 14d\pr{\frac{\sqrt{\Xi}}{m}}^{\frac{1}{S}}.
    \end{equation}
\end{prop}

\begin{proof}
    Let $z\in \C$. For $x\in [0,1]^d$, denote 
    \begin{equation*}
        E_{z,t}^x:=\acc{\xi\in \T^d\m;\m \bb{p(x,\xi)-z}^2\leq t}.
    \end{equation*}
    Let $x\in [0,1]^d$ such that \eqref{Généralisation BPZ quat - Hypothèse - 2} holds and let us denote
    \begin{equation*}
        f(x):=\sum_{\substack{k\in \Lambda\\k\neq 0}}\bb{p_k(x)}.
    \end{equation*}
    Since $f(x)\geq m>0$, from Proposition \ref{Généralisation BPZ - 1}, we get 
    \begin{equation}\label{Généralisation BPZ quat - preuve - 0}
        \begin{split}
            \mathrm{L}(E^x_{z,t})\leq 14d\pr{\frac{\sqrt{\Xi}}{f(x)}}^{\frac{1}{S}}t^{\kappa}\leq Ct^\kappa.
        \end{split}
    \end{equation}
    Functions $p_k$ being measurable for all $k\in \Lambda$, so is $p(\cdot,\xi)$ for all $\xi\in \T^d$. In addition, for all $x\in [0,1]^d$, $p(x,\cdot)$ is analytic, so measurable. This ensures the measurability of $E_{z,t}$ and therefore implies that of $(x,\xi)\mapsto\mathbb{1}_{E_{z,t}}(x,\xi)$, which is also $L^1([0,1]^d\times \T^d)$. Then, Fubini's theorem gives
    \begin{equation}\label{Généralisation BPZ quat - preuve - 1}
        \begin{split}
            \mathrm{L}(E_{z,t}) &= \int_{[0,1]^d\times \T^d}\mathbb{1}_{E_{z,t}}(x,\xi)dxd\xi=\int_{[0,1]^d}\pr{\int_{\T^d}\mathbb{1}_{E_{z,t}}(x,\xi)d\xi}dx.
        \end{split}
    \end{equation}
    But noticing that for all $(x,\xi)\in [0,1]^d\times \T^d$, $\mathbb{1}_{E_{z,t}}(x,\xi)=\mathbb{1}_{E^x_{z,t}}(\xi)$, \eqref{Généralisation BPZ quat - preuve - 1} can be rewritten as
    \begin{equation}\label{Généralisation BPZ quat - preuve - 2}
        \mathrm{L}(E_{z,t}) =\int_{[0,1]^d}\pr{\int_{\T^d}\mathbb{1}_{E^x_{z,t}}(\xi)d\xi}dx=\int_{[0,1]^d}\mathrm{L}(E^x_{z,t})dx.
    \end{equation}
    However, inequality \eqref{Généralisation BPZ quat - Hypothèse - 2} holds for almost every $x\in [0,1]^d$, and so does inequality \eqref{Généralisation BPZ quat - preuve - 0}. Thus, from \eqref{Généralisation BPZ quat - preuve - 2}, we obtain,
    \begin{equation*}
        \mathrm{L}(E_{z,t}) \leq \int_{[0,1]^d}Ct^\kappa dx=Ct^\kappa
    \end{equation*}
    which provides \eqref{Généralisation BPZ quat - Hypothèse - 3}.
\end{proof}

\nopagebreak
\printbibliography

\end{document}